\documentclass[10pt]{amsart}
\usepackage{amsmath}
\usepackage{amsxtra}
\usepackage{amscd}
\usepackage{amsthm}
\usepackage{amsfonts}
\usepackage{amssymb}
\usepackage{eucal}
\usepackage{scalerel}
\usepackage{verbatim}
\usepackage[matrix,arrow,curve]{xy}

\usepackage{stmaryrd}

\usepackage{upgreek}

\usepackage[T1]{fontenc}

\usepackage[
pdftex,
bookmarks=false,
colorlinks=true,
debug=true,
pdfnewwindow=true]{hyperref}

\usepackage{lineno}
\makeatletter
\let\@wraptoccontribs\wraptoccontribs
\makeatother

\newtheorem{cor}[equation]{Corollary}
\newtheorem{lem}[equation]{Lemma}
\newtheorem{prop}[equation]{Proposition}

\newtheorem{defn}[equation]{Definition}
\newtheorem{rem}[equation]{Remark}
\numberwithin{equation}{section}

\theoremstyle{definition}

\theoremstyle{remark}

\newcommand{\nc}{\newcommand}
\nc{\renc}{\renewcommand} \nc{\ssec}{\subsection}
\nc{\sssec}{\subsubsection}

\nc{\on}{\operatorname} \nc{\wh}{\widehat}
\nc\ol{\overline} \nc\ul{\underline} \nc\wt{\widetilde}

\usepackage[dvipsnames]{xcolor}

\nc{\BA}{{\mathbb{A}}} \nc{\BC}{{\mathbb{C}}}
\nc{\BD}{{\mathbb{D}}} \nc{\BQ}{{\mathbb{Q}}}
\nc{\BM}{{\mathbb{M}}} \nc{\BN}{{\mathbb{N}}}
\nc{\BP}{{\mathbb{P}}} \nc{\BR}{{\mathbb{R}}}
\nc{\BZ}{{\mathbb{Z}}} \nc{\BS}{{\mathbb{S}}} \nc{\BW}{{\mathbb{W}}}

\nc{\CA}{{\mathcal{A}}} \nc{\CB}{{\mathcal{B}}} \nc{\CalD}{{\mathcal{D}}}
\nc{\CalC}{{\mathcal{C}}} \nc{\CE}{{\mathcal{E}}} \nc{\CF}{{\mathcal{F}}}
\nc{\CG}{{\mathcal{G}}} \nc{\CH}{{\mathcal{H}}}
\nc{\CI}{{\mathcal{I}}} \nc{\CK}{{\mathcal{K}}} \nc{\CL}{{\mathcal{L}}}
\nc{\CM}{{\mathcal{M}}} \nc{\CN}{{\mathcal{N}}}
\nc{\CO}{{\mathcal{O}}} \nc{\CP}{{\mathcal{P}}}
\nc{\CQ}{{\mathcal{Q}}} \nc{\CR}{{\mathcal{R}}}
\nc{\CS}{{\mathcal{S}}} \nc{\CT}{{\mathcal{T}}}
\nc{\CU}{{\mathcal{U}}} \nc{\CV}{{\mathcal{V}}}  \nc{\CY}{{\mathcal Y}}
\nc{\CW}{{\mathcal{W}}} \nc{\CZ}{{\mathcal{Z}}}

\nc{\cM}{{\check{\mathcal M}}{}} \nc{\csM}{{\check{\mathcal A}}{}}
\nc{\oM}{{\overset{\circ}{\mathcal M}}{}}
\nc{\obM}{{\overset{\circ}{\mathbf M}}{}}
\nc{\oCA}{{\overset{\circ}{\mathcal A}}{}}
\nc{\obA}{{\overset{\circ}{\mathbf A}}{}}
\nc{\ooM}{{\overset{\circ}{M}}{}}
\nc{\osM}{{\overset{\circ}{\mathsf M}}{}}
\nc{\vM}{{\overset{\bullet}{\mathcal M}}{}}
\nc{\nM}{{\underset{\bullet}{\mathcal M}}{}}
\nc{\oD}{{\overset{\circ}{\mathcal D}}{}}
\nc{\obD}{{\overset{\circ}{\mathbf D}}{}}
\nc{\oA}{{\overset{\circ}{\mathbb A}}{}}
\nc{\op}{{\overset{\bullet}{\mathbf p}}{}}
\nc{\cp}{{\overset{\circ}{\mathbf p}}{}}
\nc{\oU}{{\overset{\bullet}{\mathcal U}}{}}
\nc{\ff}{{\mathfrak{f}}} \nc{\fv}{{\mathfrak{v}}}
\nc{\fa}{{\mathfrak{a}}} \nc{\fb}{{\mathfrak{b}}}
\nc{\fd}{{\mathfrak{d}}} \nc{\fe}{{\mathfrak{e}}}
\nc{\fg}{{\mathfrak{g}}} \nc{\fgl}{{\mathfrak{gl}}}
\nc{\fh}{{\mathfrak{h}}} \nc{\fri}{{\mathfrak{i}}}
\nc{\fj}{{\mathfrak{j}}} \nc{\fk}{{\mathfrak{k}}} \nc{\fl}{{\mathfrak{l}}}
\nc{\fm}{{\mathfrak{m}}} \nc{\fn}{{\mathfrak{n}}}
\nc{\ft}{{\mathfrak{t}}} \nc{\fu}{{\mathfrak{u}}}
\nc{\fw}{{\mathfrak{w}}} \nc{\fz}{{\mathfrak{z}}}
\nc{\fp}{{\mathfrak{p}}} \nc{\frr}{{\mathfrak{r}}}
\nc{\fs}{{\mathfrak{s}}} \nc{\fsl}{{\mathfrak{sl}}}
\nc{\hsl}{{\widehat{\mathfrak{sl}}}}
\nc{\hgl}{{\widehat{\mathfrak{gl}}}}
\nc{\hg}{{\widehat{\mathfrak{g}}}}
\nc{\chg}{{\widehat{\mathfrak{g}}}{}^\vee}
\nc{\hn}{{\widehat{\mathfrak{n}}}}
\nc{\chn}{{\widehat{\mathfrak{n}}}{}^\vee}

\nc{\fA}{{\mathfrak{A}}} \nc{\fB}{{\mathfrak{B}}} \nc{\fC}{{\mathfrak{C}}}
\nc{\fD}{{\mathfrak{D}}} \nc{\fE}{{\mathfrak{E}}}
\nc{\fF}{{\mathfrak{F}}} \nc{\fG}{{\mathfrak{G}}} \nc{\fH}{{\mathfrak{H}}}
\nc{\fI}{{\mathfrak{I}}} \nc{\fJ}{{\mathfrak{J}}}
\nc{\fK}{{\mathfrak{K}}} \nc{\fL}{{\mathfrak{L}}}
\nc{\fM}{{\mathfrak{M}}} \nc{\fN}{{\mathfrak{N}}}
\nc{\frP}{{\mathfrak{P}}} \nc{\fQ}{{\mathfrak{Q}}}
\nc{\fT}{{\mathfrak{T}}} \nc{\fU}{{\mathfrak{U}}}
\nc{\fV}{{\mathfrak{V}}} \nc{\fW}{{\mathfrak{W}}}
\nc{\fX}{{\mathfrak{X}}} \nc{\fY}{{\mathfrak{Y}}}
\nc{\fZ}{{\mathfrak{Z}}}

\nc{\ba}{{\mathbf{a}}}
\nc{\bb}{{\mathbf{b}}} \nc{\bc}{{\mathbf{c}}}
\nc{\be}{{\mathbf{e}}} \nc{\bj}{{\mathbf{j}}}
\nc{\bn}{{\mathbf{n}}} \nc{\bp}{{\mathbf{p}}}
\nc{\bq}{{\mathbf{q}}} \nc{\br}{{\mathbf{r}}} \nc{\bt}{{\mathbf{t}}}
\nc{\bfu}{{\mathbf{u}}} \nc{\bv}{{\mathbf{v}}}
\nc{\bx}{{\mathbf{x}}} \nc{\by}{{\mathbf{y}}}
\nc{\bw}{{\mathbf{w}}} \nc{\bA}{{\mathbf{A}}}
\nc{\bB}{{\mathbf{B}}} \nc{\bC}{{\mathbf{C}}}
\nc{\bD}{{\mathbf{D}}} \nc{\bF}{{\mathbf{F}}}
\nc{\bH}{{\mathbf{H}}} \nc{\bJ}{{\mathbf{J}}} \nc{\bK}{{\mathbf{K}}}
\nc{\bM}{{\mathbf{M}}} \nc{\bN}{{\mathbf{N}}}
\nc{\bO}{{\mathbf{O}}} \nc{\bS}{{\mathbf{S}}} \nc{\bT}{{\mathbf{T}}}
\nc{\bV}{{\mathbf{V}}} \nc{\bW}{{\mathbf{W}}}
\nc{\bX}{{\mathbf{X}}}
\nc{\bY}{{\mathbf{Y}}} \nc{\bP}{{\mathbf{P}}}
\nc{\bZ}{{\mathbf{Z}}} \nc{\bh}{{\mathbf{h}}}

\nc{\sA}{{\mathsf{A}}} \nc{\sB}{{\mathsf{B}}}
\nc{\sC}{{\mathsf{C}}} \nc{\sD}{{\mathsf{D}}}
\nc{\sE}{{\mathsf{E}}} \nc{\sF}{{\mathsf{F}}} \nc{\sG}{{\mathsf{G}}}
\nc{\sI}{{\mathsf{I}}} \nc{\sK}{{\mathsf{K}}} \nc{\sL}{{\mathsf{L}}}
\nc{\sfm}{{\mathsf{m}}} \nc{\sM}{{\mathsf{M}}} \nc{\sO}{{\mathsf{O}}}
\nc{\sQ}{{\mathsf{Q}}} \nc{\sP}{{\mathsf{P}}}
\nc{\sT}{{\mathsf{T}}} \nc{\sZ}{{\mathsf{Z}}}
\nc{\sV}{{\mathsf{V}}} \nc{\sW}{{\mathsf{W}}}
\nc{\sfp}{{\mathsf{p}}} \nc{\sr}{{\mathsf{r}}}
\nc{\st}{{\mathsf{t}}} \nc{\sfb}{{\mathsf{b}}}
\nc{\sfc}{{\mathsf{c}}} \nc{\sd}{{\mathsf{d}}}
\nc{\sz}{{\mathsf{z}}}

\nc{\tA}{{\widetilde{\mathbf{A}}}}
\nc{\tB}{{\widetilde{\mathcal{B}}}}
\nc{\tg}{{\widetilde{\mathfrak{g}}}} \nc{\tG}{{\widetilde{G}}}
\nc{\TM}{{\widetilde{\mathbb{M}}}{}}
\nc{\tO}{{\widetilde{\mathsf{O}}}{}}
\nc{\tU}{{\widetilde{\mathfrak{U}}}{}} \nc{\TZ}{{\tilde{Z}}}
\nc{\tx}{{\tilde{x}}} \nc{\tbv}{{\tilde{\bv}}}
\nc{\tfP}{{\widetilde{\mathfrak{P}}}{}} \nc{\tz}{{\tilde{\zeta}}}
\nc{\tmu}{{\tilde{\mu}}}

\nc{\urho}{\underline{\rho}} \nc{\uB}{\underline{B}}
\nc{\uC}{{\underline{\mathbb{C}}}} \nc{\ui}{\underline{i}}
\nc{\uj}{\underline{j}} \nc{\ofP}{{\overline{\mathfrak{P}}}}
\nc{\oB}{{\overline{\mathcal{B}}}}
\nc{\og}{{\overline{\mathfrak{g}}}} \nc{\oI}{{\overline{I}}}

\nc{\eps}{\varepsilon} \nc{\hrho}{{\hat{\rho}}}
\nc{\blambda}{{\bar\lambda}} \nc{\bmu}{{\bar\mu}} \nc{\bnu}{{\bar\nu}}

\nc{\one}{{\mathbf{1}}} \nc{\two}{{\mathbf{t}}}

\nc{\Sym}{{\mathop{\operatorname{\rm Sym}}}}
\nc{\Tot}{{\mathop{\operatorname{\rm Tot}}}}
\nc{\Spec}{{\mathop{\operatorname{\rm Spec}}}}
\nc{\Ker}{{\mathop{\operatorname{\rm Ker}}}}
\nc{\Hilb}{{\mathop{\operatorname{\rm Hilb}}}}
\nc{\End}{{\mathop{\operatorname{\rm End}}}}
\nc{\Ext}{{\mathop{\operatorname{\rm Ext}}}}
\nc{\Hom}{{\mathop{\operatorname{\rm Hom}}}}
\nc{\CHom}{{\mathop{\operatorname{{\mathcal{H}}\it om}}}}
\nc{\GL}{{\mathop{\operatorname{\rm GL}}}}
\nc{\gr}{{\mathop{\operatorname{\rm gr}}}}
\nc{\Id}{{\mathop{\operatorname{\rm Id}}}}
\nc{\defi}{{\mathop{\operatorname{\rm def}}}}
\nc{\length}{{\mathop{\operatorname{\rm length}}}}
\nc{\supp}{{\mathop{\operatorname{\rm supp}}}}
\nc{\HC}{{\mathcal H}{\mathcal C}}

\nc{\Cliff}{{\mathsf{Cliff}}}
\nc{\Fl}{{\mathcal{F}\ell}} \nc{\Fib}{{\mathsf{Fib}}}
\nc{\Coh}{{\mathsf{Coh}}} \nc{\FCoh}{{\mathsf{FCoh}}}

\nc{\reg}{{\text{\rm reg}}}
\nc{\pos}{{\text{\rm pos}}}
\nc{\gvee}{{\mathfrak g}^{\!\scriptscriptstyle\vee}}
\nc{\tvee}{{\mathfrak t}^{\!\scriptscriptstyle\vee}}
\nc{\nvee}{{\mathfrak n}^{\!\scriptscriptstyle\vee}}
\nc{\bvee}{{\mathfrak b}^{\!\scriptscriptstyle\vee}}

\nc{\cplus}{{\mathbf{C}_+}} \nc{\cminus}{{\mathbf{C}_-}}
\nc{\cthree}{{\mathbf{C}_*}} \nc{\Qbar}{{\bar{Q}}}

\newcommand{\oZ}{\vphantom{j^{X^2}}\smash{\overset{\circ}{\vphantom{\rule{0pt}{0.55em}}\smash{Z}}}}
\newcommand{\ofZ}{\vphantom{j^{X^2}}\smash{\overset{\circ}{\vphantom{\rule{0pt}{0.55em}}\smash{\mathfrak Z}}}}
\newcommand\iso{\,\vphantom{j^{X^2}}\smash{\overset{\sim}{\vphantom{\rule{0pt}{0.20em}}\smash{\longrightarrow}}}\,}
\nc{\Gtimes}{\vphantom{j^{X^2}}\smash{\overset{\mathsf G}{\vphantom{\rule{0pt}{0.30em}}\smash{\times}}}}

\nc{\bOmega}{{\overline{\Omega}}}

\nc{\seq}[1]{\stackrel{#1}{\sim}}

\nc{\Sch}{{\operatorname{Sch}}}
\nc{\aff}{{\operatorname{aff}}}
\nc{\fin}{{\operatorname{fin}}}
\nc{\Gr}{{\operatorname{Gr}}}
\nc{\GR}{{\mathbf{Gr}}}
\nc{\Perv}{{\operatorname{Perv}}}
\nc{\Rep}{{\operatorname{Rep}}}
\nc{\IC}{{\operatorname{IC}}}
\nc{\Bun}{{\operatorname{Bun}}}
\nc{\Proj}{{\operatorname{Proj}}}
\nc{\pt}{{\operatorname{pt}}}
\nc{\bfmu}{{\boldsymbol{\mu}}}
\nc{\bfomega}{{\boldsymbol{\omega}}}
\nc{\la}{\lambda}
\nc{\La}{\Lambda}
\nc{\ra}{\rightarrow}
\nc{\BG}{\mathbb{G}}
\nc{\BX}{\mathbb{X}}
\nc{\al}{\alpha}
\nc{\Maps}{{\bf{Maps}}}
\nc{\lvee}{\!\scriptscriptstyle\vee}
\nc{\longonto}{\ontoover{\ }}
\nc{\longinto}{\lhook\joinrel\longrightarrow}
\nc{\longintointo}{\lhook\joinrel\lhook\joinrel\lhook\joinrel\longrightarrow}

\begin{document}

\title
%[Dva Idiota]
[Drinfeld-Gaitsgory-Vinberg interpolation Grassmannian]
{Drinfeld-Gaitsgory-Vinberg interpolation Grassmannian and geometric Satake equivalence}

\author{Michael Finkelberg}
\address{
National Research University Higher School of Economics, Russian Federation\newline
Department of Mathematics, 6 Usacheva st., Moscow 119048;\newline
Skolkovo Institute of Science and Technology;\newline
Institute for Information Transmission Problems of RAS}
\email{fnklberg@gmail.com}
\author{Vasily Krylov}
\address{National Research University Higher School of Economics, Russian Federation\newline
Department of Mathematics, 6 Usacheva st., Moscow 119048;\newline
Skolkovo Institute of Science and Technology
}
\email{kr-vas57@yandex.ru}
\author{Ivan Mirkovi\'c}
\address{University of Massachusetts, Department of Mathematics and Statistics, 710 N.~Pleasant
  st., Amherst MA 01003}
\email{mirkovic@math.umass.edu}
\contrib[with Appendix by]{Dennis Gaitsgory}

\dedicatory{To Ernest Borisovich Vinberg with admiration}

%\thanks{{\bf Mathematics Subject Classification (2000).}
%19E08, (22E65, 37K10).}

%\thanks{{\bf Key words.} $q$-difference Toda lattice, Equivariant
%$K$-theory, Laumon compactification.}

%\thanks{The work of L.R. was partially supported
%by  RFBR grants 07-01-92214-CNRSL-a and 05-01-02805-CNRSL-a.
%L.R. gratefully acknowledges the support from Deligne 2004 Balzan
%prize in mathematics.}

\begin{abstract}
  Let $G$ be a reductive complex algebraic group. We fix a pair of opposite Borel subgroups
  and consider the corresponding semiinfinite orbits in the affine Grassmannian $\Gr_G$.
  We prove Simon Schieder's conjecture identifying his bialgebra formed by the top compactly
  supported cohomology of the intersections of opposite semiinfinite orbits with $U(\nvee)$
  (the universal enveloping algebra of the positive nilpotent subalgebra of the Langlands dual
  Lie algebra $\gvee$). To this end we construct an action of Schieder bialgebra on the
  geometric Satake fiber functor. We propose a conjectural construction of Schieder bialgebra
  for an arbitrary symmetric Kac-Moody Lie algebra in terms of Coulomb branch of the corresponding
  quiver gauge theory.
\end{abstract}
\maketitle

\section{Introduction}
\subsection{}
\label{Schubert}
Let $\Lambda=\bigoplus_n\Lambda^n$ be the ring of symmetric functions equipped with the
base of Schur functions $s_\lambda$. It also carries a natural coproduct.
The classical Schubert calculus is the based isomorphism of the bialgebra $\Lambda$ with
$H^\bullet(\Gr,\BZ)$ (doubling the degrees) taking $s_\lambda$ to the fundamental class of
the corresponding Schubert variety $\sigma_\lambda$. Here $\Gr$ is the infinite
Grassmannian $\Gr=\lim\limits_\to\Gr(k,m)\simeq BU(\infty)$, and the coproduct on
$H^\bullet(\Gr,\BZ)$ comes from the $H$-space structure on the classifying space
$BU(\infty)$.

Here is a more algebraic geometric construction of the coproduct on $H^\bullet(\Gr,\BZ)$.
We have $H^{2n}(\Gr,\BZ)=H^{2n}(\ol{\Sch}_n,\BZ)=H^{2n}_c(\ol{\Sch}_n,\BZ)=H^{2n}_c(\Sch_n,\BZ)$
where $\Sch_n\subset\Gr$ (resp.\ $\ol{\Sch}_n\subset\Gr$) stands for the union of all
$n$-dimensional (resp. $\leq n$-dimensional) Schubert cells
(with respect to a fixed flag).

Recall the Calogero-Moser phase space $\CalC_n$: the space of pairs of $n\times n$-matrices
$(X,Y)$ such that $[X,Y]+\on{Id}$ has rank 1, modulo the simultaneous conjugation
of $X,Y$. The integrable system $\pi_n\colon\CalC_n\to\BA^{(n)}$ takes $(X,Y)$ to the
spectrum of $X$. Wilson~\cite{w} has discovered the following two key properties of
the Calogero-Moser integrable system:

\textup{(a)} for $n_1+n_2=n$, a {\em factorization} isomorphism
\begin{equation*}
\CalC_n\times_{\BA^{(n)}}(\BA^{(n_1)}\times\BA^{(n_2)})_{\on{disj}}\iso(\CalC_{n_1}\times\CalC_{n_2})
\times_{(\BA^{(n_1)}\times\BA^{(n_2)})}(\BA^{(n_1)}\times\BA^{(n_2)})_{\on{disj}}.
\end{equation*}

\textup{(b)} For $x\in\BA^1$, an isomorphism $\pi_n^{-1}(n\cdot x)\iso\Sch_n$.\\
Now the desired coproduct
\begin{equation*}
  \Delta=\bigoplus_{n_1+n_2=n}\Delta_{n_1,n_2}\colon
H^{2n}_c(\Sch_n,\BZ)\to\bigoplus_{n_1+n_2=n}H^{2n_1}_c(\Sch_{n_1},\BZ)\otimes
H^{2n_2}_c(\Sch_{n_2},\BZ)
\end{equation*}
is nothing but the cospecialization\footnote{terminology of~\cite[6.2.7]{s3}.} morphism for the
compactly supported cohomology of the fibers of $\pi_n$ restricted to the subfamily
$\pi_n^{-1}(n_1\cdot x+n_2\cdot y)\subset\CalC_n$ (from the fibers over the diagonal
$x=y$ to the off-diagonal fibers $x\ne y$), cf.~\cite{fg}.

\subsection{}
\label{setup}
Given a reductive complex algebraic group $G$, Schieder~\cite{s3} constructed a bialgebra
$\CA$ playing the role of $\bigoplus_nH^{2n}_c(\Sch_n,\BC)$ for the affine Grassmannian
$\Gr_G$ in place of $\Gr$. In order to explain his construction, we set up the basic
notations for $G$ and $\Gr_G$.

We fix a Borel and a Cartan subgroup $G\supset B\supset T$, and denote by $W$
the Weyl group of $(G,T)$. Let $N$ denote the unipotent radical of
the Borel $B$, and let $N_-$ stand for the unipotent radical of
the opposite Borel $B_-$. Let $\Lambda$ (resp.\ $\Lambda^\vee$)
be the coweight (resp.\ weight) lattice, and let
$\Lambda^+\subset\Lambda$ (resp.\ $\Lambda^{\vee+}\subset\Lambda^\vee$) be the
cone of dominant coweights (resp.\ weights).
%submonoid
%spanned by the fundamental coweights (resp.\ weights) $\omega_i,\ i\in I$
%(resp.\ $\omega^{\!\scriptscriptstyle\vee}_i,\ i\in I$).
Let also $\Lambda^\pos\subset\Lambda$ (resp.\ $\Lambda^{\vee,\pos}\subset\Lambda^\vee$)
be the submonoid spanned by the simple coroots (resp.\ roots)
$\alpha_i,\ i\in I$ (resp.\ $\alpha^{\!\scriptscriptstyle\vee}_i,\ i\in I$).
We denote by $G^\vee\supset T^\vee$ the Langlands dual
group, so that $\Lambda$ (resp.\ $\Lambda^\vee$)
is the weight (resp.\ coweight) lattice of $G^\vee$.

Let $\CO$ denote the formal power series ring $\BC[[z]]$, and let $\CK$ denote
its fraction field $\BC((z))$.
The affine Grassmannian $\Gr_G=G_\CK/G_\CO$ is an ind-projective scheme,
the union $\bigsqcup_{\lambda\in\Lambda^+}\Gr_G^\lambda$ of $G_\CO$-orbits.
The closure of $\Gr_G^\lambda$ is a projective variety
$\ol\Gr{}^\lambda_G=\bigsqcup_{\mu\leq\lambda}\Gr^\mu_G$. The fixed point set
$\Gr^T_G$ is naturally identified with the coweight lattice $\Lambda$;
and $\mu\in\Lambda$ lies in $\Gr_G^\lambda$ iff $\mu\in W\lambda$.

For a coweight
$\nu\in\Lambda=\Gr_G^T$, we denote by $S_\nu\subset\Gr_G$
(resp.\ $T_\nu\subset\Gr_G$) the orbit of $N(\CK)$
(resp.\ of $N_-(\CK)$) through $\nu$. The intersections $S_\nu\cap\ol\Gr{}^\lambda_G$
(resp.\ $T_\nu\cap\ol\Gr{}^\lambda_G$) are the
{\em attractors} (resp.\ {\em repellents}) of $\BC^\times$ acting via
its homomorphism $2\rho$ to the Cartan torus
$T\curvearrowright\ol\Gr{}^\lambda_G\colon
S_\nu\cap\ol\Gr{}^\lambda_G=\{x\in\ol\Gr{}^\lambda_G :
\lim\limits_{c\to0}2\rho(c)\cdot x=\nu\}$ and
$T_\nu\cap\ol\Gr{}^\lambda_G=\{x\in\ol\Gr{}^\lambda_G :
\lim\limits_{c\to\infty}2\rho(c)\cdot x=\nu\}$. Going to the limit
$\Gr_G=\lim\limits_{\lambda\in\Lambda^+}\ol\Gr{}^\lambda_G,\ S_\nu$
(resp.\ $T_\nu$) is the attractor (resp.\ repellent) of $\nu$ in $\Gr_G$.
The closure $\ol{S}_\nu$ is the union $\bigsqcup_{\mu\leq\nu}S_\mu$, while
$\ol{T}_\nu=\bigsqcup_{\mu\geq\nu}T_\mu$.

\begin{defn}
  {\em
  \textup{(a)} For $\theta\in\Lambda^\pos$ we denote by $\Sch_\theta$ (resp.\ $\ol\Sch_\theta$)
  the intersection $S_\theta\cap T_0$ (resp.\ $\ol{S}_\theta\cap\ol{T}_0$).\footnote{Here Sch
    stands for Schieder.} It is equidimensional of dimension
  $\langle\rho^{\!\scriptscriptstyle\vee},\theta\rangle$ (see~\cite[\S6.3]{bfgm}).\footnote{Strictly
    speaking, only the inequality
    $\dim(S_\theta\cap T_0)\leq\langle\rho^{\!\scriptscriptstyle\vee},\theta\rangle$ is verified
    right after the proof of~\cite[Proposition~6.4]{bfgm}. The opposite inequality follows e.g.\
    from the existence of the factorizable family $\pi_\theta\colon\oZ^\theta\to X^\theta$ mentioned
    in the next paragraph.}

  \textup{(b)} We set
  $\CA_\theta:=H_c^{\langle2\rho^{\!\scriptscriptstyle\vee},\theta\rangle}(\Sch_\theta,\BC)=
  H_c^{\langle2\rho^{\!\scriptscriptstyle\vee},\theta\rangle}(\ol\Sch_\theta,\BC)$, and
  $\CA:=\bigoplus_{\theta\in\Lambda^\pos}\CA_\theta$.}
\end{defn}

Given a smooth curve $X$ and $\theta\in\Lambda^\pos$, the open zastava space $\oZ^\theta$
(see e.g.~\cite{bfgm}) is equipped with the projection
$\pi_\theta\colon\oZ^\theta\to X^\theta$ to the
degree $\theta$ configuration space of $X$. It enjoys the factorization property, and
for any $x\in X$, we have a canonical isomorphism
$\pi_\theta^{-1}(\theta\cdot x)\iso\Sch_\theta$. Given $\theta_1,\theta_2\in\Lambda^\pos$ such that
$\theta_1+\theta_2=\theta$, the coproduct
$\Delta_{\theta_1,\theta_2}\colon\CA_\theta\to\CA_{\theta_1}\otimes\CA_{\theta_2}$ is defined
just like in~\ref{Schubert} via the cospecialization morphism for the subfamily
$\pi_\theta^{-1}(\theta_1\cdot x+\theta_2\cdot y)$.

To construct the product
$\sfm\colon \bigoplus_{\theta_1+\theta_2=\theta}\CA_{\theta_1}\otimes
\CA_{\theta_2}\to\CA_\theta$
we need the Drinfeld-Gaitsgory interpolation $\widetilde{\Sch}_\theta\to\BA^1$~\cite[\S2.2]{dg2}
constructed with respect to the $\BC^\times$-action on $\ol{\Sch}_\theta$ arising from
the cocharacter $2\rho$ of $T$. The key property of $\widetilde{\Sch}_\theta\to\BA^1$
is that the fibers over $a\ne0$ are all isomorphic to $\ol{\Sch}_\theta$, while the
zero fiber $(\widetilde{\Sch}_\theta)_0$ is isomorphic to the disjoint union
$\bigsqcup\limits_\lambda\ol{\Sch}{}_\theta^{+,\lambda}\times\ol{\Sch}{}_\theta^{-,\lambda}$.
Here $\lambda$ (a coweight in $\Lambda^\pos$ such that $\lambda\leq\theta$)
runs through the set of
$\BC^\times$-fixed points of $\ol{\Sch}_\theta$, and $\ol{\Sch}{}_\theta^{+,\lambda}$
(resp.\ $\ol{\Sch}{}_\theta^{-,\lambda}$) stands for the corresponding attractor
(resp.\ repellent), equal to $S_\lambda\cap\ol{T}_0$ (resp.\ to $\ol{S}_\theta\cap T_\lambda$).
It is easy to see that
$H^{\langle2\rho^{\!\scriptscriptstyle\vee},\theta\rangle}_c((\widetilde{\Sch}_\theta)_0,\BC)=
\bigoplus_{\theta_1+\theta_2=\theta}
H^{\langle2\rho^{\!\scriptscriptstyle\vee},\theta_1\rangle}_c(\Sch_{\theta_1},\BC)\otimes
H^{\langle2\rho^{\!\scriptscriptstyle\vee},\theta_2\rangle}_c(\Sch_{\theta_2},\BC)$, and the desired
product $\sfm$ is nothing but the cospecialization morphism for the compactly
supported cohomology of the fibers of the Drinfeld-Gaitsgory family.

Schieder conjectured that the bialgebra $\CA$ is isomorphic to the universal enveloping
algebra $U(\nvee)$ of $\on{Lie}(N^\vee)$ where $N^\vee\subset B^\vee\subset G^\vee$ is the
unipotent radical of Borel subgroup of $G^\vee$. The goal of the present work is a proof
of Schieder's conjecture.

\subsection{}
\label{inter Gra}
In order to produce an isomorphism $U(\nvee)\iso\CA$, we construct an action of $\CA$ on
the geometric Satake fiber functor. More precisely, we denote by $r_{\nu,+}$
(resp.\ $r_{\nu,-}$) the locally closed embedding
$S_\nu\hookrightarrow\Gr_G$ (resp.\ $T_\nu\hookrightarrow\Gr_G$).
We also denote by $\iota_{\nu,+}$ (resp.\ $\iota_{\nu,-}$) the closed embedding
of the point $\nu$ into $S_\nu$ (resp.\ into $T_\nu$).

According to~\cite{br,dg2}, there is a canonical isomorphism of functors
$\iota_{\nu,-}^*r_{\nu,-}^!\simeq\iota_{\nu,+}^!r_{\nu,+}^*\colon
D^b_{G_\CO}(\Gr_G)\to D^b(\on{Vect})$.
For a sheaf $\CP\in D^b_{G_\CO}(\Gr_G)$, its hyperbolic stalk at $\nu$ is
defined as
$\Phi_\nu(\CP):=\iota_{\nu,-}^*r_{\nu,-}^!\CP\simeq\iota_{\nu,+}^!r_{\nu,+}^*\CP$.
According to~\cite{mv}, for $\CP\in\on{Perv}_{G_\CO}(\Gr_G)$,
the hyperbolic stalk $\Phi_\nu(\CP)$
is concentrated in degree $\langle2\rho^{\!\scriptscriptstyle\vee},\nu\rangle$,
and there is a canonical direct sum decomposition $H^\bullet(\Gr_G,\CP)=
\bigoplus_{\nu\in\Lambda}\Phi_\nu(\CP)$. Moreover, the abelian category
$\on{Perv}_{G_\CO}(\Gr_G)$ is monoidal with respect to the convolution
operation $\star$, and the functor
$H^\bullet(\Gr_G,-)\colon(\on{Perv}_{G_\CO}(\Gr_G),\star)\to(\on{Vect},\otimes)$
is a fiber functor identifying $(\on{Perv}_{G_\CO}(\Gr_G),\star)$ with the tensor
category $\on{Rep}(G^\vee)$ (geometric Satake equivalence).
%We denote by
%${\mathbb S}\colon (\on{Rep}(G^\vee),\otimes)\iso (\on{Perv}_{G_\CO}(\Gr_G),\star)$
%the inverse tensor equivalence.

We define a morphism of functors $\CA_\theta\otimes\Phi_\nu\to\Phi_{\nu+\theta}$
in~\ref{The action}. To this end (and also in order to check various tensor
compatibilities) we consider the {\em Drinfeld-Gaitsgory-Vinberg interpolation Grassmannian}:
a relative compactification $\on{VinGr}_G^{\on{princ}}$ of the Drinfeld-Gaitsgory
interpolation $\widetilde\Gr_G\to\BA^1$.
We also consider an extended version
$\on{VinGr}_G\to T^+_{\on{ad}}:=\on{Spec}\BC[\Lambda^{\vee,\pos}]$ and its version
$\on{VinGr}_{G,X^n}$ for the Beilinson-Drinfeld Grassmannian.
It was implicit already in Schieder's work, and it was made explicit
by D.~Gaitsgory and D.~Nadler, cf.\ an earlier work~\cite{gn}.
We believe it is a very interesting object
in its own right. For example, let $\omega\in\Lambda^+$ be a minuscule dominant coweight.
Then the Schubert variety $\Gr^\omega_G$ is isomorphic to a parabolic flag variety
$G/P_\omega$, and the corresponding subvariety $\on{VinGr}_G^\omega$ of $\on{VinGr}_G$
is isomorphic to Brion's degeneration of $\Delta_{G/P_\omega}$ in
$\on{Hilb}(G/P_\omega\times G/P_\omega)\times T^+_{\on{ad}}$~\cite[\S3]{bri}.

\begin{rem}
  {\em \textup{(a)} By the geometric Satake equivalence, for
    $\CP\in\on{Perv}_{G_\CO}(\Gr_G)$, the cohomology $H^\bullet(\Gr_G,\CP)$ is
    equipped with an action of $U(\gvee)$. For example, the action of the Cartan
    subalgebra $U(\tvee)\subset U(\gvee)$ comes from the grading
    $H^\bullet(\Gr_G,\CP)=\bigoplus_{\nu\in\Lambda}\Phi_\nu(\CP)$. The action of
    $U(\nvee)$ comes from the geometric action of the Schieder bialgebra $\CA$
    on the geometric Satake fiber functor, and the isomorphism $U(\nvee)\iso\CA$.
    Finally, the action of $U(\nvee_-)$ is conjugate to the action of $U(\nvee)$ with
    respect to the Lefschetz bilinear form on $H^\bullet(\Gr_G,\CP)$.

    \textup{(b)} By construction, the Schieder algebra $\CA$ comes equipped with a
    basis (fundamental classes of irreducible components of $\Sch_\theta$). The corresponding
    integral form is denoted $\CA_\BZ\subset\CA$.
    On the other hand, $U(\nvee)$ is equipped with 
    the {\em semicanonical basis}~\cite{lus}.\footnote{constructed under
      the assumption that $G$ is simply laced.} The corresponding integral form is nothing but
    the Chevalley-Kostant integral form $U(\nvee)_\BZ$. According to~Proposition~\ref{integr},
    the isomorphism $U(\nvee)\iso\CA$ gives rise to an isomorphism of their integral forms:
    $U(\nvee)\supset U(\nvee)_\BZ\iso\CA_\BZ\subset\CA$.
 In the simplest example when $G=SL(2),\ \CA$ is $\BN$-graded, and each graded component $\CA_n$
 is one-dimensional with the basis vector $e_n$; one can check
 $e_ne_m=\binom{n+m}{n}e_{n+m}$. Hence the two bases match under the isomorphism
 $U(\nvee)_\BZ\iso\CA_\BZ$. However, for general $G$ the two bases do {\em not} match, as
 seen in an example for $G$ of type $A_5$ in degree $(2,4,4,4,2)$ 
 in~\cite[Appendix~A]{bkk}. Thus the Higgs branch realization~\cite{lus}
 of $U(\nvee)$ is different from the Coulomb branch realization~\cite{s3} of $U(\nvee)$.

\textup{(c)} According to~\cite[\S4.4]{bag}, the irreducible components of the intersection
$\overline{S}_\theta\cap T_0$ are closely related to the algebraic varieties introduced
by G.~Lusztig in~\cite[\S5]{lu}.}
\end{rem}

\subsection{}
The paper is organized as follows. In a lengthy~\S\ref{two} we review the
works of Schieder and other related materials. To simplify the exposition somewhat we assume
that the derived subgroup $[G,G]\subset G$ is simply connected.
The Schieder algebra $\CA$ is defined in~\S\ref{schieder bialgebra}.
Note that the construction of multiplication in~\S\ref{multiplication} does not use the
Drinfeld-Gaitsgory interpolation in contrast to the construction at the end of~\S\ref{setup}.
However, the two constructions are equivalent as a consequence of~Proposition~\ref{DG completion}.
We define the Drinfeld-Gaitsgory-Vinberg interpolation Grassmannian in~\S\ref{three}.
We define the action of the Schieder
algebra $\CA$ on the geometric Satake fiber functor $\Phi$ and check various compatibilities
in~\S\ref{four}. We deduce the Schieder conjecture $\CA\simeq U(\nvee)$
in~Corollary~\ref{five}. In~\S\ref{ivan} we identify the above action of $\CA$ on $\Phi$
with another action going back to~\cite{ffkm}.
In the last~\S\ref{5five} we explain the changes needed in the case
of arbitrary reductive $G$. In~\S\ref{Coulomb}, for a quiver $Q$ without loop edges, we propose
a conjectural geometric construction of $U(\fg_Q^+)$ (positive subalgebra of the corresponding
symmetric Kac-Moody Lie algebra) in the framework of Coulomb branch of the corresponding
quiver gauge theory. Appendix~\ref{app} written by D.~Gaitsgory contains proofs
of~Proposition~\ref{proper over Bun} stating that the Drinfeld-Lafforgue-Vinberg compactification
$\ol{\on{Bun}}_G$ of $\on{Bun}_G$ is proper over $\on{Bun}_G\times\on{Bun}_G$,
and of~Proposition~\ref{DG completion} stating that the Drinfeld-Gaitsgory-Vinberg interpolation
Grassmannian $\on{VinGr}^{\on{princ}}_G$ is a relative compactification of the
Drinfeld-Gaitsgory interpolation $\widetilde\Gr_G$.

This work is a result of generous explanations by D.~Gaitsgory and S.~Schieder.
We are deeply grateful to them. We would also like to thank R.~Bezrukavnikov, A.~Braverman,
A.~Kuznetsov, G.~Lusztig, H.~Nakajima and V.~Ostrik for useful discussions.
Finally, thanks are due to Lin Chen and Ekaterina Bogdanova for spotting a few inaccuracies in the published version,
corrected in the present version of the text.
M.F.\ was partially funded within the framework of the HSE University Basic Research
Program and the Russian Academic Excellence Project `5-100'.
The research of V.K.\ was supported by the grant RSF-DFG 16-41-01013.

\section{Review of Schieder's work}
\label{two}

Till~\S\ref{5five} we assume that the derived subgroup $[G,G]\subset G$ is simply connected.
It implies in particular that $\Lambda^\pos:=\bigoplus_{i\in I}\BN\alpha_i$ is equal to
$\Lambda_{\geq0}:=\{\mu\in\Lambda : \langle\lambda^{\!\scriptscriptstyle\vee},\mu\rangle\geq0\
\forall \lambda^{\!\scriptscriptstyle\vee}\in\Lambda^{\vee+}\}$.

\subsection{General recollections}

\sssec{The affine Grassmannian~\cite[\S4.5]{bd}, \cite[\S1,2,~1.4]{zh}}
Let $X$ be a smooth projective curve over $\BC$. Let us fix a point $x \in X$. The functor of points
\begin{equation*}
  \Gr_G \colon {\bf{Sch}} \ra {\bf{Set}},~S \mapsto \Gr_G(S)
\end{equation*}
can be described as follows. For a scheme $S$ the set $\Gr_G(S)$ consists of the following data:

1) a $G$-bundle $\CF$ on $S \times X$,

2) a trivialization $\sigma$ of $\CF$ on $S \times (X\setminus \{x\})$.

\begin{rem}
{\em{For a space $\CY$, one defines a space
$\on{Maps}(X,\CY) $
parametrizing morphisms from the curve $X$ to $\CY$ as
$\on{Maps}(X,\CY)(S):=\CY(X\times S).$
Note that the ind-scheme $\on{Gr}_{G}$
is isomorphic to the fibre product
\begin{equation*}
\on{Maps}(X,\on{pt}\!/G)
\underset{\on{Maps}(X\setminus \{x\},\on{pt}\!/G)}\times \on{pt}
\end{equation*}
where the morphism
$\on{pt} \ra \on{Maps}(X\!\!\!\setminus\!\!\!\{x\},\on{pt}\!/G)$
is the composition of the isomorphism
$\on{pt} \iso \on{Maps}(X\!\!\setminus\!\!\{x\},\on{pt})$
with the morphism
$\on{Maps}(X\!\!\setminus\!\!\{x\},\on{pt}) \ra \on{Maps}(X\!\!\setminus\!\!\{x\},\on{pt}\!/G)$
induced by the morphism $\on{pt}\ra \on{pt}\!/G$.
}}
\end{rem}

\sssec{The Beilinson-Drinfeld Grassmannian~\cite[\S5.3.10,~5.3.11]{bd},~\cite[\S3.1]{zh},~\cite[\S5]{mv}}
\label{BD Grass}
For $n \in \BN$, $\Gr_{G,X^n}$ is the moduli space of the following data: it associates to a scheme $S$

1) a collection of $S$-points $\ul{x}=(x_1,\dots, x_n) \in X^n(S)$ of the curve $X$,

2) a $G$-bundle $\CF$ on $S\times X$,

3) a trivialization $\sigma$ of $\CF$ on
$(S \times X) \setminus \{\varGamma_{x_1}\cup\dots\cup\varGamma_{x_n}\}$,
where $\varGamma_{x_k} \subset S \times X$ is the graph of $x_k$.

We have a projection $\pi_{n}\colon \Gr_{G,X^n} \ra X^n$
that forgets the data of $\CF$ and $\sigma$.

Let us denote by $\Delta_X \hookrightarrow X^n$ the diagonal embedding. Take a point $x \in X$. Note that the fiber of the morphism $\pi_{n}$ over the point $(x,\dots,x) \in \Delta_X$ is isomorphic to $\Gr_G$. We will denote this fiber by the same symbol.

\begin{rem}\label{second def of GrBD}{\em{Let us fix $n \in \BN$. Let us denote by
$\on{Bun}_{G}(\mathfrak{U}_n)$
the following stack over $X^{n}$:
it associates to a scheme $S$

1) a collection of points $\ul{x}=(x_1,\dots,x_n) \in X^{n}(S)$
of the curve X,

2) a $G$-bundle $\CF$ on $(S\times X) \setminus \{\varGamma_{x_{1}} \cup \dots \cup \varGamma_{x_{n}}\}$.

We have a restriction morphism
$X^n \times \on{Maps}(X,\on{pt}\!/G) \ra \on{Bun}_{G}(\mathfrak{U}_n)$.
We also have a morphism $X^{n} \ra \on{Bun}_{G}(\mathfrak{U}_n)$
that sends a collection of $S$-points $\ul{x}\in X^{n}(S)$
to the data $(\ul{x},\CF_{(S\times X) \setminus \{\varGamma_{x_{1}} \cup \dots \cup \varGamma_{x_{n}}\}}^{\on{triv}}) \in \on{Bun}_G(\mathfrak{U}_n)(S)$.
The Beilinson-Drinfeld Grassmannian $\on{Gr}_{G,X^n}$ is nothing but
\begin{equation*}
(X^n \times \on{Maps}(X,\on{pt}\!/G))
\underset{\on{Bun}_G(\mathfrak{U}_n)}\times
X^{n}.
\end{equation*}

}
}
\end{rem}

\sssec{} Let $G_{X,\CO}$ be the group-scheme (over $X$) that represents the following functor. It associates to a scheme $S$

1) an $S$-point $x\in X(S)$ of the curve $X$,

2) a point $\sigma \in G(\hat{\varGamma}_{x})$, where $\hat{\varGamma}_x$
is the completion of $\varGamma_x$ in $S\times X$.

We have a projection $G_{X,\CO} \ra X$ that
forgets the data of $\sigma$.
Recall the projection
$\pi_1\colon \Gr_{G,X} \ra X$ of \S\ref{BD Grass}.
Thus schemes $\Gr_{G,X}, G_{X,\CO}$ are endowed with structures
of schemes over the curve $X$. We have an action
$G_{X,\CO} \curvearrowright \Gr_{G,X}$ by changing the trivialization.
Hence we can define a category
$\on{Perv}_{G_{X,\CO}}(\Gr_{G,X})$
as the category of $G_{X,\CO}$-equivariant perverse sheaves on $\Gr_{G,X}$.
Let us fix a point $x \in X$. Consider the closed embedding
$\iota_x\colon \Gr_G \hookrightarrow \Gr_{G,X}$.
It follows from \cite[Remark~5.1]{mv}, \cite[\S5.4]{zh} that
there exists a functor $\mathfrak{p}^0\colon\on{Perv}_{G_{\CO}}(\Gr_{G}) \ra \on{Perv}_{G_{X,\CO}}(\Gr_{G,X})$ such that the composition
$\iota_x^*[-1] \circ \mathfrak{p}^0$ is isomorphic to $\on{Id}$.
\begin{comment}
  This is junk.
\end{comment}

\sssec{Tensor structure on $\on{Perv}_{G_\CO}(\Gr_G)$ via Beilinson-Drinfeld Grassmanian} \label{tensor structure} Let $\CP_1, \CP_2 \in \on{Perv}_{G_{\CO}}(\Gr_G)$.
Let us describe the convolution $\CP_1 \star \CP_2$.
Recall the projection
$
\pi_2\colon\Gr_{G,X^2} \ra X^2
$
of \S\ref{BD Grass}. Let $\Delta_{X} \hookrightarrow X^2$ be
the closed embedding of the diagonal. Let
$U \hookrightarrow X^2$ be
the open embedding of the complement to the diagonal.
It follows from~\cite[\S5.3.12]{bd},~\cite[Proposition~3.1.13]{zh} that
the restriction of the family
$\pi_2\colon\Gr_{G,X^2} \ra X^2$
to the open subvariety $U \hookrightarrow X^2$
is isomorphic to $(\Gr_{G,X}\times \Gr_{G,X})|_U$
and the restriction of the family
$\pi_2\colon\Gr_{G,X^2} \ra X^2$
to the closed subvariety $\Delta_{X} \hookrightarrow X^2$
is isomorphic to $\Gr_{G,X}$.

Let us denote by $\mathfrak{j}$ the open embedding
\begin{equation*}
(\Gr_{G,X}\times \Gr_{G,X})|_U\simeq
\pi_2^{-1}(U) \hookrightarrow \Gr_{G,X^2}.
\end{equation*}
Set \begin{equation*}\CP_1\underset{X}\circ \CP_2:=
(\mathfrak{p}^0\CP_1\boxtimes \mathfrak{p}^0\CP_2)|_U,~
\CP_1 \underset{X}\star \CP_2:=\mathfrak{j}_{!*}(\CP_1\underset{X}\circ \CP_2).
\end{equation*}
Then according to~\cite[\S5]{mv},
we have $\mathfrak{p}^0(\CP_1 \star \CP_2) \simeq \CP_1\underset{X}\star \CP_2$, $\CP_1 \star \CP_2 \simeq \iota^{*}_x[-1](\CP_1\underset{X}\star \CP_2)$.

\sssec{Rational morphisms}

Let $\CF$ be a coherent sheaf on $S\times X$.
Let us fix two numbers $n,m \in \BN$. Let us also fix
a collection of $S$-points $\ul{x}=(x_1,\dots,x_n) \in X^n(S)$
of the curve $X$. Let us denote by
$\on{QM}^{m}_{\ul{x}}(\CO_{S\times X},\CF)$
the set of morphisms
$\CO_{S\times X}(-m\cdot(\varGamma_{x_{1}}\cup\dots\cup \varGamma_{x_{n}})) \ra \CF$.
For $m_1, m_2 \in \BN$, $m_1 \leq m_2$, we have the natural embeddings
$\on{QM}^{m_1}_{\ul{x}}(\CO_{S\times X},\CF) \hookrightarrow
\on{QM}^{m_2}_{\ul{x}}(\CO_{S\times X},\CF)$
given by the composition with the morphism
\begin{equation*}
\CO_{S\times X}(-m_2 \cdot(\varGamma_{x_{1}}\cup\dots\cup \varGamma_{x_{n}}))
\hookrightarrow
\CO_{S\times X}(-m_1 \cdot(\varGamma_{x_{1}}\cup\dots\cup \varGamma_{x_{n}})).
\end{equation*}
Let us denote by $\on{RM}_{\ul{x}}(\CO_{S\times X},\CF)$
the inductive limit
$\lim\limits_{m\in \BN}\on{QM}^m_{\ul{x}}(\CO_{S\times X},\CF)$.
The elements of the set $\on{RM}_{\ul{x}}(\CO_{S\times X},\CF)$
will be called rational morphisms from $\CO_{S\times X}$ to $\CF$
regular on $U:=(S\times X)\setminus \{\varGamma_{x_{1}}\cup\dots\cup \varGamma_{x_{n}}\}$.

The set of rational morphisms from $\CF$ to $\CO_{S\times X}$ regular on $U$
is defined analogously as the limit of the sets of morphisms
$\CF \ra \CO_{S\times X}(m \cdot(\varGamma_{x_{1}}\cup\dots\cup \varGamma_{x_{n}}))$.

\sssec{The Tannakian approach} \label{Tannakian def of BD}
For an algebraic group $H$, a right $H$-torsor $\CE$ over a scheme $X$, and $V \in \on{Rep}(H)$ we denote by $\CV_\CE$ the
associated vector bundle over $X\colon \CV_\CE:=\CE\stackrel{H}{\times}V$. For every $\la^{\lvee} \in \La^{\vee+}$ let us fix a highest weight vector $v_{\la^{\lvee}} \in V^{\la^{\lvee}}$ with respect to the Borel subgroup $B \subset G$. For $\la^{\lvee}, \mu^{\lvee} \in \La^{\vee+}$ we get the $G$-morphism $\on{pr}_{\la^{\lvee},\mu^{\lvee}}\colon V^{\la^{\lvee}}\otimes V^{\mu^{\lvee}} \twoheadrightarrow  V^{\la^{\lvee}+\mu^{\lvee}}$ that is uniquely determined by the following property: $\on{pr}_{\la^{\lvee},\mu^{\lvee}}(v_{\la^{\lvee}}\otimes v_{\mu^{\lvee}})=v_{\la^{\lvee}+\mu^{\lvee}}$.
Let $\Gr'_{G,X^n}$ be the following moduli space.
It associates to a scheme $S$

1) a collection of $S$-points $\ul{x}=(x_1,\dots,x_n) \in X^{n}(S)$ of the curve $X$,

2) a $G$-bundle $\CF$ on $S\times X$,

3) rational $N_-$ and $N$-structures on $\CF$, i.e.\ (cf.~\cite[Theorem~1.1.2]{bg})
for every $\la^{\!\scriptscriptstyle\vee} \in \Lambda^{\vee+}$, rational morphisms \begin{equation*} \CO_{S\times X} \xrightarrow{\eta_{\la^{\!\scriptscriptstyle\vee}}}
\CV^{\la^{\!\scriptscriptstyle\vee}}_{\CF} \xrightarrow{\zeta_{\la^{\!\scriptscriptstyle\vee}}} \CO_{S\times X}\end{equation*}
regular on
$U:=(S\times X)\setminus\{\varGamma_{x_1}\cup\dots\cup\varGamma_{x_n}\},$
satisfying the following conditions.

a) For every $\la^{\!\scriptscriptstyle\vee} \in \Lambda^{\vee+}$ the composition \begin{equation*}(\zeta_{\la^{\!\scriptscriptstyle\vee}} \circ \eta_{\la^{\!\scriptscriptstyle\vee}})|_{(S\times X)\setminus\{\varGamma_{x_1}\cup\dots\cup\varGamma_{x_n}\}}\end{equation*}
is the identity morphism.

b) For every $\lambda^{\!\scriptscriptstyle\vee}, \mu^{\!\scriptscriptstyle\vee} \in \Lambda^{\vee+}$ let $\on{pr}_{\la^{\!\scriptscriptstyle\vee},\mu^{\!\scriptscriptstyle\vee}}\colon  V^{\la^{\!\scriptscriptstyle\vee}} \otimes V^{\mu^{\!\scriptscriptstyle\vee}} \twoheadrightarrow  V^{\la^{\!\scriptscriptstyle\vee}+\mu^{\!\scriptscriptstyle\vee}}$ be the projection morphism. We have the corresponding morphisms \begin{equation*}\on{pr}^{\CF}_{\la^{\!\scriptscriptstyle\vee},\mu^{\!\scriptscriptstyle\vee}}\colon  \CV^{\la^{\!\scriptscriptstyle\vee}}_{\CF} \otimes \CV^{\mu^{\!\scriptscriptstyle\vee}}_{\CF} \ra \CV^{\la^{\!\scriptscriptstyle\vee}+\mu^{\!\scriptscriptstyle\vee}}_{\CF}.\end{equation*}

Then the following diagrams are commutative:

\begin{equation*}\begin{CD}
\CO_{U}\otimes \CO_{U}@>\on{Id}\otimes \on{Id}>> \CO_{U}\\
@VV\eta_{\la^{\!\scriptscriptstyle\vee}} \otimes \eta_{\mu^{\!\scriptscriptstyle\vee}}V @VV\eta_{\la^{\!\scriptscriptstyle\vee}+\mu^{\!\scriptscriptstyle\vee}}V\\
(\CV^{\la^{\!\scriptscriptstyle\vee}}_{\CF} \otimes \CV^{\mu^{\!\scriptscriptstyle\vee}}_{\CF})|_U @>{\on{pr}^{\CF}_{\la^{\!\scriptscriptstyle\vee},\mu^{\!\scriptscriptstyle\vee}}}|_U>> (\CV^{\la^{\!\scriptscriptstyle\vee}+\mu^{\!\scriptscriptstyle\vee}}_{\CF})|_U,
\end{CD}\end{equation*}

\bigskip

\begin{equation*}\begin{CD}
(\CV^{\la^{\!\scriptscriptstyle\vee}}_{\CF} \otimes \CV^{\mu^{\!\scriptscriptstyle\vee}}_{\CF})|_U @>{\on{pr}^{\CF}_{\la^{\!\scriptscriptstyle\vee},\mu^{\!\scriptscriptstyle\vee}}}|_U>> (\CV^{\la^{\!\scriptscriptstyle\vee}+\mu^{\!\scriptscriptstyle\vee}}_{\CF})|_U\\
@VV\zeta_{\la^{\!\scriptscriptstyle\vee}} \otimes \zeta_{\mu^{\!\scriptscriptstyle\vee}}V @VV\zeta_{\la^{\!\scriptscriptstyle\vee}+\mu^{\!\scriptscriptstyle\vee}}V\\
\CO_{U}\otimes \CO_{U} @>\on{Id}\otimes \on{Id}>> \CO_{U}.\end{CD} \end{equation*}

c) Given a morphism $\on{pr}\colon  V^{\la^{\!\scriptscriptstyle\vee}} \otimes
V^{\mu^{\!\scriptscriptstyle\vee}} \ra V^{\nu^{\!\scriptscriptstyle\vee}}$ for
$\la^{\!\scriptscriptstyle\vee},\mu^{\!\scriptscriptstyle\vee},\nu^{\!\scriptscriptstyle\vee} \in \La^{\vee+},\
\nu^{\!\scriptscriptstyle\vee} < \la^{\!\scriptscriptstyle\vee}+\mu^{\!\scriptscriptstyle\vee}$, we have
\begin{equation*}
  \on{pr}^{\CF} \circ (\eta_{\la^{\!\scriptscriptstyle\vee}} \otimes \eta_{\mu^{\!\scriptscriptstyle\vee}})=0,\
  (\zeta_{\la^{\!\scriptscriptstyle\vee}} \otimes \zeta_{\mu^{\!\scriptscriptstyle\vee}}) \circ \on{pr}^{\CF}=0.
\end{equation*}

d) For $\la^{\!\scriptscriptstyle\vee}=0$ we have $\zeta_{\la^{\!\scriptscriptstyle\vee}} = \on{Id}$ and $\eta_{\la^{\!\scriptscriptstyle\vee}} = \on{Id}$.

\begin{prop} \label{tannak via ordinary def of Gr}
For $n \in \BN$, the functors $\Gr_{G,X^n}$ and $\Gr'_{G,X^n}$
are isomorphic.
\end{prop}
\begin{proof}
Let us construct a morphism of functors
$\Xi\colon\Gr'_{G,X^n} \ra \Gr_{G,X^n}$.
Take an $S$-point $(\ul{x},\CF,\eta_{\la^{\lvee}},\zeta_{\la^{\lvee}}) \in \Gr'_{G,X^n}(S)$. Restrictions of the morphisms $\eta_{\la^{\lvee}},\zeta_{\la^{\lvee}}$
to the open subvariety $U \subset S\times X$
define the transversal $N$ and $N_{-}$-structures in the $G$-bundle $\CF|_U$.
They define a trivialization $\sigma$ of $\CF|_U$.
Set $\Xi(\ul{x},\CF,\eta_{\la^{\lvee}},\zeta_{\la^{\lvee}}):=(\ul{x},\CF,\sigma) \in \Gr_{G,X^n}$.

Let us construct the inverse morphism $\Xi^{-1}\colon\Gr_{G,X^n} \ra \Gr'_{G,X^n}$.
Take an $S$-point $(\ul{x},\CF,\sigma)\in \Gr_{G,X^n}$. Note that the standard $N,N_{-}$-structures in
the trivial $G$-bundle on $S\times X$ define
via $\sigma$ the $N$ and $N_{-}$-structures in $\CF|_{U}$.
Thus for every $\la^{\lvee} \in \La^{\vee+}$ we have morphisms of vector bundles
$\CO_{U} \xrightarrow{\eta'_{\la^{\!\scriptscriptstyle\vee}}} {\CV^{\la^{\!\scriptscriptstyle\vee}}_{\CF}}|_{U}
\xrightarrow{\zeta'_{\la^{\!\scriptscriptstyle\vee}}}\CO_{U}$.
The morphisms $\eta'_{\la^{\lvee}},~\zeta'_{\la^{\lvee}}$ come from
morphisms
\begin{equation*}
  \CO_{S\times X}(l_{\la^{\lvee}}\cdot(\varGamma_{x_1}\cup\dots\cup\varGamma_{x_n})) \xrightarrow{\eta_{\la^{\lvee}}} {\CV^{\la^{\lvee}}_{\CF}} \xrightarrow{\zeta_{\la^{\lvee}}}
 \CO_{S\times X}(k_{\la^{\lvee}}\cdot(\varGamma_{x_1}\cup\dots\cup\varGamma_{x_n}))
\end{equation*}
for some integers $l_{\la^{\lvee}},k_{\la^{\lvee}}$.

Set $\Xi^{-1}(\ul{x},\CF,\sigma):=(\ul{x},\CF,\eta_{\la^{\lvee}},\zeta_{\la^{\lvee}})$.
It is easy to see that the morphisms $\Xi,\Xi^{-1}$ are mutually inverse.
\end{proof}

\sssec{Definition of $\ol{S}_\nu$ via Tannakian approach} We fix a point $x \in X$.
Let us give the Tannakian definition of the ind-scheme $\ol{S}_\nu\subset\Gr_G$.
The corresponding functor of points associates to a scheme $S$

1) a $G$-bundle $\CF$ on $S\times X$,

2) for every $\la^{\!\scriptscriptstyle\vee} \in \Lambda^{\vee+}$, morphisms of sheaves
$
\eta_{\la^{\lvee}}\colon\CO_{S\times X}(-\langle\la^{\lvee},\nu\rangle \cdot (S\times x)) \ra \CV^{\la^{\!\scriptscriptstyle\vee}}_{\CF}
$ and rational morphisms
$
\zeta_{\la^{\!\scriptscriptstyle\vee}}\colon \CV^{\la^{\!\scriptscriptstyle\vee}}_{\CF} \ra \CO_{S\times X}
$
regular on $S\times (X\setminus \{x\}),$
satisfying the same conditions as in the definition of $\Gr'_{G,X^n}$
in~\S\ref{Tannakian def of BD}.

Allowing $x\in X$ to vary, we obtain an ind-scheme $\ol{S}_{\nu,X}\to X$.

\begin{rem} {\em{Note that the open ind-subscheme $S_{\nu} \subset \ol{S}_{\nu}$ consists of such $(\CF,\eta_{\la^{\lvee}},\zeta_{\la^{\lvee}}) \in \ol{S}_{\nu}$
that $\eta_{\la^{\lvee}}$ are injective morphisms of vector bundles.}}
\end{rem}

\sssec{Definition of $\ol{T}_\nu$ via Tannakian approach} We fix a point $x \in X$. Let us give the Tannakian definition of the ind-scheme $\ol{T}_\nu$. The corresponding functor of points associates to a scheme $S$

1) a $G$-bundle $\CF$ on $X\times S$,

2) for every $\la^{\!\scriptscriptstyle\vee} \in \Lambda^{\vee+}$, rational morphisms $\eta_{\la^{\lvee}}\colon\CO_{S\times X} \ra \CV^{\la^{\!\scriptscriptstyle\vee}}_{\CF}$ regular on $S \times (X\setminus \{x\}),$ and morphisms of sheaves $\zeta_{\la^{\!\scriptscriptstyle\vee}}\colon \CV^{\la^{\!\scriptscriptstyle\vee}}_{\CF} \ra \CO_{S\times X}(-\langle\la^{\lvee},\nu\rangle \cdot (S\times x))$,
satisfying the same conditions as in the definition of $\Gr'_{G,X^n}$
in~\S\ref{Tannakian def of BD}.

Allowing $x\in X$ to vary, we obtain an ind-scheme $\ol{T}_{\nu,X}\to X$.

\begin{rem} {\em{Note that the open ind-subscheme $T_{\nu} \subset \ol{T}_{\nu}$ consists of such $(\CF,\eta_{\la^{\lvee}},\zeta_{\la^{\lvee}}) \in \ol{T}_{\nu}$
that $\zeta_{\la^{\lvee}}$ are surjective morphisms of vector bundles.}}
\end{rem}

%\sssec{Definition of $\ol{S}_{\nu,X}$}
%Let $\ol{S}_{\nu,X}$ be the ind-scheme that represents the following functor.
%It associates to a scheme $S$

%1) an $S$-point $x \in X(S)$ of the curve $X$,

%2) a $G$-bundle $\CF$ on $S\times X$,

%3) for every $\la^{\!\scriptscriptstyle\vee} \in \Lambda^{\vee+}$, morphisms of sheaves
%$\eta_{\la^{\lvee}}\colon\CO_{S\times X}(-\langle\la^{\lvee},\nu\rangle \cdot \varGamma_x)
%\ra \CV^{\la^{\!\scriptscriptstyle\vee}}_{\CF}$, and rational morphisms
%$\zeta_{\la^{\!\scriptscriptstyle\vee}}\colon \CV^{\la^{\!\scriptscriptstyle\vee}}_{\CF}
%\ra \CO_{S\times X}$ regular on $(S \times X)\setminus\varGamma_x,$

%satisfying the same conditions as in the definition of $\Gr_{G,X^n}$
%in \S\ref{Tannakian def of BD}.

\sssec{Definition of $\ol{S}_{\theta_1,\theta_2}$} Fix two cocharacters
$\theta_1,\theta_2 \in \La$. Let $\ol{S}_{\theta_1,\theta_2}$
be the following moduli space: it
associates to a scheme $S$

1) a pair of $S$-points $(x_1,x_2) \in X^2(S)$ of the curve $X$,

2) a $G$-bundle $\CF$ on $S\times X$,

3) for every $\la^{\lvee} \in \La^{\vee+}$,  morphisms of sheaves
$\eta_{\la^{\lvee}}$,
\begin{equation*}
\eta_{\la^{\lvee}}\colon
\CO_{S\times X}(-\langle\la^{\lvee},\theta_1\rangle\cdot\varGamma_{x_1}
-\langle\la^{\lvee},\theta_2\rangle\cdot\varGamma_{x_2})
\ra \CV^{\la^{\lvee}}_{\CF},
\end{equation*}
and rational  morphisms
$
\zeta_{\la^{\lvee}}\colon\CV^{\la^{\lvee}}_{\CF} \ra
\CO_{S\times X}
$
regular on $(S\times X) \setminus
\{\varGamma_{x_1} \cup \varGamma_{x_2}\},$
satisfying the same conditions as in the definition of $\Gr'_{G,X^n}$ in \S\ref{Tannakian def of BD}.

%\sssec{Definition of $\ol{T}_{\nu,X}$}
%Fix a cocharacter $\nu \in \La$. Let $\ol{T}_\nu\subset\Gr_G$ be the ind-scheme that
%represents the following functor.
%It associates to a scheme $S$

%1) an $S$-point $x \in X(S)$ of the curve $X$,

%2) a $G$-bundle $\CF$ on $S\times X$,

%3) for every $\la^{\!\scriptscriptstyle\vee} \in \Lambda^{\vee+}$ morphisms of sheaves
%\begin{equation*}\zeta_{\la^{\!\scriptscriptstyle\vee}}\colon
%\CV^{\la^{\!\scriptscriptstyle\vee}}_{\CF} \ra \CO_{S\times X}(-\langle\la^{\lvee},\nu\rangle
%\cdot \varGamma_x),\end{equation*} and rational morphisms
%$\eta_{\la^{\lvee}}\colon\CO_{S\times X} \ra \CV^{\la^{\!\scriptscriptstyle\vee}}_{\CF}$
%that become regular being restricted to $(S \times X)\setminus\varGamma_x,$

%satisfying the same conditions as in the definition of $\Gr_{G,X^n}$
%in~\S\ref{Tannakian def of BD}.

\sssec{Definition of $\ol{T}_{\theta_1,\theta_2}$} Fix two cocharacters
$\theta_1,\theta_2 \in \La$. Let $\ol{T}_{\theta_1,\theta_2}$
be the following moduli space: it
associates to a scheme $S$

1) a pair of $S$-points $(x_1,x_2) \in X^2(S)$ of the curve $X$,

2) a $G$-bundle $\CF$ on $S\times X$,

3) for every $\la^{\lvee} \in \La^{\vee+}$,  morphisms of sheaves,
\begin{equation*}
\zeta_{\la^{\lvee}}\colon\CV^{\la^{\lvee}}_{\CF} \ra
\CO_{S\times X}(-\langle\la^{\lvee},\theta_1\rangle\cdot\varGamma_{x_1}
-\langle\la^{\lvee},\theta_2\rangle\cdot\varGamma_{x_2})
\end{equation*}
and rational  morphisms
$\eta_{\la^{\lvee}}\colon \CO_{S\times X}\ra \CV^{\la^{\lvee}}_{\CF}$
regular on $(S\times X) \setminus
\{\varGamma_{x_1} \cup \varGamma_{x_2}\},$
satisfying the same conditions as in the definition of $\Gr'_{G,X^n}$ in \S\ref{Tannakian def of BD}.

\subsection{The Vinberg semigroup}
In~\cite{v,v2} Vinberg has defined a multi-parameter degeneration $\on{Vin}_G \ra \BA^r$ of a reductive group $G$ of semisimple rank $r$. Let us recall two equivalent constructions of $\on{Vin}_G$.

\sssec{Rees construction of the Vinberg semigroup~\cite{v2}}
Consider the regular action of $G\times G$ on $G\colon (g_1, g_2) \cdot g:=
g_1 \cdot g \cdot g_2^{-1}$.

\begin{prop}[Peter-Weyl theorem]
\label{PW}
The morphism \begin{equation*}\psi\colon \bigoplus_{\lambda^{\!\scriptscriptstyle\vee}\in\Lambda^{\vee+}} 
V^{\lambda^{\!\scriptscriptstyle\vee}} \otimes (V^{\lambda^{\!\scriptscriptstyle\vee}})^* \iso\BC[G],\
  v \otimes v^{\!\scriptscriptstyle\vee} 
  \mapsto [g \mapsto\langle v,gv^{\!\scriptscriptstyle\vee}\rangle]\end{equation*}
is an isomorphism of $G \times G$-modules.
\end{prop}

The isomorphism $\psi$ induces a $\Lambda^{\vee+}$-grading on $\BC[G]$. This grading is not
compatible with the algebra structure on $\BC[G]$. It is easy to see that the algebra
structure on $\BC[G]$ is compatible with the corresponding $\Lambda^\vee$-filtration
(the Peter-Weyl filtration):
\begin{equation*}\BC[G]_{\leq \lambda^{\!\scriptscriptstyle\vee}}:=
\psi(\bigoplus_{\mu^{\!\scriptscriptstyle\vee} \leq \lambda^{\!\scriptscriptstyle\vee}}
V^{\mu^{\!\scriptscriptstyle\vee}} \otimes (V^{\mu^{\!\scriptscriptstyle\vee}})^*).
\end{equation*}

Let $\BC[\La^\vee]$ denote the group algebra of $\La^\vee$: it is generated by the formal variables $t^{\lambda^{\!\scriptscriptstyle\vee}}$ with relations $t^{\lambda^{\!\scriptscriptstyle\vee}}
\cdot t^{\mu^{\!\scriptscriptstyle\vee}} = t^{\lambda^{\!\scriptscriptstyle\vee}+\mu^{\!\scriptscriptstyle\vee}}$
for $\lambda^{\!\scriptscriptstyle\vee}, \mu^{\!\scriptscriptstyle\vee} \in \Lambda^\vee$.

\begin{defn} \label{Vinberg via Rees}
  \label{Vin} The Vinberg semigroup $\on{Vin}_G$ for $G$ is defined as the spectrum of the Rees algebra for $\BC[G]$ with the Peter-Weyl filtration:
  \begin{equation*}\on{Vin}_G:= \on{Spec}(\bigoplus_{\lambda^{\!\scriptscriptstyle\vee} \in \Lambda^\vee}
  \BC[G]_{\leq \lambda^{\!\scriptscriptstyle\vee}}t^{\lambda^{\!\scriptscriptstyle\vee}}).\end{equation*}
\end{defn}

Note that $\on{Vin}_G$ is equipped with a natural $G \times G$-action. Let us also note that the algebra $\BC[\on{Vin}_G]$ can be equipped with the comultiplication morphism \begin{equation*} \label{1} \Delta\colon \BC[\on{Vin}_G] \ra \BC[\on{Vin}_G] \otimes \BC[\on{Vin}_G]\end{equation*}
that is induced from the comultiplication morphisms for Hopf algebras $\BC[G],~\BC[\La^\vee]$.
It follows that $\BC[\on{Vin}_G]$ carries a bialgebra structure. So $\on{Vin}_G$ is an algebraic
monoid (semigroup).

\sssec{Construction of the morphism $\Upsilon\colon \on{Vin}_G \ra T^+_{\on{ad}}$}
\label{Morphism v for Vin} We consider the polynomial subalgebra
$\BC[t^{\alpha^{\!\scriptscriptstyle\vee}_i},\ i\in I]$ of
$\BC[\La^\vee]$ generated by the elements $t^{\alpha^{\!\scriptscriptstyle\vee}_i},\ i \in I$.
Set $T_{\on{ad}}^+:=\on{Spec}(\BC[t^{\alpha^{\!\scriptscriptstyle\vee}_i},\ i\in I])\simeq\BA^r$.
Let us denote by $Z_G \subset T$ the center of the group $G$. Set $T_{\on{ad}}:=T/Z_G$. We have a natural open embedding
$T_{\on{ad}}=\on{Spec}(\BC[t^{\pm\alpha^{\!\scriptscriptstyle\vee}_i},\ i\in I])\hookrightarrow T_{\on{ad}}^+$.
Thus $T_{\on{ad}}^+$ is the toric variety for $T_{\on{ad}}$, and
the corresponding cone is $\Lambda^{\vee,\pos}$.
Note that the variety $T_{\on{ad}}^+$ has a unique monoid structure extending the multiplication in
$T_{\on{ad}}$.

Observe that $t^{\alpha_i^{\!\scriptscriptstyle\vee}} \in \BC[\on{Vin}_G]$ for any simple root
$\alpha^{\!\scriptscriptstyle\vee}_i$, so there is a homomorphism
$\BC[t^{\alpha^{\!\scriptscriptstyle\vee}_i},\ i\in I] \hookrightarrow \BC[\on{Vin}_G]$.
Thus we get a homomorphism of monoids $\Upsilon\colon \on{Vin}_G \ra T_{\on{ad}}^+$.
Note that the algebra of functions $\BC[\on{Vin}_G]$ is $\La^\vee$-graded.
It follows that the torus $\on{Spec}(\BC[\La^\vee])=T$ acts on the scheme $\on{Vin}_G$.
It also acts on the affine space $T_{\on{ad}}^+$ in the following way:
\begin{equation*}t\cdot (a_i):=(\alpha^{\!\scriptscriptstyle\vee}_i(t)\cdot a_i)~\on{in~coordinates}~
\alpha^{\!\scriptscriptstyle\vee}_i,\ i\in I.\end{equation*}
The morphism $\Upsilon$ is $T$-equivariant. Recall that $Z_G \subset T$ is the center of the group $G$.
It acts trivially on $T_{\on{ad}}^+$, thus the action of $T$ on $T_{\on{ad}}^+$ factors
through the action of $T_{\on{ad}}:=T/Z_G$.

The fiber $\Upsilon^{-1}(1)$ over $1\in T_{\on{ad}}\subset T_{\on{ad}}^+$
naturally identifies with the group $G$. Recall that the morphism $\Upsilon$ is $T$-equivariant.
It follows that the preimage $\Upsilon^{-1}(T_{\on{ad}})$ is isomorphic to
\begin{equation*} (G \times T)/Z_G=:G_{\on{enh}}\end{equation*} where the action of $Z_G \curvearrowright G \times T$ is given by the formula
$(g,t)\cdot z:= (z\cdot g, z^{-1} \cdot t)$.

We denote by $j$ the corresponding open embedding
$G_{\on{enh}} \hookrightarrow \on{Vin}_G$.

\sssec{The section of the morphism $\Upsilon$} \label{section}
Let us construct a surjective morphism
$\mathfrak{s}^{*}\colon\BC[\on{Vin}_G] \twoheadrightarrow \BC[T^+_{\on{ad}}]$.
Decompose \begin{equation*}\BC[\on{Vin}_G]=\bigoplus\limits_{\mu^{\!\scriptscriptstyle\vee} \leq \lambda^{\!\scriptscriptstyle\vee} \in\Lambda^{\vee}}
 \psi(
V^{\mu^{\!\scriptscriptstyle\vee}} \otimes(V^{\mu^{\!\scriptscriptstyle\vee}})^*)t^{\lambda^{\!\scriptscriptstyle\vee}}
=
\bigoplus\limits_{\mu^{\!\scriptscriptstyle\vee} \leq \lambda^{\!\scriptscriptstyle\vee} \in\Lambda^{\vee},~\nu^{\lvee}_1,\nu^{\lvee}_2 \in \La^{\vee}}
 \psi(
V^{\mu^{\!\scriptscriptstyle\vee}}_{\nu^{\lvee}_1}
\otimes(V^{\mu^{\!\scriptscriptstyle\vee}}_{\nu^{\lvee}_2})^*)t^{\lambda^{\!\scriptscriptstyle\vee}}.
\end{equation*}
Set
$\mathfrak{s}^{*}(\psi(w_{\nu^{\lvee}_1}\otimes f_{\nu^{\lvee}_2})t^{\la^{\lvee}}):=
\langle w_{\nu^{\lvee}_1},f_{\nu^{\lvee}_2}\rangle t^{\la^{\lvee}-\nu^{\lvee}_1}$
for $w_{\nu^{\lvee}_1} \in V^{\mu^{\!\scriptscriptstyle\vee}}_{\nu^{\lvee}_1}, f_{\nu^{\lvee}_2} \in
(V^{\mu^{\!\scriptscriptstyle\vee}}_{\nu^{\lvee}_2})^*$.
The morphism $\mathfrak{s}^{*}$ corresponds to the section
$\mathfrak{s}\colon T^+_{\on{ad}} \hookrightarrow \on{Vin}_G$
of the morphism $\Upsilon$.

\sssec{The nondegenerate locus of the Vinberg semigroup~\cite[\S0.8]{v2},~\cite[Section~D.4]{dg1}}
\label{nondeg locus}
The Vinberg semigroup contains a dense open subvariety
$_0\!\!\on{Vin}_G \subset \on{Vin}_G$, the {\em{nondegenerate locus}} of $\on{Vin}_G$.
It is uniquely characterized by the fact that it meets
each fiber of the morphism
$\Upsilon\colon\on{Vin}_G \ra T^+_{\on{ad}}$ in the open
$G \times G$-orbit of that fiber; i.e., for any
$t \in T^+_{\on{ad}}$ we have:
\begin{equation*}
{\on{Vin}_G}|_{t} \cap~_0\!\!\on{Vin}_G= (G\times G)  \cdot\mathfrak{s}(t).
\end{equation*}
The Tannakian characterization of $_0\!\!\on{Vin}_G$ will be given
in \S\ref{nondegenerate in Tannak}.

\sssec{The zero fiber of $\Upsilon$}
Let us describe the zero fiber of the morphism $\Upsilon\colon \on{Vin}_G \ra T^+_{\on{ad}}$. Let us denote
$
(\on{Vin}_G)_{0}:=\Upsilon^{-1}(0).
$

\begin{lem} (cf.~\cite[Lemma~2.1.11]{s2}) \label{fiber for Vin} The scheme $(\on{Vin}_G)_{0}$ is isomorphic to \begin{equation*}(\ol{G/N_{-}} \times \ol{G/N})/\!\!/T:=\on{Spec}((\BC[G]^{{N_-}} \otimes \BC[G]^{{N}})^T) \end{equation*}
where the action of $T\curvearrowright G/N_{-},~G/N$ is given by the right multiplication. The $G\times G$-action on $(\on{Vin}_G)_{0}$ corresponds to
the action via the left multiplication $G\times G \curvearrowright(\ol{G/N_{-}} \times \ol{G/N})/\!\!/T$.
\end{lem}

\begin{proof} It follows from Definition~\ref{Vinberg via Rees} that the $G\times G$-module $\BC[(\on{Vin}_G)_{0}]$ is isomorphic to
$
\bigoplus_{\lambda^{\!\scriptscriptstyle\vee}\in\Lambda^{\vee+}}
V^{\lambda^{\!\scriptscriptstyle\vee}} \otimes (V^{\lambda^{\!\scriptscriptstyle\vee}})^*,
$
  and the algebra structure on $\BC[(\on{Vin}_G)_{0}]$ is given by the projection morphisms \begin{equation*}(V^{\lambda_1^{\!\scriptscriptstyle\vee}} \otimes (V^{\lambda_1^{\!\scriptscriptstyle\vee}})^*) \otimes (V^{\lambda_2^{\!\scriptscriptstyle\vee}} \otimes (V^{\lambda_2^{\!\scriptscriptstyle\vee}})^*) \twoheadrightarrow V^{\lambda_1^{\!\scriptscriptstyle\vee}+\lambda_2^{\!\scriptscriptstyle\vee}} \otimes (V^{\lambda_1^{\!\scriptscriptstyle\vee}+\lambda_2^{\!\scriptscriptstyle\vee}})^*. \end{equation*}
So there is an embedding of algebras \begin{equation*}\BC[(\on{Vin}_G)_{0}]=\bigoplus_{\lambda^{\!\scriptscriptstyle\vee}\in\Lambda^{\vee+}}
V^{\lambda^{\!\scriptscriptstyle\vee}} \otimes (V^{\lambda^{\!\scriptscriptstyle\vee}})^* \hookrightarrow \bigoplus_{\lambda^{\!\scriptscriptstyle\vee}, \mu^{\!\scriptscriptstyle\vee} \in\Lambda^{\vee+}}
V^{\lambda^{\!\scriptscriptstyle\vee}} \otimes (V^{\mu^{\!\scriptscriptstyle\vee}})^*=\BC[G]^{N_{-}}\otimes \BC[G]^{N}.\end{equation*}  It is easy to see that the image of $\BC[(\on{Vin}_G)_{0}]$ coincides with $(\BC[G]^{N_-} \otimes \BC[G]^{N})^T$.
\end{proof}

\begin{rem}\label{fiber for def free Vin}{\em{
It follows from Lemma~\ref{fiber for Vin} that the open subscheme
$_{0}\!\on{Vin}_G \cap (\on{Vin}_G)_0~\subset (\on{Vin}_G)_0$ is isomorphic to
$(G/N_{-} \times G/N_{})/T \subset (\ol{G/N_{-}} \times \ol{G/N})/\!\!/T$}}.
\end{rem}

\sssec{The scheme $\on{Vin}_{G}/\!\!/(N\times N_{-})$} Let us denote by $\ol{T}$ the monoid
\begin{equation*}
  \ol{T}:=(\on{Vin}_G)_{0}/\!\!/(N\times N_{-})=
  \on{Spec}((\BC[G]^{N\times N_{-}} \otimes \BC[G]^{N_{-} \times N})^T)=
  \on{Spec}\BC[\Lambda^{\vee+}].\end{equation*}

The monoid $\ol{T}$ is the toric variety for the torus $T$.

\begin{rem} {\em Let us point out the difference between the monoids $T^+_{\on{ad}},~\ol{T}$.
The monoid $T^+_{\on{ad}}$ is the toric variety for the torus $T_{\on{ad}}$
and the corresponding cone is $\Lambda^{\vee,\pos}$.
The monoid $\ol{T}$ is the toric variety for the torus $T$, the corresponding cone is $\La^{\vee+}$.}
\end{rem}

\begin{lem} (cf.~\cite[Lemma~4.1.3]{s3}) \label{Factor by unipotent} The scheme $\on{Vin}_{G}/\!\!/(N\times N_{-})$ is isomorphic to \begin{equation*} ((\on{Vin}_{G})_{0}/\!\!/(N\times N_{-})) \times T^+_{\on{ad}}=\ol{T}\times T^+_{\on{ad}}.\end{equation*}
\end{lem}

\begin{proof} Note that \begin{equation*}\BC[\on{Vin}_{G}/\!\!/(N\times N_{-})]:=\BC[\on{Vin}_{G}]^{N\times N_{-}}=\bigoplus_{\lambda^{\!\scriptscriptstyle\vee} \in \Lambda^{\vee}}
  \BC[G]_{\leq \lambda^{\!\scriptscriptstyle\vee}}^{N \times N_{-}}t^{\lambda^{\!\scriptscriptstyle\vee}}.\end{equation*}
The algebra $\BC[G]^{N\times N_{-}}$ is isomorphic to
$\bigoplus_{\lambda^{\!\scriptscriptstyle\vee} \in \Lambda^{\vee+}}
\BC\cdot (v_{\la^{\lvee}} \otimes v_{\la^{\lvee}}^{\lvee})
\simeq \BC[\ol{T}]$, where $v_{\la^{\lvee}}\in V^{\lambda^{\lvee}}$ (resp.\
$v_{\la^{\lvee}}^{\lvee}\in(V^{\lambda^{\lvee}})^*$) is the highest (resp.\ lowest) vector.
Thus we have \begin{equation*}\BC[G]_{\leq \lambda^{\!\scriptscriptstyle\vee}}^{N\times N_{-}}=
  \bigoplus_{\mu^{\lvee} \leq \lambda^{\!\scriptscriptstyle\vee} \in \Lambda^{\vee}}
\BC\cdot (v_{\mu^{\lvee}} \otimes v_{\mu^{\lvee}}^{\lvee}).\end{equation*}
Now the isomorphism \begin{equation*}
\BC[\ol{T}
    \times T^+_{\on{ad}}] \simeq
    \bigoplus_{\mu^{\!\scriptscriptstyle\vee} \in \Lambda^{\vee+},~\la^{\lvee} \in \Lambda^{\vee,\pos}} \BC
  \cdot (v_{\mu^{\lvee}} \otimes v_{\mu^{\lvee}}^{\lvee} \otimes t^{\la^{\lvee}}) \iso
  \bigoplus_{\lambda^{\!\scriptscriptstyle\vee} \in \Lambda^{\vee}}
  \BC[G]_{\leq \lambda^{\!\scriptscriptstyle\vee}}^{N\times N_{-}}t^{\lambda^{\!\scriptscriptstyle\vee}}=
\BC[\on{Vin}_{G}/\!\!/(N\times N_{-})]
  \end{equation*}
is given by
$
v_{\mu^{\lvee}} \otimes
v_{\mu^{\lvee}}^{\lvee} \otimes t^{\la^{\lvee}} \mapsto
(v_{\mu^{\lvee}} \otimes v_{\mu^{\lvee}}^{\lvee})\cdot t^{\la^{\lvee}+\mu^{\lvee}}.
$

\end{proof}

\sssec{The example $G=\on{SL}_2$}
For $G=\on{SL}_2$ the semigroup $\on{Vin}_G$ is equal to the semigroup of $2\times 2$ matrices $\on{Mat}_{2\times 2}$. The $\on{SL}_2 \times \on{SL}_2$-action is given by the left and right multiplication, the action of $T\simeq \BC^{\times}$ is given by the scalar multiplication. The morphism $\Upsilon$ is equal to the determinant map:
\begin{equation*} \Upsilon\colon \on{Vin}_G=\on{Mat}_{2\times 2} \xrightarrow{\on{det}} \BA^1=T^+_{\on{ad}}.\end{equation*}
It follows that the preimage $\Upsilon^{-1}(T_{\on{ad}})$ is equal to $\on{GL}_2$.
The nondegenerate locus $_0\!\!\on{Vin}_G$ is equal to $\on{Mat}_{2\times 2} \setminus \{0\}$. The section $\mathfrak{s}$ of \S\ref{section} is
given by $a \mapsto (\begin{smallmatrix} 1&0 \\ 0& a \end{smallmatrix})$.

\sssec{Tannakian definition of the Vinberg semigroup} \label{Tannakian def of Vin}
Recall the open dense embedding $j\colon G_{\on{enh}} \hookrightarrow \on{Vin}_G$
of algebraic monoids~(\S\ref{Morphism v for Vin}). Its image coincides with the group of units in $\on{Vin}_G$. The group $G_{\on{enh}}$ is reductive, and
$\on{Vin}_G$ is irreducible and affine. The tensor category $\on{Rep}(\on{Vin}_G)$ of finite-dimensional representations of $\on{Vin}_G$ is the full tensor subcategory in the category $\on{Rep}(G_{\on{enh}})$.
Let us describe this subcategory $\on{Rep}(\on{Vin}_G) \subset \on{Rep}(G_{\on{enh}})$.

To do so, we first introduce the following notation. Any representation $V$ of $G_{\on{enh}}$
admits a canonical decomposition as $G_{\on{enh}}$-representations
$V=\bigoplus\limits_{\lambda^{\!\scriptscriptstyle\vee} \in \La^\vee} V_{\lambda^{\!\scriptscriptstyle\vee}}$
according to the action of the center $Z_{G_{\on{enh}}}=(Z_G \times T)/Z_G \simeq T$:
the center $Z_{G_{\on{enh}}}$ acts on the summand $V_{\lambda^{\!\scriptscriptstyle\vee}}$
via the character $\lambda^{\!\scriptscriptstyle\vee}$. Each summand
$V_{\lambda^{\!\scriptscriptstyle\vee}}$ in this decomposition is a $G$-representation via the
inclusion $G \hookrightarrow G_{\on{enh}},~g\mapsto (g,1)$.

The subcategory $\on{Rep}(\on{Vin}_G) \subset \on{Rep}(G_{\on{enh}})$ consists of
$V \in \on{Rep}(G_{\on{enh}})$ such that for each
$\lambda^{\!\scriptscriptstyle\vee} \in \La^\vee$ the weights of the summand
$V_{\lambda^{\!\scriptscriptstyle\vee}}$ considered as a $G$-representation, are all
$\leq \lambda^{\!\scriptscriptstyle\vee}$.

According to the Tannakian formalism, the functor of points
\begin{equation*}\on{Vin}_G\colon{\bf{Sch}} \ra {\bf{Set}},~S \mapsto \on{Vin}_G(S)\end{equation*}
can be described as follows. The set $\on{Vin}_G(S)$ consists of the following data:

1) for every $\la^{\!\scriptscriptstyle\vee} \in \La^{\vee+}$, a morphism
$g_{\la^{\!\scriptscriptstyle\vee}}\colon S \ra \on{End}(V^{\la^{\!\scriptscriptstyle\vee}})$,

2) for every $\mu^{\!\scriptscriptstyle\vee} \in \Lambda^{\vee,\pos}$, a regular function
$\tau_{\mu^{\!\scriptscriptstyle\vee}}$ on $S$,\\
such that the following conditions hold.

a) For any $\mu_1^{\lvee}, \mu_2^{\lvee} \in \Lambda^{\vee,\pos}$ we have $\tau_{\mu_1^{\lvee}} \cdot \tau_{\mu_2^{\lvee}}=\tau_{\mu_1^{\lvee}+\mu_2^{\lvee}}.$

b) For any $\la^{\!\scriptscriptstyle\vee}_{1}, \la^{\!\scriptscriptstyle\vee}_{2},
\nu^{\!\scriptscriptstyle\vee} \in\Lambda^{\vee+}$ such that $V^{\nu^{\!\scriptscriptstyle\vee}}$
enters $V^{\la^{\!\scriptscriptstyle\vee}_{1}} \otimes V^{\la^{\!\scriptscriptstyle\vee}_{2}}$ with nonzero
multiplicity, we denote by $\iota\colon  W^{\nu^{\!\scriptscriptstyle\vee}} \hookrightarrow
V^{\la^{\!\scriptscriptstyle\vee}_{1}} \otimes V^{\la^{\!\scriptscriptstyle\vee}_{2}}$ the embedding of
the corresponding isotypical component, and by $\on{pr}\colon  V^{\la^{\!\scriptscriptstyle\vee}_{1}}
\otimes V^{\la^{\!\scriptscriptstyle\vee}_{2}} \twoheadrightarrow W^{\nu^{\!\scriptscriptstyle\vee}}$ the corresponding
projection. Then we have \begin{equation*}\on{pr} \circ (g_{\la_{1}^{\!\scriptscriptstyle\vee}}
  \otimes g_{\la_{2}^{\!\scriptscriptstyle\vee}}) \circ \iota=
  \tau_{\la^{\!\scriptscriptstyle\vee}_{1}+\la^{\!\scriptscriptstyle\vee}_{2}-\nu^{\!\scriptscriptstyle\vee}}
  \cdot g_{\nu^{\!\scriptscriptstyle\vee}},\end{equation*}

c) $g_0$ sends $S$ to $\on{Id} \in \on{End}(\BC)$, $\tau_0$ sends $S$ to $1 \in \BC$.

\sssec{The morphism $\Upsilon$ and the section $\mathfrak{s}$ via
Tannakian approach} \label{Upsilon and s via Tannak}
The functor of points
\begin{equation*}
T^+_{\on{ad}}\colon {\bf{Sch}} \ra {\bf{Set}},~ S \mapsto T^+_{\on{ad}}(S),
\end{equation*}
can be described as follows. The set $T^+_{\on{ad}}(S)$
consists of the following data:

For every $\mu^{\lvee} \in \Lambda^{\vee,\pos}$, a regular function $\tau_{\mu^{\lvee}}$ on $S$,
such that for any $\mu^{\lvee}_1, \mu^{\lvee}_2 \in \Lambda^{\vee,\pos}$ we have
$
\tau_{\mu^{\lvee}_1}\cdot\tau_{\mu^{\lvee}_2}=
\tau_{\mu^{\lvee}_1+\mu^{\lvee}_2}.
$
and $\tau_0$ sends $S$ to $1 \in \BC$.

The morphism $\Upsilon\colon \on{Vin}_G \ra T^+_{\on{ad}}$ of \S\ref{Morphism v for Vin}
forgets the data of $g_{\la^{\lvee}}.$ The morphism $\mathfrak{s}$ of
\S\ref{section} sends an $S$-point $(\tau_{\mu^{\lvee}}) \in T^+_{\on{ad}}(S)$ to an $S$-point
$(g_{\la^{\lvee}},\tau_{\mu^{\lvee}}) \in \on{Vin}_G(S)$
where $g_{\la^{\lvee}}\colon S \ra \on{End}(V^{\la^{\lvee}})$
acts on the weight component
$(V^{\la^{\lvee}})_{\la^{\lvee}-\mu^{\lvee}}$ via the multiplication by
$\tau_{\mu^{\lvee}}$.

\sssec{Nondegenerate locus of the Vinberg semigroup via Tannakian aproach~\cite[\S D.4]{dg1}}
\label{nondegenerate in Tannak}
Recall the nondegenerate locus
$
_0\!\!\on{Vin}_G \subset \on{Vin}_G
$
of \S\ref{nondeg locus}.

In Tannakian terms it corresponds to $(g_{\la^{\!\scriptscriptstyle\vee}},
\tau_{\mu^{\lvee}}) \in \on{Vin}_G(S)$
such that $g_{\la^{\!\scriptscriptstyle\vee}}\ne0$ for any $\la^{\lvee} \in \La^{\vee+}$.

\begin{rem}{\em
The scheme $_0\!\!\on{Vin}_G/T$ is nothing but the
De Concini - Procesi wonderful compactification of
$G_{\on{ad}}:=G/Z_G$~\cite[8.6]{v2}.}
\end{rem}

\sssec{The principal degeneration} Let us denote by
$\on{Vin}_G^{\on{princ}}$ the restriction of the multi-parameter family
$\Upsilon\colon \on{Vin}_G \ra T_{\on{ad}}^{+}$ of \S\ref{Morphism v for Vin} to the~``principal'' line
\begin{equation*}\BA^1 \hookrightarrow T_{\on{ad}}^+,\ a \mapsto
(a, \dots, a)~\on{in~coordinates}~\al^{\lvee}_i,\ i \in I.
\end{equation*}
Let us denote the corresponding morphism $\on{Vin}^{\on{princ}}_G \ra \BA^1$ by
$\Upsilon^{\on{princ}}$. The morphism $\Upsilon^{\on{princ}}$ is $\BC^{\times}$-equivariant with
respect to the ``diagonal'' $\BC^{\times}$-action via \begin{equation*}\BC^{\times} \hookrightarrow T,~c \mapsto 2\rho(c).\end{equation*}
It follows from \S\ref{Morphism v for Vin} that the preimage $(\Upsilon^{\on{princ}})^{-1}(\BG_m)$
is isomorphic to $G \times \BG_m$.

Let us denote by $_0\!\!\on{Vin}^{\on{princ}}_G$ the intersection \begin{equation*}_0\!\!\on{Vin}^{\on{princ}}_G:=~_0\!\!\on{Vin}_G \cap \on{Vin}^{\on{princ}}_G.\end{equation*}

\subsection{Drinfeld-Lafforgue-Vinberg degeneration of $\on{Bun}_G(X)$~\cite[\S2.2]{s2}}
\label{VinBun}
Let $X$ be a smooth projective curve over $\BC$. Let us now recall the definition of
Drinfeld-Lafforgue-Vinberg multi-parameter degeneration $\on{VinBun}_G(X)$ of the moduli stack
$\on{Bun}_G(X)$. Since the curve $X$ is fixed throughout the paper, we will use the simplified
notations $\on{VinBun}_G, \on{Bun}_G$ instead of $\on{VinBun}_G(X), \on{Bun}_G(X)$.

Recall that for a stack $\CY$, one defines a stack
$ \on{Maps}(X,\CY) $
parametrizing morphisms from the curve $X$ to $\CY$ as
$\on{Maps}(X,\CY)(S):=\CY(X\times S). $

Thus for example we have $\on{Bun}_G=\on{Maps}(X,\on{pt}\!/G)$. Similarly, given an open substack $_0\CY \subset \CY$, a stack
$\on{Maps}_{\on{gen}}(X,\CY \supset~_0\CY)$ associates to a scheme $S$ the following data:
morphisms $f\colon S\times X \ra \CY$ such that for every geometric point $\bar{s} \ra S$ there exists an open dense subset $U \subset \bar{s} \times X$ on which the restricted morphism $f|_{U}\colon U \ra \CY$ factors through the open substack $_0\CY \subset \CY$.

\begin{defn}
  \label{vinbun}
  \cite[\S2.2.2]{s2}. Consider an algebraic stack
  $\on{Vin}_G/(G\times G)$ and an open substack
$_0\!\!\on{Vin}_G/(G\times G) \subset \on{Vin}_G/(G\times G)$.
We define \begin{equation*}
\on{VinBun}_G:= \on{Maps}_{\on{gen}}(X,\on{Vin}_G/(G\times G) \supset~_0\!\!\on{Vin}_G/(G\times G)).
\end{equation*}
We denote by
$\wp_{1,2}\colon \on{VinBun}_G \ra \on{Bun}_G=\on{Maps}(X,\pt/G)$ the two natural projections
arising from $\on{Vin}_G/(G\times G)\to\pt/(G\times G)$.
\end{defn}

Note that the action of $T\curvearrowright \on{Vin}_G$ of \S\ref{Morphism v for Vin} commutes with the $G\times G$-action on $\on{Vin}_G$. It follows that the $T$-action on $\on{Vin}_G$ induces the $T$-action
$
T \curvearrowright \on{VinBun}_G.
$
\sssec{The degeneration morphism $\on{VinBun}_G \ra T^+_{\on{ad}}$}
\label{degeneration mor for VinBun}
The morphism $\Upsilon\colon \on{Vin}_G \ra T^+_{\on{ad}}$
is $G\times G$-equivariant, so we obtain the morphism
\begin{equation*} \Upsilon_{/(G\times G)}\colon\on{Vin}_G/(G\times G) \ra T^+_{\on{ad}}/(G\times G)=T^+_{\on{ad}} \times (\on{pt}/(G\times G)).
\end{equation*}
Let $\on{pr}_1\colon T^+_{\on{ad}} \times (\on{pt}/(G\times G)) \ra T^{+}_{\on{ad}}$ denote the projection morphism to the first factor. The curve $X$ is proper while the variety $T^+_{\on{ad}}$ is affine, so it follows that the morphism

\begin{equation*} T^{+}_{\on{ad}}=\on{Maps}(\on{pt},T^+_{\on{ad}})~\iso~\on{Maps}(X,T^+_{\on{ad}})
\end{equation*}
induced by the morphism $X \ra \on{pt}$ is an isomorphism.

Define
\begin{equation*}
\Upsilon\colon \on{VinBun}_G \ra T^+_{\on{ad}},~\Upsilon(f):= \on{pr}_1 \circ (\Upsilon_{/(G\times G)}) \circ f.
\end{equation*}
The morphism $\Upsilon$ is $T$-equivariant. From \S\ref{Morphism v for Vin} it follows that the preimage $\Upsilon^{-1}(T_{\on{ad}})$ is isomorphic to
$
\on{Bun}_G \times T_{\on{ad}}.
$

\sssec{Defect free locus of $\on{VinBun}_G$}
Set $_{0}\!\!\on{VinBun}_{G}:=
\on{Maps}(X,~_{0}\!\!\on{Vin}_G/(G\times G))$.
The stack $_{0}\!\!\on{VinBun}_{G}$
is a smooth open substack of $\on{VinBun}_G$.
Let us denote by $_{0}\!\Upsilon$ the restriction
of the morphism $\Upsilon$ to $_{0}\!\!\on{VinBun}_{G}$.
Denote
$(_{0}\!\on{VinBun}_{G})_{0}:=~_{0}\!\Upsilon^{-1}(0)$.

\begin{prop} \label{fiber for vinbun} \cite[\S2.4]{s2}.
The stack $(_{0}\!\on{VinBun}_{G})_{0}$ is isomorphic
to the fibre product $\on{Bun}_{B_{-}}\underset{\on{Bun}_T}\times \on{Bun}_{B_{}}$.
\end{prop}

\begin{proof}
Follows from Remark~\ref{fiber for def free Vin}.
\end{proof}

\sssec{The relative compactification $\ol{\on{Bun}}_G$}
Define $\ol{\on{Bun}}_G:=\on{VinBun}_G/T$.

\begin{rem}{\em{
The stack $\ol{\on{Bun}}_G$ contains the stack $\on{Bun}_G\times (\on{pt}\!/Z_G)$ as a smooth open substack.}}
\end{rem}

\begin{prop}[D.~Gaitsgory]
  \label{proper over Bun}
The natural morphism $\kappa\colon\ol{\on{Bun}}_G \ra \on{Bun}_G\times \on{Bun}_G$ is proper.
\end{prop}

\begin{proof} Will be given in Appendix~\ref{proof proper over Bun}.
\end{proof}

\sssec{Tannakian approach to the definition of $\on{VinBun}_G$}
 Let us explain how to deduce the Tannakian description of $\on{VinBun}_G$ from the Tannakian description of $\on{Vin}_G$ of \S\ref{Tannakian def of Vin}.
An $S$-point of $\on{VinBun}_G$ is a morphism $f$ from  $S\times X$ to $\on{Vin}_G/(G\times G)$. Such a morphism is the same as a pair of $G$-torsors $\mathcal{F}_{1},\mathcal{F}_{2}$ on $S\times X$ and a $G\times G$-equivariant morphism from $\mathcal{F}_1\times \mathcal{F}_2$ to $\on{Vin}_G$.
It follows from \S\ref{Tannakian def of Vin} that this morphism is the same as the family of $G\times G$-equivariant morphisms $g_{\lambda^{\!\scriptscriptstyle\vee}}$ ($\lambda^{\!\scriptscriptstyle\vee} \in \Lambda^{\vee+}$) from $ \mathcal{F}_1\times \mathcal{F}_2$ to $\on{End}(V^{\lambda^{\!\scriptscriptstyle\vee}})$ and a $G\times G$-equivariant regular functions $\tau_{\mu^{\!\scriptscriptstyle\vee}},\
\mu^{\!\scriptscriptstyle\vee} \in \La^{\vee,\pos}$, on $\mathcal{F}_1\times \mathcal{F}_2$, satisfying certain conditions.
Note that a $G\times G$-equivariant morphism from $\mathcal{F}_1\times \mathcal{F}_2$ to $\on{End}(V^{\lambda^{\!\scriptscriptstyle\vee}})$ is the same as the morphism from $\mathcal{V}^{\lambda^{\!\scriptscriptstyle\vee}}_{\mathcal{F}_1}$ to $\mathcal{V}^{\lambda^{\!\scriptscriptstyle\vee}}_{\mathcal{F}_2}$ and  a $G\times G$-equivariant regular function on $\mathcal{F}_1\times \mathcal{F}_2$ is the same as the regular function on $S\times X$.

It follows that
the stack $\on{VinBun}_G\colon {\bf{Sch}} \ra {\bf{Gpd}}$
can be described in the following way. It assigns to a scheme $S$ the following data:

1) two right $G$-torsors $\CF_1, \CF_2$ on $X$,

2) for every $\la^{\lvee} \in {\La^{\vee+}}$ a morphism $\varphi_{\la^{\!\scriptscriptstyle\vee}}\colon  \CV^{\la^{\!\scriptscriptstyle\vee}}_{\CF_1} \ra \CV^{\la^{\!\scriptscriptstyle\vee}}_{\CF_2}$,

3) for every $\mu^{\!\scriptscriptstyle\vee} \in \Lambda^{\vee,\pos}$ a morphism $\tau_{\mu^{\!\scriptscriptstyle\vee}}\colon  \CO_{S \times X} \ra \CO_{S \times X}$,\\
such that the following conditions hold.

a) For any $\mu^{\!\scriptscriptstyle\vee}_{1}, \mu^{\!\scriptscriptstyle\vee}_{2} \in \Lambda^{\vee,\pos}$ we have $\tau_{\mu^{\!\scriptscriptstyle\vee}_{1}}\otimes \tau_{\mu^{\!\scriptscriptstyle\vee}_{2}}=\tau_{\mu^{\!\scriptscriptstyle\vee}_{1}+\mu^{\!\scriptscriptstyle\vee}_{2}}.$

b) For every geometric point $\ol{s}\ra S$ and any dominant character $\la^{\lvee} \in \La^\vee$
the morphism $\varphi_{\la^{\lvee}}|_{\ol{s}\times X}\colon
{\CV^{\la^{\!\scriptscriptstyle\vee}}_{\CF_1}}|_{\ol{s}\times X} \ra
{\CV^{\la^{\!\scriptscriptstyle\vee}}_{\CF_2}}|_{\ol{s}\times X}$ is nonzero.

c) For any $\la^{\!\scriptscriptstyle\vee}_{1}, \la^{\!\scriptscriptstyle\vee}_{2},
\nu^{\!\scriptscriptstyle\vee} \in\Lambda^{\vee+}$ such that $V^{\nu^{\!\scriptscriptstyle\vee}}$
enters $V^{\la^{\!\scriptscriptstyle\vee}_{1}} \otimes V^{\la^{\!\scriptscriptstyle\vee}_{2}}$ with nonzero
multiplicity, we denote by $\iota\colon  W^{\nu^{\!\scriptscriptstyle\vee}} \hookrightarrow
V^{\la^{\!\scriptscriptstyle\vee}_{1}} \otimes V^{\la^{\!\scriptscriptstyle\vee}_{2}}$ the embedding of
the corresponding isotypical component and by $\on{pr}\colon  V^{\la^{\!\scriptscriptstyle\vee}_{1}}
\otimes V^{\la^{\!\scriptscriptstyle\vee}_{2}} \twoheadrightarrow W^{\nu^{\!\scriptscriptstyle\vee}}$ the corresponding
projection. We denote by $\iota^{\CF_1}, \on{pr}^{\CF_2}$ the corresponding morphisms between
the induced vector bundles. Then we have \begin{equation*}\on{pr}^{\CF_2} \circ
  (\varphi_{\la^{\!\scriptscriptstyle\vee}_{1}} \otimes \varphi_{\la^{\!\scriptscriptstyle\vee}_{2}})
  \circ \iota^{\CF_1}=
  (\tau_{\la^{\!\scriptscriptstyle\vee}_{1}+\la^{\!\scriptscriptstyle\vee}_{2}-\nu^{\!\scriptscriptstyle\vee}})
  \otimes \varphi_{\nu^{\!\scriptscriptstyle\vee}}.\end{equation*}

d) The morphism $\varphi_0$ coincides with the identity morphism, and
the morphism $\tau_0$ coincides with the identity morphism.

\medskip

Condition b) is a consequence of the fact that the morphism $f$ generically lands into ${}_{0}\!\!\on{Vin}_G/(G\times G)$. Conditions a),c),d) follow from the conditions a),b),c) of \S\ref{Tannakian def of Vin}, respectively.

In Tannakian terms, the morphism $\Upsilon\colon \on{VinBun}_G \ra T^+_{\on{ad}}$ sends an $S$-point
$(\CF_1,\CF_2,\varphi_{\la^{\lvee}},\tau_{\mu^{\lvee}}) \in \on{VinBun}_G(S)$
to the point $(\tau_{\mu^{\lvee}}) \in T^{+}_{\on{ad}}(S)$.

\sssec{Description of $\on{VinBun}_G$ for $G=SL_2$}
Following~\cite[Section~2.1.1]{s1} we can describe the stack $\on{VinBun}_{SL_2}$ in the following way. It assigns to a scheme $S$ a pair of vector bundles $\CF_1, \CF_2$ of rank $2$ on $X\times S$ together with trivializations of their determinant line bundles $\on{det}\CF_1, \on{det}\CF_2$ and a morphism of coherent sheaves $\varphi\colon \CF_1 \ra \CF_2$ satisfying the following condition: For each geometric point $\ol{s} \ra S$ the morphism $\varphi|_{X\times\ol{s}}$ is nonzero. In these terms the morphism $\Upsilon\colon\on{VinBun}_{SL_2} \ra T^+_{\on{ad}}=\BA^1$ sends an $S$-point to the determinant $\on{det}\varphi \in \Gamma(\CO_{X\times S})=\BA^1(S)$.

\subsection{Local model for $\on{Bun}_B$} We now recall the construction of a certain
           {\em local model} for $\on{Bun}_B$ from \cite{bfgm}. This local model is called
           {\em open zastava}. We need some notations first.

           \sssec{The open Bruhat locus of $G$} Let us define {\em{the open Bruhat locus}}
           $G^{\on{Bruhat}} \subset G$
as the $B\times N_-$-orbit of $1 \in G$.

\begin{defn} The open zastava is defined as
\begin{equation*} \oZ:=\on{Maps}_{\on{gen}}(X, G/(B\times N_{-}) \supset G^{\on{Bruhat}}/(B\times N_{-})=\on{pt}).\end{equation*}
\end{defn}

\sssec{Connected components of the open zastava}
\label{Tannak for divisor}
For any positive integer $n \in \BN$ we denote the $n$-th symmetric power of the curve $X$ by
$X^{(n)}$.
Let $\theta\in\Lambda^\pos$. A point $D\in X^\theta$ is a collection of effective divisors
$D_{\la^{\lvee}} \in X^{(\langle \la^{\lvee}, \theta\rangle)}$ for $\la \in \La^{\vee+}$ such that for every
$\la^{\lvee}_1, \la^{\lvee}_2 \in \La^{\vee+}$ we have
$D_{\la^{\lvee}_1}+D_{\la^{\lvee}_2}=D_{\la^{\lvee}_1+\la^{\lvee}_2}$.
Since the derived subgroup of $G$ is assumed to be simply connected, for
$\theta=\sum\limits_{i\in I} n_i\al_i\in \Lambda^\pos$ we have
$X^\theta=\prod\limits_{i\in I}X^{(n_i)}.$

Let us describe the connected components of the scheme $\oZ$. Note that the natural morphism
$G/(N\times N_{-}) \ra G/\!\!/(N\times N_{-})\simeq\ol{T}$
induces the morphism $\pi\colon \oZ \ra \on{Maps}_{\on{gen}}(X,\ol{T}/T \supset \on{pt})$.
The scheme $\on{Maps}_{\on{gen}}(X,\ol{T}/T \supset T/T=\on{pt})$ is isomorphic to
$\coprod\limits_{\theta \in \La^\pos}X^{\theta}.$
For $\theta \in \Lambda^\pos$, set $\oZ^\theta:=\pi^{-1}(X^{\theta}).$

\begin{prop} \cite[Proposition~2.25]{bfg}. The decomposition
  $\oZ=\coprod\limits_{\theta \in \Lambda^\pos}\oZ^{\theta}$
coincides with the decomposition of $\oZ$ into connected components.
\end{prop}

\sssec{Factorization of zastava~\cite[\S2.3]{bfgm}}
Let us fix $\theta_1,\theta_2,\theta \in \Lambda^\pos$ such that
$\theta=\theta_1+\theta_2$.
We have an addition morphism
$\on{add}\colon X^{\theta_1}\times X^{\theta_2} \ra X^{\theta}$.
An open subset
$(X^{\theta_1}\times X^{\theta_2})_{\on{disj}} \subset
X^{\theta_1}\times X^{\theta_2}$
is formed by the pairs $(D_{\theta_1},D_{\theta_2})$
of disjoint divisors.
The {\em{factorization}} is a canonical isomorphism
\begin{equation*}
  \mathfrak{f}_{\theta_1,\theta_2}\colon(X^{\theta_1}\times X^{\theta_2})_{\on{disj}}
  \underset{X^{\theta}}\times\oZ^{\theta}
\iso (X^{\theta_1}\times X^{\theta_2})_{\on{disj}}
\underset{X^{\theta_1}\times X^{\theta_2}}\times(\oZ^{\theta_1}
\times\oZ^{\theta_2}).
\end{equation*}
For a point $x \in X$ and a cocharacter $\theta \in \Lambda^\pos$, let us denote by $\ofZ^\theta$ the fiber $\pi^{-1}(\theta \cdot x)$.
For two points $x,y \in X$, let us denote by
$\theta_1\cdot x +\theta_2\cdot y \in X^{\theta}$ the point
$\on{add}(\theta_1\cdot x, \theta_2 \cdot y)$.
It follows from factorization that for $x\neq y$ the fiber
$\pi^{-1}(\theta_1\cdot x +\theta_2\cdot y)$
can be canonically identified with
$\ofZ^{\theta_1}_x\times\ofZ^{\theta_2}_y$.

\sssec{Tannakian approach to the definition of $Z^{\theta}$} \label{zastava tannak}
For $\la^{\lvee} \in \La^{\vee}$, let $\BC^{\la^{\lvee}}$
be the one-dimensional representation of $T$ via character
$\la^{\lvee}\colon T \ra \BC^{\times}$. Fix a positive cocharacter $\theta \in \Lambda^\pos$.
We recall the Tannakian definition of the functor
$Z^{\theta}\colon {\bf{Sch}} \ra {\bf{Set}},~S \mapsto Z^{\theta}(S)$.
It associates to a scheme $S$

1) a $G$-bundle $\CF$ on $S\times X$,

2) a $T$-bundle $\CT$ on $S \times X$ of degree $-\theta$,

3) for every $\la^{\lvee} \in \La^{\vee+}$, a morphism of coherent sheaves
$\BC^{\la^{\lvee}}_{\CT}\xrightarrow{\eta_{\la^{\!\scriptscriptstyle\vee}}}
\CV^{\la^{\!\scriptscriptstyle\vee}}_{\CF}$
and a surjective morphism of vector bundles
$\CV^{\la^{\!\scriptscriptstyle\vee}}_{\CF} \xrightarrow{\zeta_{\la^{\!\scriptscriptstyle\vee}}} \CO_{S\times X}$
satisfying the following conditions:

a) for every $\la^{\!\scriptscriptstyle\vee} \in \Lambda^{\vee+}$ the composition
$(\zeta_{\la^{\!\scriptscriptstyle\vee}} \circ \eta_{\la^{\!\scriptscriptstyle\vee}})$
is an isomorphism generically,

b) the Pl\"ucker relations hold.

\begin{rem} The open subscheme
      $\oZ^{\theta} \subset Z^{\theta}$ consists of
      $(\CF,\CT,\eta_{\la^{\lvee}},\zeta_{\la^{\lvee}}) \in Z^{\theta}$
    such that $\eta_{\la^{\lvee}}$ are injective morphisms of vector bundles.
%\textup{(b)} In case the derived subgroup $[G,G]$ is not simply connected,
%    the above definition gives a scheme $Z^\theta_{\on{or}}$ containing
%    $\oZ^\theta$ as an open subscheme, not dense in general, cf.~\cite[\S7]{s0}.
%    For a definition of a closed subscheme $Z^\theta\subset Z^\theta_{\on{or}}$
%    containing $\oZ^\theta$ as a dense open subscheme, see~\cite[\S7]{s0}.
%    In what follows, we will refer to the relations cutting out $Z^\theta$ in
%  $Z^\theta_{\on{or}}$ as} Schieder relations.
\end{rem}

\sssec{Matrix description of open zastava for $G=SL_2$}\label{ex_SL_2_zast}
Let $G=SL_2$, $T=\{\on{diag}(t,t^{-1})\,|\, t \in \BC^\times\} \subset SL_2$ and identify $\BZ\iso\La$ via $n \mapsto (t \mapsto \on{diag}(t^{n},t^{-n}))$. We assume that $X=\BA^1$ (note that zastava spaces are defined even if $X$ is not projective). Fix $n \in \BZ_{\geqslant 0}=\Lambda^\pos$.
The variety $\oZ^n$ can be idenified with the space of matrices ${\bf{M}}=
\begin{pmatrix} A & B \\
C & D
\end{pmatrix} \in \on{Mat}_{2\times 2}[z]
$
such that $A$ is a monic polynomial of degree $n$, while the
degrees of $B$ and $C$ are strictly less than $n$, and $\on{det}{\bf{M}} =1$
(see e.g.~\cite[\S2(xii)]{bfn}).
The morphism $\pi\colon \oZ^n \ra \BA^{(n)}$ in these terms sends matrix ${\bf{M}}$ to the set of
roots of $A$ with multiplicities.

\sssec{}
Let us fix two
cocharacters $\theta_1, \theta_2 \in \La, \theta_1\geq \theta_2$ and a point $x \in X$.
\begin{defn}
$\ol{\mathfrak{Z}}{}^{\theta_1,\theta_2}$ is the moduli space of
the following data: it associates to a scheme $S$

1) a $G$-bundle $\CF$ on $S\times X$,

2) for every $\la^{\lvee} \in \La^{\vee+}$, a
morphism of {\em{sheaves}}
$\CO_{S\times X}(-\langle\la^{\lvee},\theta_1\rangle \cdot (S\times x))
\xrightarrow{\eta_{\la^{\!\scriptscriptstyle\vee}}} \CV^{\la^{\!\scriptscriptstyle\vee}}_{\CF}
$
and a morphism of {\em{sheaves}}
$
  \CV^{\la^{\!\scriptscriptstyle\vee}}_{\CF} \xrightarrow{\zeta_{\la^{\!\scriptscriptstyle\vee}}}
  \CO_{S\times X}(-\langle\la^{\lvee},\theta_2\rangle \cdot (S\times x))$
satisfying the following conditions:

a) for every $\la^{\!\scriptscriptstyle\vee} \in \Lambda^{\vee+}$ the composition $(\zeta_{\la^{\!\scriptscriptstyle\vee}} \circ \eta_{\la^{\!\scriptscriptstyle\vee}})$
coincides with the canonical embedding
$\CO_{S\times X}(-\langle\la^{\lvee},\theta_1\rangle \cdot (S\times x))\hookrightarrow
\CO_{S\times X}(-\langle\la^{\lvee},\theta_2\rangle \cdot (S\times x))$,

b) the Pl\"ucker relations hold.
\end{defn}

\begin{rem} \label{zastava as intersection} {\em Note that the scheme
    $\ol{\mathfrak{Z}}{}^{\theta_1,\theta_2}$
naturally identifies with $\ol{T}_{\theta_2}\cap \ol{S}_{\theta_1}.$
}
\end{rem}

\subsection{Local model for $\on{VinBun}_G$} We now recall the construction of certain {\em{local model}} for $\on{VinBun}_G$ of \cite[\S 6.1.6]{s2}. First we need some notations.

\sssec{The open Bruhat locus of the Vinberg semigroup} Let us define {\em{the open Bruhat locus}} $\on{Vin}^{\on{Bruhat}}_G \subset \on{Vin}_G$
as $B\times N_{-}$-orbit of the image of the section
$
\mathfrak{s}\colon T^+_{\on{ad}} \hookrightarrow \on{Vin}_G.
$

\begin{defn} \cite[\S 6.1.6]{s2} The local model is defined as
\begin{equation*} Y:=\on{Maps}_{\on{gen}}(X, \on{Vin}_G/(B\times N_{-}) \supset \on{Vin}^{\on{Bruhat}}_G/(B\times N_{-})).\end{equation*}
\end{defn}

\sssec{Defect free locus of $Y$}
Set \begin{equation*}
_{0}Y:=\on{Maps}_{\on{gen}}(X,~_{0}\!\!\on{Vin}_{G}/(B\times N_{-}) \supset
\on{Vin}^{\on{Bruhat}}_G/(B\times N_{-})).
\end{equation*}
The scheme $_{0}Y$ is a smooth open subscheme of $Y$.

\sssec{Local models $Y^{\theta}$}
\label{local Y}
Using Lemma~\ref{Factor by unipotent}, we get the natural morphism \begin{equation*}\on{Vin}_G/(N\times N_{-}) \ra \on{Vin}_G/\!\!/(N\times N_{-})\simeq\ol{T}\times T^+_{\on{ad}}\end{equation*} that gives us the morphism
$\on{Vin}_G/(B\times N_{-}) \ra (\ol{T}/T)\times T^+_{\on{ad}}$,
which, in turn,
induces the morphism \begin{equation*}\pi\colon Y \ra \on{Maps}_{\on{gen}}(X,(\ol{T}/T)\times T^{+}_{ad} \supset \on{pt}\times T^+_{\on{ad}})\simeq \on{Maps}_{\on{gen}}(X,\ol{T}/T \supset \on{pt}).\end{equation*}
The ind-scheme $\on{Maps}_{\on{gen}}(X,\ol{T}/T \supset T/T=\on{pt})$ is isomorphic to
$\coprod\limits_{\theta \in \Lambda^\pos}X^{\theta}.$
For $\theta \in \Lambda^\pos$ set $Y^\theta:=\pi^{-1}(X^{\theta})$.

For a positive cocharacter $\theta \in \Lambda^\pos$ set
$
_{0}Y^{\theta}:=~_{0}Y \cap Y^\theta:
$
an open dense smooth subscheme of $Y^{\theta}$.

\sssec{Tannakian approach to the definition of $Y^{\theta}$}
The scheme $Y^{\theta}$ can be described in the following way.
The corresponding functor of points assigns to a scheme $S$
the following data:

1) an $S$-point
$(\CF_1,\CF_2,\varphi_{\la^{\lvee}},\tau_{\mu^{\lvee}})
\in \on{VinBun}_G(S)$ of $\on{VinBun}_G$,

2) a $T$-bundle $\CT$ on $S\times X$ of degree $-\theta$,

3) for every $\la^{\lvee} \in \La^{\vee+}$, morphisms of vector bundles
\begin{equation*}
\eta_{\la^{\lvee}}\colon \BC^{\la^{\lvee}}_{\CT}
\hookrightarrow \CV^{\la^{\lvee}}_{\CF_1},~
\zeta_{\la^{\lvee}}\colon
\CV^{\la^{\lvee}}_{\CF_2} \twoheadrightarrow \CO_{S\times X},
\end{equation*}
satisfying the following conditions:

a) for every $\la^{\lvee} \in \La^{\vee+}$,
the composition
\begin{equation*}
\zeta_{\la^{\lvee}}\circ \eta_{\la^{\lvee}}\colon
\BC^{\la^{\lvee}}_{\CT} \ra \CO_{S\times X}
\end{equation*}
is an isomorphism generically.

b) The Pl\"ucker relations hold.

We have the morphism $\Upsilon^\theta\colon Y{}^\theta \ra T^+_{\on{ad}}$
that forgets the data of $\CF_1, \CF_2, \varphi_{\la^{\lvee}}, \eta_{\la^{\lvee}}, \zeta_{\la^{\lvee}}$.

\sssec{Matrix description of $Y^\theta$ for $G=SL_2$}
Let $G=SL_2, X=\BA^1$. We keep the notations of Section~\ref{ex_SL_2_zast}. Fix $n \in \BZ_{\geqslant 0}=\Lambda^\pos$.   
It can be deduced from~\cite[Section~5.3]{s1} that the variety $Y^n$ is isomorphic to the space of matrices ${\bf{M}}=
\begin{pmatrix} A & B \\
C & D
\end{pmatrix} \in
\on{Mat}_{2\times 2}[z]$
such that $A$ is a monic polynomial of degree $n$, while the
degrees of $B$ and $C$ are strictly less than $n$, and $\on{det}{\bf{M}}\in\BC\subset\BC[z]$.
The morphism $\Upsilon^n$ in these terms sends ${\bf{M}}$ to $\on{det}{\bf{M}}$.
The morphism $\pi$ sends ${\bf{M}}$ to the set of roots of $A$ with multiplicities.

\sssec{Compactified local model $\ol{Y}{}^\theta$} \label{compact Sch loc model}
For a positive cocharacter $\theta \in \Lambda^\pos$,
we define a certain compactification $\ol{Y}{}^\theta$ of
the local model $Y^{\theta}$. It associates to a scheme S

1) an $S$-point
$(\CF_1,\CF_2,\varphi_{\la^{\lvee}},\tau_{\mu^{\lvee}})
\in \on{VinBun}_G(S)$ of $\on{VinBun}_G$,

2) a $T$-bundle $\CT$ on $S\times X$ of degree $-\theta$,

3) for every $\la^{\lvee} \in \La^{\vee+}$, morphisms of coherent sheaves
\begin{equation*}
\eta_{\la^{\lvee}}\colon \BC^{\la^{\lvee}}_{\CT}
\hookrightarrow \CV^{\la^{\lvee}}_{\CF_1},~
\zeta_{\la^{\lvee}}\colon
\CV^{\la^{\lvee}}_{\CF_2} \ra \CO_{S\times X},
\end{equation*}
satisfying the following conditions:

a) for every $\la^{\lvee} \in \La^{\vee+}$,
the composition
\begin{equation*}
\zeta_{\la^{\lvee}}\circ \eta_{\la^{\lvee}}\colon
\BC^{\la^{\lvee}}_{\CT} \ra \CO_{S\times X}
\end{equation*}
is an isomorphism generically.

b) The Pl\"ucker relations hold.

\medskip

We have the morphism $\Upsilon^\theta\colon \ol{Y}{}^\theta \ra T^+_{\on{ad}}$
that forgets the data of $\CF_1, \CF_2, \varphi_{\la^{\lvee}}, \eta_{\la^{\lvee}}, \zeta_{\la^{\lvee}}$.
Let us denote by $\ol{Y}{}^{\theta,\on{princ}}$
the restriction of the multi-parameter family
$\ol{Y}{}^\theta \ra T^+_{\on{ad}}$
%$\Upsilon^{\theta}\colon
%\ol{Y}{}^{\theta} \ra T^+_{\on{ad}}$
to the ``principal'' line
\begin{equation*}\BA^1 \hookrightarrow T_{\on{ad}}^+,\ a
  \mapsto (a, \dots, a)~\on{in~coordinates}~\al^{\lvee}_i,\ i \in I.\end{equation*}
Let us denote by $_{0}\!\Upsilon^{\theta,\on{princ}}$ the restriction of the morphism $\Upsilon^\theta$ to $_{0}Y^\theta \subset \ol{Y}{}^\theta$.

\bigskip

Recall the definition of $X^\theta$ of \S\ref{Tannak for divisor}.
Let $\ol{\pi}_{\theta}\colon \ol{Y}{}^{\theta} \ra X^{\theta}$ be the morphism given by
\begin{equation*}
(\CF_1,\CF_2,\varphi_{\la^{\lvee}},\tau_{\mu^{\lvee}},
\eta_{\la^{\lvee}},\zeta_{\la^{\lvee}}) \mapsto
(\on{Div}(\zeta_{\la^{\lvee}}\circ \eta_{\la^{\lvee}})).
\end{equation*}
Here by $\on{Div}(\zeta_{\la^{\lvee}}\circ \eta_{\la^{\lvee}})$
we mean the divisor of zeros of
$\zeta_{\la^{\lvee}}\circ \eta_{\la^{\lvee}}$
considered as a global section of $(\BC^{\la^{\lvee}}_{\CT})^{*}$.
%Let us denote by $\Upsilon^{\theta}$ the composition
%$\ol{Y}{}^{\theta} \ra \on{VinBun}_G \ra T^+_{\on{ad}}.$
Note that the fiber of $\ol{\pi}_\theta$ over $1\in T_{\on{ad}}\subset T_{\on{ad}}^+$ is nothing but the
compactified zastava space $\ol{Z}{}^\theta$ of~\cite[\S7.2]{gai}.

For $x \in X$, set
$\ol{\mathfrak{Y}}{}^\theta:=\ol{\pi}_\theta^{-1}(\theta\cdot x)$,
$\ol{\mathfrak{Y}}{}^{\theta,\on{princ}}:= \ol{\mathfrak{Y}}{}^\theta
\cap \ol{Y}{}^{\theta,\on{princ}}$.

\begin{rem}[D.~Gaitsgory]
  {\em We have
    \begin{equation*}
      \bigsqcup_{\theta\in\Lambda^\pos}\ol{Y}{}^\theta=
\on{Maps}_{\on{gen}}\left(X, G\backslash\left(\ol{G/N_-}\times\on{Vin}_G\times\ol{N\backslash G}/T\right)/G
\supset G\backslash\left((G/N_-)\times\on{Vin}_G^{\on{Bruhat}}\times(B\backslash G)\right)/G\right).
  \end{equation*}}
\end{rem}

\sssec{Factorization of compactified local models}
\label{fibers for Schieder lm}
The compactified local models $\ol{Y}{}^{\theta}$ factorize in families over $T^+_{\on{ad}}$
in the sense of the following Lemma:

\begin{lem}
Let $\theta_1, \theta_2 \in \Lambda^\pos$ and let $\theta := \theta_1+\theta_2$.
Then the addition morphism
$(X^{\theta_1}\times X^{\theta_2})_{\on{disj}} \xrightarrow{\on{add}}
X^{\theta}$
induces a cartesian square
\begin{equation*}
\xymatrix{
(X^{\theta_1} \times X^{\theta_2})_{\on{disj}} \underset{X^{\theta_1}\times X^{\theta_2}}\times
(\ol{Y}{}^{\theta_1}\underset{T^+_{\on{ad}}}\times \ol{Y}{}^{\theta_2}) \ar[r] \ar[d] & \ol{Y}{}^{\theta} \ar[d]^{\ol{\pi}_\theta}\\
(X^{\theta_1}\times X^{\theta_2})_{\on{disj}} \ar[r]^{\on{add}} & X^{\theta}.
}
\end{equation*}
\end{lem}

\begin{proof}
The proof is analogous to the one of~\cite[Proposition~5.1.2]{s1}.
\end{proof}

The above Lemma implies that the fibers of the morphism
$\ol{Y}{}^{\theta}\ra T^+_{\on{ad}}$ are factorizable. That is for each
$t \in T^+_{\on{ad}}$ we have an isomorphism
\begin{equation*}
(X^{\theta_1}\times X^{\theta_2})_{\on{disj}} \underset{X^{\theta}}\times (\ol{Y}{}^{\theta}|_{t})
\iso (X^{\theta_1}\times X^{\theta_2})_{\on{disj}}
\underset{X^{\theta_1}\times X^{\theta_2}}\times(\ol{Y}{}^{\theta_1}|_{t}
\times \ol{Y}{}^{\theta_2}|_{t}).
\end{equation*}

\sssec{} \label{shifted compact Sch loc mod} Let us fix two cocharacters
$\theta_1, \theta_2 \in \La,\ \theta_1 \geq \theta_2$,
and a point $x \in X$.
\begin{defn}
  The space $\ol{\mathfrak{Y}}{}^{\theta_1,\theta_2}$ associates to a scheme $S$

1) an $S$-point $(\CF_1,\CF_2,\varphi_{\la^{\lvee}},\tau_{\mu^{\lvee}}) \in \on{VinBun}_G(S)$
of $\on{VinBun}_G$,

2) for every $\la^{\lvee} \in \La^{\vee+}$, a
morphism of vector bundles
$ \CO_{S\times X}(-\langle\la^{\lvee},\theta_1\rangle \cdot (S\times x))
\xrightarrow{\eta_{\la^{\!\scriptscriptstyle\vee}}}
\CV^{\la^{\!\scriptscriptstyle\vee}}_{\CF_1}
$
and a morphism of vector bundles
$
  \CV^{\la^{\!\scriptscriptstyle\vee}}_{\CF_2} \xrightarrow{\zeta_{\la^{\!\scriptscriptstyle\vee}}}
  \CO_{S\times X}(-\langle\la^{\lvee},\theta_2\rangle \cdot (S\times x))$
satisfying the following conditions.

a) For every $\la^{\!\scriptscriptstyle\vee} \in \Lambda^{\vee+}$ the composition $(\zeta_{\la^{\!\scriptscriptstyle\vee}} \circ \varphi_{\la^{\lvee}} \circ \eta_{\la^{\!\scriptscriptstyle\vee}})$
coincides with the canonical morphism
$\CO_{S\times X}(-\langle\la^{\lvee},\theta_1\rangle \cdot (S\times x))\hookrightarrow
\CO_{S\times X}(-\langle\la^{\lvee},\theta_2\rangle \cdot (S\times x))$.

b) The Pl\"ucker relations hold.
\end{defn}
Let
$\Upsilon^{\theta_1}_{\theta_2}\colon
\ol{\mathfrak{Y}}{}^{\theta_1,\theta_2} \ra T^+_{\on{ad}}$
be the morphism given by
\begin{equation*}
(\CF_1,\CF_2,\varphi_{\la^{\lvee}},\tau_{\mu^{\lvee}},
\eta_{\la^{\lvee}},\zeta_{\la^{\lvee}}) \mapsto
(\tau_{\mu^{\lvee}}).
\end{equation*}
We denote by $\ol{\mathfrak{Y}}{}^{\theta_1,\theta_2,\on{princ}}$
the restriction of the multi-parameter family
$\Upsilon^{\theta_1}_{\theta_2}\colon
\ol{\mathfrak{Y}}{}^{\theta_1,\theta_2} \ra T^+_{\on{ad}}$
to the ``principal'' line
\begin{equation*}\BA^1 \hookrightarrow T_{\on{ad}}^+,\ a \mapsto (a, \dots, a)~\on{in~coordinates}~\al^{\lvee}_i,\ i \in I.\end{equation*}
We will prove in Lemma~\ref{identification of local models} that the family
$\ol{\mathfrak{Y}}{}^{\theta_1,\theta_2} \ra T^+_{\on{ad}}$
is isomorphic to
$\ol{\mathfrak{Y}}{}^{\theta_1-\theta_2} \ra T^+_{\on{ad}}.$

\subsection{Schieder bialgebra}
\label{schieder bialgebra}
Let us recall the construction
of certain bialgebra from \cite{s3}.
We set \begin{equation*}\CA:=\bigoplus\limits_{\theta \in \Lambda^\pos}H_{c}^{\langle2\rho^{\lvee},\theta\rangle}(\ofZ^\theta).\end{equation*}

\begin{rem}{\em It follows from~\cite[\S6]{bfgm} that for $\theta \in \Lambda^\pos$ the vector space
$H_{c}^{\langle2\rho^{\lvee},\theta\rangle}(\ofZ^\theta)$
can be identified with the $\theta$-weight component
$U(\mathfrak{n}^{\vee})_{\theta}$. It follows that the algebra $\CA$ is
isomorphic to $U(\mathfrak{n}^{\vee})$ as a graded vector space.}
\end{rem}

\sssec{The two-parameter degeneration~\cite[\S6.2.2]{s3}} Take $\theta \in \Lambda^\pos$.
Let us consider the morphism
\begin{equation*}\pi_{\theta}\times {_{0}\!\Upsilon^{\theta,\on{princ}}}\colon _0Y^{\theta,\on{princ}} \ra
X^{\theta}\times \BA^1.\end{equation*}
Fix $x \in X$ and
positive cocharacters
$\theta_1,\theta_2 \in \Lambda^\pos$
such that $\theta_1+\theta_2=\theta$.
We define the family
$\varPi\colon Q \ra X \times \BA^1$
as the pullback of the family
$_{0}Y^{\theta,\on{princ}} \ra
X^{\theta}\times \BA^1$ along
\begin{equation*}X\times \BA^1 \hookrightarrow X^{\theta}\times \BA^1
,~(y,a) \mapsto
(\theta_1 \cdot x +\theta_2\cdot y,~a).\end{equation*}

Let $\varPi_{\on{comult}}$ denote the one-parameter family
over $X$ obtained by restricting
the family $\varPi$ above along the inclusion
$
X\times \{1\} \hookrightarrow X \times \BA^1.
$
Similarly, let $\varPi_{\on{mult}}$ denote the one-parameter
family over $\BA^1$ obtained by restricting
the family $\varPi$ above along the inclusion
$
\{x\}\times \BA^1 \hookrightarrow X \times \BA^1.
$

\begin{prop} \cite[Corollary 6.2.5]{s3}.

a) The one-parameter family $\varPi_{\on{comult}}$
is trivial over $X\setminus \{x\}$.
The special fiber is
\begin{equation*}
Q|_{(\theta\cdot x,1)}=\ofZ^{\theta}.
\end{equation*}
A general fiber is
\begin{equation*}
Q|_{(\theta_1\cdot x+\theta_2\cdot y,1)} = \ofZ^{\theta_1}\times \ofZ^{\theta_2}.
\end{equation*}

b) The one-parameter family $\varPi_{\on{mult}}$
is trivial over $\BA^1\setminus \{0\}$.
The special fiber is
\begin{equation*}
Q|_{(\theta\cdot x,0)}=\bigsqcup\limits_{\theta_1+\theta_2=\theta}\ofZ^{\theta_1}\times \ofZ^{\theta_2}.
\end{equation*}
A general fiber is $Q|_{(\theta\cdot x,1)}=\ofZ^{\theta}$.

\end{prop}

\sssec{Comultiplication~\cite[\S 2.11.1]{ffkm},~\cite[\S 6.3.2]{s3}}
\label{comultiplication}
Given positive coweights
$\theta,\theta_1,\theta_2\in \Lambda^\pos$ with
$\theta_1+\theta_2=\theta$ we define a morphism of vector spaces
\begin{equation*}
\Delta_{\theta_1,\theta_2}\colon\CA_{\theta} \ra \CA_{\theta_1}\otimes \CA_{\theta_2}
\end{equation*}
as the cospecialization morphism
\begin{equation*}
H_{c}^{\langle2\rho^{\lvee},\theta\rangle}(\ofZ^\theta)=
H_{c}^{\langle2\rho^{\lvee},\theta\rangle}(Q|_{(\theta\cdot x,1)}) \ra
H_{c}^{\langle2\rho^{\lvee},\theta\rangle}(Q|_{(\theta_1\cdot x+\theta_2\cdot y,1)})=
H_{c}^{\langle2\rho^{\lvee},\theta_1\rangle}(\ofZ^\theta_1)\otimes
H_{c}^{\langle2\rho^{\lvee},\theta_2\rangle}(\ofZ^\theta_2)
\end{equation*}
corresponding to the one-parameter
degeneration $\varPi_{\on{comult}}$.
Summing over all such triples $(\theta,\theta_1,\theta_2)$
we obtain a comultiplication morphism
\begin{equation*}
\Delta\colon\CA \ra \CA\otimes \CA.
\end{equation*}

\sssec{Multiplication~\cite[\S 6.3.3]{s3}}
\label{multiplication}
Given a positive coweight
$\theta\in \Lambda^\pos$ we define a morphism of vector spaces
\begin{equation*}
{\bf{m}}_{\theta}:=\bigoplus\limits_{\theta_1+\theta_2=\theta}{\bf{m}}_{\theta_1, \theta_2}\colon\bigoplus\limits_{\theta_1+\theta_2=\theta}
\CA_{\theta_1}\otimes \CA_{\theta_2} \ra \CA_{\theta}
\end{equation*}
as the cospecialization morphism
\begin{equation*}
\bigoplus\limits_{\theta_1+\theta_2=\theta}
H_{c}^{\langle2\rho^{\lvee},\theta_1\rangle}(\ofZ^\theta_1)\otimes
H_{c}^{\langle2\rho^{\lvee},\theta_2\rangle}(\ofZ^\theta_2)=
H_{c}^{\langle2\rho^{\lvee},\theta\rangle}(Q|_{(\theta x,0)}) \ra
H_{c}^{\langle2\rho^{\lvee},\theta\rangle}(Q|_{(\theta x,1)})=
H_{c}^{\langle2\rho^{\lvee},\theta\rangle}(\ofZ^\theta)
\end{equation*}
corresponding to the one-parameter
degeneration $\varPi_{\on{mult}}$.
Summing over all $\theta \in \Lambda^\pos$ we obtain a multiplication morphism
\begin{equation*}
{\bf{m}}\colon\CA\otimes \CA \ra \CA.
\end{equation*}

\section{Drinfeld-Gaitsgory-Vinberg interpolation Grassmannian}
\label{three}

\subsection{A degeneration of the Beilinson-Drinfeld Grassmannian}

\sssec{The example $G=\on{SL}_2$} Let us start from the case $G=\on{SL}_2$. Let us fix a point $x \in X$. \begin{defn}
  \label{BD degeneration for SL_2} For $G=\on{SL}_2$, an $S$-point of $\on{VinGr}_{G,x}$
  consists of the following data:

1) two vector bundles $\CV_1, \CV_2$ of rank two on $S \times X$, together with trivializations of their determinant line bundles $\on{det}\CV_1, \on{det}\CV_2$ and a morphism of coherent sheaves $\varphi\colon \CV_1 \ra \CV_2$,

2) the rational morphisms $\eta\colon \CO \ra \CV_1,\ \zeta\colon \CV_2 \ra \CO$ regular
on $S\times (X\setminus \{x\})$ such that
the composition \begin{equation*}(\zeta \circ \varphi \circ \eta)|_{S\times (X\setminus \{x\})}
\colon \CO|_{S\times (X\setminus \{x\})}
\ra \CO|_{S\times (X\setminus \{x\})}
\end{equation*}
is the identity morphism.
\end{defn}
We have a morphism
$\Upsilon\colon \on{VinGr}_{G,x} \ra \BA^1$ which sends an $S$-point above to the determinant \begin{equation*}\on{det}\varphi \in \on{\Gamma}(\CO_{X\times S})=\on{\Gamma}(\CO_{S})=\BA^1(S).\end{equation*}

It follows from Proposition~\ref{repres of BD} below that the functor $\on{VinGr}_{G,x}$ is represented by an ind-scheme (ind-projective over $\BA^1$) denoted by the same symbol.

Let $(\on{VinGr}_{G,x})_a$ denote the fiber of the morphism $\Upsilon$ over $a \in \BA^1$.

\begin{prop} \label{triv for SL_2}
The fiber of the morphism $\Upsilon$ over the point $1 \in \BA^1$ is  isomorphic to the affine Grassmannian $\on{Gr}_{G}$.
\end{prop}

\begin{proof}
Let us construct a morphism of functors
$\Theta\colon (\on{VinGr}_{G,x})_{1} \ra \Gr_{G}.$
Note that the vector bundles $\CV_1,\ \CV_2$ are identified via $\varphi$.
From condition~\ref{BD degeneration for SL_2}.2) it follows that the rational morphisms $\eta, \zeta$ define transversal $N$ and $N_{-}$ structures in the vector bundle \begin{equation*}\CV:=\CV_1|_{S\times (X\setminus \{x\})} \simeq \CV_2|_{S\times (X\setminus \{x\})}.\end{equation*} Thus we get a trivialization $\sigma$ of the vector bundle $\CV|_{S\times (X\setminus \{x\})}$. Set \begin{equation*} \Theta(\CV_1,\CV_2,\varphi,\zeta,\eta):= (\CV,\sigma) \in \Gr_{G}(S).\end{equation*}
Let us construct the inverse morphism
$
\Theta^{-1}\colon \Gr_{G} \ra (\on{VinGr}_{G,x})_{1}.
$
Consider a point $(\CV,\sigma) \in \Gr_{G}(S)$.
Set
$
\CV_1:=\CV=:\CV_2,~\varphi:= \on{Id}\colon \CV_1=\CV \ra \CV=\CV_2.
$
Let us define $\eta\colon \CO \ra \CV=\CV_1$ as the composition of
$
\eta_0\colon \CO \hookrightarrow \CO \oplus \CO,~s \mapsto (s,0)
$ and $\sigma$.
Let us define $\zeta\colon \CV_2=\CV \ra \CO$ as the composition of $\sigma^{-1}$ and
$\zeta_0\colon \CO \oplus \CO \twoheadrightarrow \CO ,~(s_1,s_2) \mapsto s_1. $
Finally, we set
$\Theta^{-1}(\CV,\sigma):= (\CV_1,\CV_2,\varphi,\zeta,\eta).
$
\end{proof}

\sssec{The degeneration of $\on{Gr}_G$ for arbitrary $G$ via Tannakian approach} \label{degen of the affine Gr}
We now define a degeneration $\on{VinGr}_{G,x}$ of the affine Grassmannian $\Gr_{G}$ for arbitrary reductive group $G$. An $S$-point of $\on{VinGr}_{G,x}$ consists of the following data:

1) an $S$-point $(\CF_1, \CF_2, \varphi_{\la^{\lvee}}, \tau_{\mu^{\lvee}}) \in \on{VinBun}_G(S)$,

2) for every $\la^{\!\scriptscriptstyle\vee} \in \Lambda^{\vee+}$, the rational morphisms \begin{equation*}
\eta_{\la^{\!\scriptscriptstyle\vee}}\colon  \CO_{S\times X} \ra \CV^{\la^{\!\scriptscriptstyle\vee}}_{\CF_1},
~\zeta_{\la^{\!\scriptscriptstyle\vee}}\colon  \CV^{\la^{\lvee}}_{\CF_2} \ra \CO_{S\times X}\end{equation*}
regular on $U:=S\times (X \setminus \{x\})$,\\
satisfying the following conditions.

a) For every $\la^{\!\scriptscriptstyle\vee} \in \Lambda^{\vee+}$, the composition \begin{equation*}(\zeta_{\la^{\!\scriptscriptstyle\vee}} \circ \varphi_{\la^{\!\scriptscriptstyle\vee}} \circ \eta_{\la^{\!\scriptscriptstyle\vee}})|_{U}\end{equation*}
is the identity morphism.

b) For every $\lambda^{\!\scriptscriptstyle\vee}, \mu^{\!\scriptscriptstyle\vee} \in \Lambda^{\vee+}$ let $\on{pr}_{\la^{\!\scriptscriptstyle\vee},\mu^{\!\scriptscriptstyle\vee}}\colon  V^{\la^{\!\scriptscriptstyle\vee}} \otimes V^{\mu^{\!\scriptscriptstyle\vee}} \twoheadrightarrow V^{\la^{\!\scriptscriptstyle\vee}+\mu^{\!\scriptscriptstyle\vee}}$ be the projection morphism of \S\ref{Tannakian def of BD}. We have the corresponding morphisms \begin{equation*}\on{pr}^{\CF_1}_{\la^{\!\scriptscriptstyle\vee},\mu^{\!\scriptscriptstyle\vee}}\colon  \CV^{\la^{\!\scriptscriptstyle\vee}}_{\CF_1} \otimes \CV^{\mu^{\!\scriptscriptstyle\vee}}_{\CF_1} \ra \CV^{\la^{\!\scriptscriptstyle\vee}+\mu^{\!\scriptscriptstyle\vee}}_{\CF_1},~\on{pr}^{\CF_2}_{\la^{\!\scriptscriptstyle\vee},\mu^{\!\scriptscriptstyle\vee}}\colon  \CV^{\la^{\!\scriptscriptstyle\vee}}_{\CF_2} \otimes \CV^{\mu^{\!\scriptscriptstyle\vee}}_{\CF_2} \ra \CV^{\la^{\!\scriptscriptstyle\vee}+\mu^{\!\scriptscriptstyle\vee}}_{\CF_2}.\end{equation*}
Then the following diagrams are commutative:

\begin{equation*}\begin{CD}
\CO_{U}\otimes \CO_{U}@>\on{Id}\otimes \on{Id}>> \CO_{U}\\
@VV\eta_{\la^{\!\scriptscriptstyle\vee}} \otimes \eta_{\mu^{\!\scriptscriptstyle\vee}}V @VV\eta_{\la^{\!\scriptscriptstyle\vee}+\mu^{\!\scriptscriptstyle\vee}}V\\
(\CV^{\la^{\!\scriptscriptstyle\vee}}_{\CF_1} \otimes \CV^{\mu^{\!\scriptscriptstyle\vee}}_{\CF_1})|_U @>\on{pr}^{\CF_1}_{\la^{\!\scriptscriptstyle\vee},\mu^{\!\scriptscriptstyle\vee}}>> (\CV^{\la^{\!\scriptscriptstyle\vee}+\mu^{\!\scriptscriptstyle\vee}}_{\CF_1})|_U,
\end{CD}\end{equation*}

\bigskip

\begin{equation*}\begin{CD}
(\CV^{\la^{\!\scriptscriptstyle\vee}}_{\CF_2} \otimes \CV^{\mu^{\!\scriptscriptstyle\vee}}_{\CF_2})|_U @>\on{pr}^{\CF_2}_{\la^{\!\scriptscriptstyle\vee},\mu^{\!\scriptscriptstyle\vee}}>> (\CV^{\la^{\!\scriptscriptstyle\vee}+\mu^{\!\scriptscriptstyle\vee}}_{\CF_2})|_U\\
@VV\zeta_{\la^{\!\scriptscriptstyle\vee}} \otimes \zeta_{\mu^{\!\scriptscriptstyle\vee}}V @VV\zeta_{\la^{\!\scriptscriptstyle\vee}+\mu^{\!\scriptscriptstyle\vee}}V\\
\CO_{U}\otimes \CO_{U} @>\on{Id}\otimes \on{Id}>> \CO_{U}.\end{CD} \end{equation*}

c) Given a morphism $\on{pr}\colon  V^{\la^{\!\scriptscriptstyle\vee}} \otimes
V^{\mu^{\!\scriptscriptstyle\vee}} \ra V^{\nu^{\!\scriptscriptstyle\vee}}$ for
$\la^{\!\scriptscriptstyle\vee},\mu^{\!\scriptscriptstyle\vee},\nu^{\!\scriptscriptstyle\vee} \in \La^{+},\
\nu^{\!\scriptscriptstyle\vee} < \la^{\!\scriptscriptstyle\vee}+\mu^{\!\scriptscriptstyle\vee}$, we have
\begin{equation*}\on{pr}^{\CF_1} \circ (\eta_{\la^{\!\scriptscriptstyle\vee}} \otimes
  \eta_{\mu^{\!\scriptscriptstyle\vee}})=0,~(\zeta_{\la^{\!\scriptscriptstyle\vee}} \otimes
  \zeta_{\mu^{\!\scriptscriptstyle\vee}}) \circ \on{pr}^{\CF_2}=0.
\end{equation*}

d) For $\la^{\!\scriptscriptstyle\vee}=0$, we have $\zeta_{\la^{\!\scriptscriptstyle\vee}} = \on{Id}$ and $\eta_{\la^{\!\scriptscriptstyle\vee}} = \on{Id}$.

We have a projection $\on{VinGr}_{G,x} \ra \on{VinBun}_G$
that forgets the data $\eta_{\la^{\lvee}},\zeta_{\la^{\lvee}}$.

\sssec{The degeneration ${\on{VinGr}}_{G,x}$ via mapping stacks}
\label{vingr via mapp stack}
Recall the family
$\Upsilon\colon \on{Vin}_G \ra T^+_{\on{ad}}$ of
\S\ref{Morphism v for Vin}
considered as the scheme over $T^+_{\on{ad}}$.
Fix a point $x \in X$.
Denote by $\on{VinGr}'_{G,x}$ the fibre product
\begin{equation*}
\on{Maps}_{T^+_{\on{ad}}}(X\times T^+_{\on{ad}},\on{Vin}_{G}/(G\times G))
\underset{\on{Maps}_{T^{+}_{\on{ad}}}((X\setminus \{x\})\times T^{+}_{\on{ad}},\on{Vin}_G/(G\times G))}\times T^{+}_{\on{ad}},
\end{equation*}
where the morphism
$T^{+}_{\on{ad}} \ra
\on{Maps}_{T^{+}_{\on{ad}}}((X\setminus \{x\})\times T^{+}_{\on{ad}},\on{Vin}_G/(G\times G))
$
is the composition of the isomorphism
$T^+_{\on{ad}}\iso \on{Maps}_{T^+_{\on{ad}}}((X\setminus \{x\})\times T^{+}_{\on{ad}},T^+_{\on{ad}})$
with the morphism \begin{equation*}\on{Maps}_{T^+_{\on{ad}}}((X\setminus \{x\})\times T^{+}_{\on{ad}},T^+_{\on{ad}}) \ra
\on{Maps}_{T^+_{\on{ad}}}((X\setminus \{x\})\times T^{+}_{\on{ad}},\on{Vin}_G/(G\times G))\end{equation*}
induced by the morphism
$T^+_{\on{ad}} \xrightarrow{\mathfrak{s}} \on{Vin}_G \ra \on{Vin}_G/(G\times G)$.

Let us denote by $\Upsilon'\colon\on{VinGr}'_{G,x} \ra T^+_{\on{ad}}$
the projection to the second factor.

\begin{prop} \label{isom between descr of vingr}
The families $\Upsilon\colon \on{VinGr}_{G,x} \ra T^{+}_{\on{ad}}$ and $\Upsilon'\colon\on{VinGr}'_{G,x} \ra T^+_{\on{ad}}$ are isomorphic.
\end{prop}

\begin{proof}
Follows from Proposition~\ref{ident of two def} below.
\end{proof}

\sssec{Definition of $\on{VinGr}_{G,X^n}$ via Tannakian approach}
\label{VinGr via Tannak}
We now define a degeneration $\on{VinGr}_{G,X^n}$ of the Beilinson-Drinfeld Grassmannian $\Gr_{G,X^n}$ for arbitrary reductive group $G$. An $S$-point of $\on{VinGr}_{G,X^n}$ consists of the following data:

1) a collection of $S$-points $\ul{x}=(x_1,\dots,x_n) \in X^{n}(S)$ of the curve $X$,

2) an $S$-point $(\CF_1, \CF_2, \varphi_{\la^{\lvee}}, \tau_{\mu^{\lvee}}) \in \on{VinBun}_G(S)$,

3) for every $\la^{\!\scriptscriptstyle\vee} \in \Lambda^{\vee+}$, the rational morphisms \begin{equation*}\eta_{\la^{\!\scriptscriptstyle\vee}}\colon  \CO_{S\times X} \ra \CV^{\la^{\!\scriptscriptstyle\vee}}_{\CF_1},~
\zeta_{\la^{\!\scriptscriptstyle\vee}}\colon  \CV^{\la^{\lvee}}_{\CF_2} \ra \CO_{S\times X}\end{equation*}
regular on $(S\times X)\setminus\{\varGamma_{x_1}\cup\dots\cup\varGamma_{x_n}\}$,\\
satisfying the same conditions as in the definition of $\on{VinGr}_{G,x}$ of~\S\ref{degen of the affine Gr} with $U$ replaced by $(S\times X)\setminus\{\varGamma_{x_1}\cup\dots\cup\varGamma_{x_n}\}$.

We have a projection $\on{VinGr}_{G,X^n} \ra \on{VinBun}_G$
that forgets the data of $\ul{x},\eta_{\la^{\lvee}},\zeta_{\la^{\lvee}}$.
We have a projection
$
\pi^{\on{Vin}}_n\colon \on{VinGr}_{G,X^n} \ra X^n
$
that forgets the data of $\CF_1, \CF_2, \varphi_{\la^{\lvee}}, \tau_{\la^{\lvee}}, \zeta_{\la^{\lvee}}, \eta_{\la^{\lvee}}$.

\sssec{The degeneration morphism $\Upsilon\colon \on{VinGr}_{G,X^n} \ra T^{+}_{\on{ad}}$} \label{Morphism v for VinGr} The morphism
$\Upsilon\colon \on{VinGr}_{G,X^n} \ra T^{+}_{\on{ad}}$
is defined as the composition of the morphisms
\begin{equation*}\on{VinGr}_{G,X^n} \ra \on{VinBun}_{G} \ra T^{+}_{\on{ad}}.\end{equation*}

\sssec{Second definition of $\on{VinGr}_{G,X^n}$} (cf. Remark~\ref{second def of GrBD}).
\label{second def VinGrBD}
Let us fix $n \in \BN$.
We denote by $\on{VinBun}_G(\mathfrak{U}_n)$ the following stack over $T^+_{\on{ad}} \times X^n$:
it associates to a scheme $S$

1) a collection of $S$-points $\ul{x}=(x_1,\dots,x_n) \in X^{n}(S)$ of the curve $X$,

2) an element of $Maps_{T^+_{\on{ad}}}(T^+_{\on{ad}} \times ((S\times X)\setminus \{\varGamma_{x_{1}} \cup \dots \cup \varGamma_{x_{n}}\}),
\on{Vin}_G/(G\times G))$.

We have a restriction morphism
$X^{n}\times \on{Maps}_{T^+_{\on{ad}}}(T^{+}_{\on{ad}}
\times X, \on{Vin}_{G}/(G \times G)) \ra \on{VinBun}_G(\mathfrak{U}_n)$.
We also have a morphism $T^{+}_{\on{ad}} \times X^{n} \ra \on{VinBun}_G(\mathfrak{U}_n)$
induced by the morphism
$T^+_{\on{ad}} \xrightarrow{\mathfrak{s}} \on{Vin}_G \ra \on{Vin}_{G}/(G \times G)$.
We define $\on{VinGr}'_{G,X^{n}}$ as the following family over the scheme $T^+_{\on{ad}}$:
\begin{equation*}
(X^n \times \on{Maps}_{T^{+}_{\on{ad}}}(T^{+}_{\on{ad}}\times X,\on{Vin}_G/(G\times G))\underset{\on{VinBun}_G(\mathfrak{U}_n)}\times (T^{+}_{\on{ad}}\times X^n).
\end{equation*}

The morphism $\on{VinGr}'_{G,X^n} \ra T^+_{\on{ad}}$ is denoted by $\Upsilon'$.

\begin{prop} \label{ident of two def}
The families $\Upsilon\colon \on{VinGr}_{G,X^n} \ra T^{+}_{\on{ad}}$ and $\Upsilon'\colon\on{VinGr}'_{G,X^n} \ra T^+_{\on{ad}}$ are isomorphic.
\end{prop}

\begin{proof}
Let us construct a morphism
$\mho\colon\on{VinGr}'_{G,X^n} \ra \on{VinGr}_{G,X^n}$. Let $f\colon S \ra T^+_{\on{ad}}$ be a scheme over $T^+_{\on{ad}}$. An $S$-point of the family
$\on{VinGr}'_{G,X^n}$ over $T^+_{\on{ad}}$ is the following data:

1) a collection of $S$-points $\ul{x} = (x_1,\dots,x_n) \in X^{n}(S)$
of the curve $X$,

2) two $G$-bundles $\CF_1, \CF_2$ on $S\times X$,

3) a $G\times G$-equivariant morphism
$\digamma\colon \CF_1\times \CF_2 \ra \on{Vin}_G$
of schemes over $T^+_{\on{ad}}$,

4) trivializations $\sigma_1, \sigma_2$ of $\CF_1, \CF_2$ on $U:=(S\times X) \setminus \{\varGamma_{x_1}\cup\dots\cup\varGamma_{x_n}\}$
such that the morphism $\digamma$ is identified with the morphism
$\digamma^{\on{triv}}\colon G\times U \times G \ra \on{Vin}_G$ given by
$(g_1,p,x,g_2) \mapsto g_1 \cdot \mathfrak{s}(f(p)) \cdot {g_2}^{-1},$
where $\mathfrak{s}$ is the section of
the morphism $\Upsilon\colon \on{Vin}_G \ra T^+_{\on{ad}}$
defined in \S\ref{section},~\ref{Upsilon and s via Tannak}.

From the Tannakian description of $\on{Vin}_G$ in~\S\ref{Tannakian def of Vin} and the
description of $T^+_{\on{ad}}$ in
\S\ref{Upsilon and s via Tannak} it follows
that the $G\times G$-equivariant morphism $\digamma \colon \CF_1 \times \CF_2 \ra \on{Vin}_G$ of schemes over $T^+_{\on{ad}}$
is the same as the collection of $G\times G$-equivariant morphisms $\digamma_{\la^{\lvee}} \colon \CF_1 \times \CF_2 \ra \on{End}(V^{\la^{\lvee}})$ for each $\la^{\vee} \in \La^{\vee+}$
and (fixed) $G\times G$-equivariant morphisms $\widetilde{\tau}_{\mu^{\lvee}}\colon \CF_1\times \CF_2 \ra \BC$
satisfying the same conditions as in the definition of $\on{Vin}_G$
of \S\ref{Tannakian def of Vin}.
Each $G\times G$-equivariant morphism $\digamma_{\la^{\lvee}} \colon \CF_1 \times \CF_2 \ra \on{End}(V^{\la^{\lvee}})$ induces the morphism
$\varphi_{\la^{\lvee}}\colon
\CV^{\la^{\lvee}}_{\CF_1} \ra \CV^{\la^{\lvee}}_{\CF_2}$. Note also that the morphisms $\widetilde{\tau}_{\mu^{\lvee}}$ induce the morphisms
$\tau_{\mu^{\lvee}}\colon\CO_{S\times X}=\BC_{\CF_1} \ra \BC_{\CF_2}=\CO_{S\times X}$.
The trivializations $\sigma_1, \sigma_2$ induce the trivializations of the
vector bundles $\CV^{\la^{\lvee}}_{\CF_1}, \CV^{\la^{\lvee}}_{\CF_2}$
such that the morphisms $\varphi_{\la^{\lvee}}|_{U}
\colon {\CV^{\la^{\lvee}}_{\CF_1}}|_{U} \ra {\CV^{\la^{\lvee}}_{\CF_2}}|_{U}$
get identified with
$
g_{\la^{\lvee}}\colon V^{\la^{\lvee}}\otimes \CO_{U}
\ra V^{\la^{\lvee}}\otimes \CO_{U},
$
where $g_{\la^{\lvee}}$ is the tensor product of the morphisms 
\begin{equation*}
\on{Id}\otimes (\tau_{\mu^{\vee}}|_{U})\colon (V^{\la^{\vee}})_{\la^{\vee}-\mu^{\vee}} \otimes \CO_U \ra (V^{\la^{\vee}})_{\la^{\vee}-\mu^{\vee}} \otimes \CO_U
\end{equation*} (cf.\ \S\ref{Upsilon and s via Tannak}).
The standard $N,N_{-}$-structures in the trivial $G$-bundle on $S\times X$  define via
$\sigma_1,\sigma_2$ the $N$ and $N_{-}$-structures
in ${\CF_1}|_{U},~{\CF_2}|_{U}$ respectively. They give us the collection of rational morphisms
$\CO_{S\times X} \xrightarrow{\eta_{\la^{\!\scriptscriptstyle\vee}}} {\CV^{\la^{\!\scriptscriptstyle\vee}}_{\CF_1}},~{\CV^{\la^{\!\scriptscriptstyle\vee}}_{\CF_2}}\xrightarrow{\zeta_{\la^{\!\scriptscriptstyle\vee}}}
\CO_{S\times X}$ (cf. proof of Proposition~\ref{tannak via ordinary def of Gr}).
Set
\begin{equation*}
\mho(\ul{x},\CF_1,\CF_2,\digamma,\sigma_1,\sigma_2):=(\ul{x},\CF_1,\CF_2,\varphi_{\la^{\lvee}},\tau_{\mu^{\lvee}},\zeta_{\la^{\!\scriptscriptstyle\vee}},\eta_{\la^{\!\scriptscriptstyle\vee}}).
\end{equation*}
The inverse morphism $\mho^{-1}$ can be constructed using the same arguments.
\end{proof}

Recall the action $T \curvearrowright T^+_{\on{ad}}$ of \S\ref{Morphism v for Vin}.
Note that the torus $T$ acts on the space $\on{VinGr}'_{G,X^n}=\on{VinGr}_{G,X^n}$
via the actions $T \curvearrowright T^+_{\on{ad}}, \on{Vin}_G$ of \S\ref{Morphism v for Vin}.
It is easy to see that the morphism $\Upsilon$ is $T$-equivariant.

\begin{lem} \label{restriction for VinGrBD} The family
$\Upsilon^{-1}(T_{\on{ad}})$
is isomorphic to $\on{Gr}_{G,X^n} \times T_{\on{ad}}$.
\end{lem}

\begin{proof} The morphism $\Upsilon$ is $T$-equivariant
so it is enough to identify the fiber $\Upsilon^{-1}(1)$
with $\on{Gr}_{G,X^n}$. The argument is analogous to the proof of Proposition~\ref{triv for SL_2}.
\end{proof}

Let $\Delta_{X} \hookrightarrow X^2$ be
the closed embedding of the diagonal. Let
$U \hookrightarrow X^2$ be
the open embedding of the complement to the diagonal.
\begin{prop} \label{factor for VinGr}
a) The restriction of the family
$\pi^{\on{Vin}}_2\colon\on{VinGr}_{G,X^2} \ra X^2$
to the closed subvariety $\Delta_{X} \hookrightarrow X^2$
is isomorphic to $\on{VinGr}_{G,X}$.

b) The restriction of the family
$\pi^{\on{Vin}}_2\colon\on{VinGr}_{G,X^2} \ra X^2$
to the open subvariety $U \hookrightarrow X^2$
is isomorphic to $(\on{VinGr}_{G,X}\underset{T^{+}_{\on{ad}}}\times \on{VinGr}_{G,X})|_U$.
\end{prop}

\begin{proof}
Part a) is clear, let us prove b). The proof is the same as the one of~\cite[Proposition~3.1.13]{zh}.
Take $\ul{x}=(x_1,x_2) \in U(S)\subset X^{2}(S)$.
We define a morphism $(\on{VinGr}_{G,X^2})|_{U}(S) \ra \on{VinGr}_{G,X}$
by sending $(\ul{x},\CF_1,\CF_2,\digamma,\sigma_1,\sigma_2)$
to $(x_1,(\CF_1)_{x_1},(\CF_2)_{x_1},\digamma_{x_1},(\sigma_1)_{x_1},(\sigma_2)_{x_1})$
where $(\CF_k)_{x_1}$ ($k \in \{1,2\}$) is obtained by gluing $(\CF_k)|_{((S\times X) \setminus \{\varGamma_{x_2}\})}$ and $(\CF^{\on{triv}})|_{((S\times X) \setminus \{\varGamma_{x_1}\})}$ via $\sigma_k$
and therefore is equipped with a trivialization
$(\sigma_k)_{x_1}$. The morphism $\digamma_{x_1}$ is obtained by gluing $\digamma|_{((S\times X) \setminus \{\varGamma_{x_2}\})}$ and
$\digamma^{\on{triv}}|_{((S\times X) \setminus \{\varGamma_{x_1}\})}$. Similarly, we have another morphism $(\on{VinGr}_{G,X^2})|_{U}(S) \ra \on{VinGr}_{G,X}$. Together, they define a morphism
$(\on{VinGr}_{G,X^2})|_{U} \ra (\on{VinGr}_{G,X} \underset{T^+_{\on{ad}}}\times \on{VinGr}_{G,X})|_U.$

Conversely if we have
$(x_1,(\CF_1)_{x_1},(\CF_2)_{x_1},\digamma_{x_1},(\sigma_1)_{x_1},(\sigma_2)_{x_1}) \in \on{VinGr}_{G,X}(S)$
and
$(x_2,(\CF_1)_{x_2},(\CF_2)_{x_2},\digamma_{x_2},(\sigma_1)_{x_2},(\sigma_2)_{x_2}) \in \on{VinGr}_{G,X}(S)$ such that $(x_1,x_2) \in U(S)$, we can construct $\CF_k$ ($k \in \{1,2\}$) by gluing
${(\CF_k)_{x_1}}|_{((S\times X) \setminus \{\varGamma_{x_2}\})}$
and ${(\CF_k)_{x_2}}|_{((S\times X) \setminus \{\varGamma_{x_1}\})}$
by $(\sigma_k)^{-1}_{x_2}(\sigma_k)_{x_1}$, which by definition is equipped with a trivialization $\sigma_k$ on $(S\times X) \setminus \{\varGamma_{x_1} \cup \varGamma_{x_2}\}.$ The morphism $\digamma$ is obtained by gluing $\digamma_{x_1}$ and $\digamma_{x_2}$.
\end{proof}

\sssec{Embedding of $\on{VinGr}_{G,X^n}$ into the product $(\Gr_{G,X^n}\underset{X^{n}}\times \Gr_{G,X^n})\times T^+_{\on{ad}}$} \label{embedding} Let us construct a closed embedding \begin{equation*}\vartheta\colon \on{VinGr}_{G,X^n} \hookrightarrow \Gr_{G,X^n}\underset{X^n}\times \Gr_{G,X^n}\times T^+_{\on{ad}}. \end{equation*}
It sends an $S$-point
\begin{equation*}(\CF_1,\CF_2,\varphi_{\la^{\lvee}},\tau_{\mu^{\lvee}},\zeta_{\la^{\lvee}},\eta_{\la^{\lvee}})
  \in \on{VinGr}_{G,X^n}(S)\colon \CO_{S\times X} \xrightarrow{\eta_{\la^{\lvee}}} \CV^{\la^{\lvee}}_{\CF_1}
  \xrightarrow{\varphi_{\la^{\lvee}}} \CV^{\la^{\lvee}}_{\CF_2} \xrightarrow{\zeta_{\la^{\lvee}}} \CO_{S\times X}
\end{equation*}
to the point \begin{equation*}(\CF_1,\eta_{\la^{\lvee}},\zeta_{\la^{\lvee}}\circ \varphi_{\la^{\lvee}})\times(\CF_2,\varphi_{\la^{\lvee}}\circ \eta_{\la^{\lvee}},\zeta_{\la^{\lvee}})\times\tau_{\mu^{\lvee}}.\end{equation*}

\begin{lem}\label{embedding into product} The morphism $\vartheta$ is a closed embedding (cf. \cite[Lemma~5.2.7]{s1}).
\end{lem}

\begin{proof} Let us show that $\vartheta$ induces an injective map on the level of $S$-points.
Let us identify $T^+_{\on{ad}}$ with ${\on{Maps}}(X,T^+_{\on{ad}})$. We already used this identification in the definition of the morphism $\vartheta$ because the
morphisms $\tau_{\mu^{\lvee}}\colon \CO_{S\times X} \ra \CO_{S\times X}$
define an $S$-point of ${\on{Maps}}(X,T^+_{\on{ad}})\simeq T^+_{\on{ad}}$.
Take an $S$-point \begin{equation*} (P,\tau_{\mu^{\lvee}}) \in (\Gr_{G,X^n}\underset{X^{n}}\times \Gr_{G,X^n})\times {\on{Maps}}(X,T^+_{\on{ad}}).\end{equation*}
The point $P \in (\Gr_{G,X^n}\underset{X^{n}}\times \Gr_{G,X^n})$ is represented by a collection of the following outer diamonds:
\begin{equation} \label{rhomb}\xymatrix@+10pt{
 & \CV^{\la^{\lvee}}_{\CF_{1}}
 \ar[dr]^{\zeta_{\la^{\lvee}}} \ar@{..>}^{\varphi_{\la^{\lvee}}}[dd] & \\
 \CO_{X\times S} \ar[ur]^{\eta_{\la^{\lvee}}} \ar[dr]_{\eta'_{\la^{\lvee}}} & & \CO_{X \times S}. \\
 & \CV^{\la^{\lvee}}_{\CF_{2}} \ar[ur]_{\zeta'_{\la^{\lvee}}} & \\
}\end{equation}
We must show that there is at most one collection of dotted arrows $\varphi_{\la^{\lvee}}$ making both triangles commutative.
Note that the collection of morphisms $(\eta_{\la^{\lvee}}, \zeta_{\la^{\lvee}}, \eta'_{\la^{\lvee}}, \zeta'_{\la^{\lvee}})$ defines the trivializations of
$G$-bundles
\begin{equation*}{\CF_{1}}|_{((S\times X)\setminus\{\varGamma_{x_1}\cup\dots\cup\varGamma_{x_n}\})},~
\CF_{2}|_{((S\times X)\setminus\{\varGamma_{x_1}\cup\dots\cup\varGamma_{x_n}\})}.\end{equation*}
After these trivializations and restrictions to
\begin{equation*} U:=(S\times X)\setminus\{\varGamma_{x_1}\cup\dots\cup\varGamma_{x_n}\}\end{equation*}
the outer diamond~(\ref{rhomb}) takes the following form:
\begin{equation*}\xymatrix@+10pt{
 & V^{\la^{\lvee}}\otimes\CO_{U} \ar[dr]^{\zeta_{\la^{\lvee}}} \ar@{..>}^{{\varphi_{\la^{\lvee}}}_{|U}}[dd] & \\
 \CO_U \ar[ur]^{\iota_{\la^{\lvee}}} \ar[dr]_{\iota_{\la^{\lvee}}} & & \CO_U, \\
 & V^{\la^{\lvee}}\otimes\CO_{U} \ar[ur]_{\zeta_{\la^{\lvee}}} & \\
}\end{equation*}
where the morphisms $\iota_{\la^{\lvee}}$ correspond to the highest vector embeddings $\BC \hookrightarrow V^{\la^{\lvee}}$
and the morphisms $(\zeta_{\la^{\lvee}})^{*}$ correspond to the lowest vector
embeddings $\BC \hookrightarrow (V^{\la^{\lvee}})^*$.
It follows that
$({\varphi_{\la^{\lvee}}}|_U,~\tau_{\mu^{\lvee}})$
actually corresponds to the unique point in $\on{Vin}_G(U)$
that can be described as $\mathfrak{s}({\tau_{\mu^{\lvee}}}|_U)$ (see~\S\ref{section} for the definition
of $\mathfrak{s}$). Indeed, the submonoid of $\on{Vin}_G$ consisting of $g \in \on{Vin}_G$ such
that $g_\la|_{(V^{\la^\vee})_{\la^\vee}}=\on{Id}$ for any
$\la \in \La^{\vee+}$ is $N\cdot\mathfrak{s}(T^+_{\on{ad}})$ and the submonoid of 
$\on{Vin}_G$ consisting of $g \in \on{Vin}_G$ such that 
$g^*_\la|_{(V^{\la^\vee})^{*}_{-\la^\vee}}=\on{Id}$ is $N_-\cdot\mathfrak{s}(T^+_{\on{ad}})$,
so their intersection is $\mathfrak{s}(T^+_{\on{ad}})$. The claim follows.

Let us now show that the morphism $\vartheta$ is proper. Consider the following fibre product:
\begin{equation*}{\mathcal X}:=((\Gr_{G,X^n}\underset{X^n}\times\Gr_{G,X^n})\times T^+_{\on{ad}})
\underset{\on{Bun}_G\times \on{Bun}_G}\times\on{VinBun}_{G}.\end{equation*}
We have a projection
\begin{equation*}
  {\mathcal X} \ra
(\Gr_{G,X^n}\underset{X^n}\times\Gr_{G,X^n})\times T^+_{\on{ad}}.\end{equation*}
Recall the projection $\on{VinGr}_{G,X^n} \ra \on{VinBun}_G$
of~\S\ref{VinGr via Tannak}. Together with the morphism $\vartheta$,
it induces a morphism
$\varsigma\colon \on{VinGr}_{G,X^n} \ra {\mathcal X}.$
It is easy to see that the morphism $\varsigma$ is a closed embedding.

Recall the action
$T \curvearrowright \on{VinBun}_G$
of \S\ref{VinBun}.
It follows from Proposition~\ref{proper over Bun} that the morphism
\begin{equation*}\kappa\colon\on{VinBun}_G/T \ra \on{Bun}_G\times \on{Bun}_G
\end{equation*}
is proper.
Then it follows that the morphism
\begin{equation}\label{properfactor} \widetilde{\kappa}\colon
((\Gr_{G,X^n}\underset{X^n}\times\Gr_{G,X^n})
\underset{\on{Bun}_G\times \on{Bun}_G}\times
\on{VinBun}_{G}/T) \times T^+_{\on{ad}}
\ra \Gr_{G,X^n}\underset{X^n}\times\Gr_{G,X^n}
\times T^+_{\on{ad}}\end{equation}
obtained by the base change of the morphism $\kappa$ is proper.
Let us denote the LHS of~(\ref{properfactor}) by $\BP {\mathcal X}$.
We consider the natural morphism
$\varrho\colon {\mathcal X} \ra \BP {\mathcal X}$.
Consider the following commutative diagram:
\begin{equation}\label{decomp of vartheta}\xymatrix{
 \on{VinGr}_{G,X^n} \ar[rrr]^{\varsigma} \ar[d]^{\vartheta}
 &&& {\mathcal X} \ar[d]^{\varrho} \\
 (\Gr_{G,X^n}\underset{X^n}\times\Gr_{G,X^n})
\times T^+_{\on{ad}} &&& \ar[lll]_{\widetilde{\kappa}} \BP {\mathcal X}
 }
\end{equation}
We claim that the composition
$\varrho \circ \varsigma$ is a closed embedding.
To prove this let us consider the image of the morphism $\varsigma$.
It is a closed subspace of the $T$-torsor
$\mathcal{X} \ra \BP {\mathcal X}$ that intersects each
$T$-orbit in a unique point. It follows that the image of
$\varrho \circ \varsigma$ is closed. The claim follows.
Now from the commutativity of the diagram~(\ref{decomp of vartheta})
we see that the morphism $\vartheta$
is the composition of two proper morphisms,
$\vartheta=\widetilde{\kappa}\circ (\varrho \circ \varsigma)$.
So $\vartheta$ is proper.
\end{proof}

\begin{lem} \label{morphism between products}  The morphism
    \begin{equation*}\Gr_{G,X^n}\underset{X^n}\times\Gr_{G,X^n} \ra \Gr_{G,X^n}\underset{X^n}\times \Gr_{G,X^n}
    \end{equation*} that sends
    \begin{equation*}(\ul{x}, \CO_{X\times S} \xrightarrow{\eta_{\la^{\!\scriptscriptstyle\vee}}}
      \CV^{\la^{\!\scriptscriptstyle\vee}}_{\CF} \xrightarrow{\zeta_{\la^{\!\scriptscriptstyle\vee}}} \CO_{X\times S},\
      \CO_{X\times S} \xrightarrow{\eta'_{\la^{\!\scriptscriptstyle\vee}}} \CV^{\la^{\!\scriptscriptstyle\vee}}_{\CF}
      \xrightarrow{\zeta'_{\la^{\!\scriptscriptstyle\vee}}} \CO_{X\times S}) \in
    \Gr_{G,X^n}\underset{X^n}\times \Gr_{G,X^n}\end{equation*}
to the point
\begin{equation*}(\ul{x}, \CO_{X\times S} \xrightarrow{\eta_{\la^{\!\scriptscriptstyle\vee}}}
  \CV^{\la^{\!\scriptscriptstyle\vee}}_{\CF}
  \xrightarrow{\zeta'_{\la^{\!\scriptscriptstyle\vee}}\circ \eta'_{\la^{\!\scriptscriptstyle\vee}}\circ\zeta_{\la^{\!\scriptscriptstyle\vee}}} \CO_{X\times S},~\CO_{X\times S}
  \xrightarrow{\eta'_{\la^{\!\scriptscriptstyle\vee}}\circ\zeta_{\la^{\!\scriptscriptstyle\vee}}\circ \eta_{\la^{\!\scriptscriptstyle\vee}}} \CV^{\la^{\!\scriptscriptstyle\vee}}_{\CF}
  \xrightarrow{\zeta'_{\la^{\!\scriptscriptstyle\vee}}} \CO_{X\times S}) \in \Gr_{G,X^n}\underset{X^n}\times \Gr_{G,X^n}\end{equation*}
is the identity morphism.
\end{lem}

\begin{proof}
Obvious.
\end{proof}

\begin{prop}
\label{repres of BD}
The functor $\on{VinGr}_{G,X^n}$ is represented by an ind-scheme ind-projective over $T^+_{\on{ad}}$.
\end{prop}

\begin{proof}
In Lemma~\ref{embedding into product} it is proved that $\on{VinGr}_{G,X^n}$
is a closed subfunctor of
\begin{equation*}
\Gr_{G,X^n}\underset{X^n}\times \Gr_{G,X^n} \times T^+_{\on{ad}}
\end{equation*}
and the following diagram is commutative:
\begin{equation*}
\xymatrix{
\on{VinGr}_{G,X^n} \ar[dr]^{\Upsilon} \ar[rr]^{\vartheta}
&& \ar[dl]_{\on{pr}_3}
\Gr_{G,X^n}\underset{X^n}\times \Gr_{G,X^n}\times T^+_{\on{ad}}\\
& T^+_{\on{ad}} &
}
\end{equation*}
where the morphism
$\on{pr}_3\colon \Gr_{G,X^n}\underset{X^n}\times \Gr_{G,X^n}\times T^+_{\on{ad}} \ra T^+_{\on{ad}}$
is the projection to the third factor.
It is known \cite[\S5.3.10]{bd},~\cite[~Theorem 3.1.3]{zh} that the functor $\Gr_{G,X^n}$ is
represented by an ind-projective scheme. It follows that the functor $\on{VinGr}_{G,X^n}$ is
represented by a closed ind-subscheme of
$\Gr_{G,X^n}\underset{X^n}\times \Gr_{G,X^n} \times T^+_{\on{ad}}$
ind-projective over $T^+_{\on{ad}}$.
\end{proof}

\sssec{Defect free locus of $\on{VinGr}_{G,x}$}
Set
\begin{equation*}
_{0}\!\!\on{VinGr}_{G,x}:= \on{Maps}_{T^+_{\on{ad}}}(X\times T^+_{\on{ad}},~_{0}\!\!\on{Vin}_{G}/(G\times G))
\underset{\on{Maps}_{T^{+}_{\on{ad}}}((X\setminus \{x\})\times T^{+}_{\on{ad}},~_{0}\!\!\on{Vin}_G/(G\times G))}\times T^{+}_{\on{ad}}.
\end{equation*}
%The ind-scheme $_{0}\!\!\on{VinGr}_{G,x}$ is a smooth open ind-subscheme of $\on{VinGr}_{G,x}$.
We denote by $_{0}\!\Upsilon\colon _{0}\!\!\on{VinGr}_{G,x} \ra T^+_{\on{ad}}$
the projection to the second factor.

\sssec{The principal degeneration $\on{VinGr}^{\on{princ}}_{G,X^n}$}
\label{princ deg VinGr} We denote by $\on{VinGr}^{\on{princ}}_{G,X^n}$
(resp.\ $_{0}\!\!\on{VinGr}^{\on{princ}}_{G,x}$) the restriction of the degeneration
$\Upsilon\colon \on{VinGr}_{G,X^n}\ra T_{\on{ad}}^{+}$
(resp.\ $_{0}\!\Upsilon$) to the ``principal'' line
\begin{equation*}\BA^1 \hookrightarrow T_{\on{ad}}^+,\ a \mapsto
(a, \dots, a)~\on{in~coordinates}~\al^{\lvee}_i,\ i \in I.\end{equation*}
We denote the corresponding morphism $\on{VinGr}^{\on{princ}}_{G,X^n} \ra \BA^1$ by
$\Upsilon^{\on{princ}}$ (resp.\ $_{0}\!\!\on{VinGr}^{\on{princ}}_{G,x} \ra \BA^1$ by
$_{0}\!\Upsilon^{\on{princ}}$).

Let us denote by
$\vartheta^{\on{princ}}\colon\on{VinGr}^{\on{princ}}_{G,X^n} \ra (\Gr_{G,X^n}\underset{X^{n}}\times \Gr_{G,X^n})\times \BA^1$
the restriction of the morphism $\vartheta$ of \S\ref{embedding}
to $\on{VinGr}^{\on{princ}}_{G,X^n}$.

\sssec{The special fiber of the degeneration $\on{VinGr}_{G,x}$} \label{special fiber} Let us describe the fiber over $0 \in T^+_{\on{ad}}$ of the morphism $\Upsilon$. Set
$(\on{VinGr}_{G,x})_0:=\Upsilon^{-1}(0).
$ Note that the morphism \begin{equation*}\vartheta\colon \on{VinGr}_{G,x} \hookrightarrow \Gr_G \times \Gr_G \times T^+_{\on{ad}}\end{equation*} of \S\ref{embedding} induces the closed embedding
\begin{equation*}\vartheta_0\colon (\on{VinGr}_{G,x})_0 \hookrightarrow \Gr_G \times \Gr_G.\end{equation*}

\begin{lem}\label{special fiber of vartheta}
The morphism $\vartheta_{0}$ induces an isomorphism between $(\on{VinGr}_{G,x})_0$ and
\begin{equation*}
\bigcup_{\mu \in \La}\ol{T}_{\mu} \times \ol{S}_{\mu}
\subset \Gr_G\times \Gr_G
\end{equation*}
considered as the reduced ind-schemes.
\end{lem}

\begin{proof}
We define a morphism \begin{equation*} \varsigma\colon\bigcup_{\mu \in \La}\ol{T}_{\mu} \times \ol{S}_{\mu} \hookrightarrow (\on{VinGr}_{G,x})_0.\end{equation*}
It sends a pair of points
\begin{equation*}
 \CO_{S\times X} \xrightarrow{\eta'_{\la^{\lvee}}}
 \CV^{\la^{\lvee}}_{\CF'} \xrightarrow{\zeta'_{\la^{\lvee}}} \CO_{S\times X}(\langle-\la^{\lvee},\mu\rangle\cdot (S\times x)) \in \ol{T}_{\mu}(S),\end{equation*}
\begin{equation*} \CO_{S\times X}(\langle-\la^{\lvee},\mu\rangle \cdot (S\times x)) \xrightarrow{\eta_{\la^{\lvee}}} \CV^{\la^{\lvee}}_{\CF} \xrightarrow{\zeta_{\la^{\lvee}}} \CO_{S\times X} \in \ol{S}_{\mu}(S)\end{equation*}
to the point
\begin{equation*}\CO_{S\times X} \xrightarrow{\eta'_{\la^{\lvee}}} \CV^{\la^{\lvee}}_{\CF'  } \xrightarrow{\eta_{\la^{\lvee}}\circ\zeta'_{\la^{\lvee}}}
\CV^{\la^{\lvee}}_{\CF} \xrightarrow{\zeta_{\la^{\lvee}}} \CO_{S\times X} \in (\on{VinGr_{G,x})_{0}}\end{equation*}
It follows from Lemma~\ref{morphism between products} that the composition
$\vartheta_0\circ\varsigma$ coincides with the natural closed embedding \begin{equation*}\bigcup_{\mu \in \La}\ol{T}_{\mu} \times \ol{S}_{\mu} \subset \Gr_G\times \Gr_G.\end{equation*}
It follows that the morphism $\varsigma$ is a closed embedding. It suffices to show now that the morphism $\varsigma$ is surjective on the level of $\BC$-points. Take a $\BC$-point \begin{equation*}\on{P}:=(\CF_1,\CF_2,\varphi_{\la^{\lvee}},\zeta_{\la^{\lvee}},\eta_{\la^{\vee}}) \in (\on{VinGr}_{G,x})_0.\end{equation*}
The morphisms $\varphi_{\la^{\lvee}}$ admit a unique factorization as
\begin{equation*}
\CV^{\la^{\lvee}}_{\CF_1} \xrightarrow{\zeta'_{\la^{\lvee}}} \CL^{\la^{\lvee}}  \xrightarrow{\eta'_{\la^{\lvee}}} \CV^{\la^{\lvee}}_{\CF_2}
\end{equation*}
where $\CL^{\la^{\lvee}}$ is a line bundle on the curve $X$ and the first morphism is a surjection of vector
bundles. Note that for any $\la^{\lvee}_1, \la^{\lvee}_2 \in \La^{\vee+}$ we have the identification
$\CL^{\la_1^{\lvee}} \otimes \CL^{\la_2^{\lvee}} = \CL^{\la_1^{\lvee}+\la_2^{\lvee}}$ and $\CL^{0}=\CO_{S\times X}$.
It follows that there exists $\mu \in \La$ such that
\begin{equation*}\CL^{\la^{\lvee}} \simeq \CO_{S\times X}(\langle-\la^{\lvee},\mu\rangle \cdot (S\times x)).\end{equation*}
So the $\BC$-point $\on{P}$ is the image of the $\BC$-point \begin{equation*}(\CF_1,\eta_{\la^{\lvee}}, \zeta'_{\la^{\lvee}}) \times (\CF_2,\eta'_{\la^{\lvee}},\zeta_{\la^{\lvee}}) \in \ol{T}_{\mu}\times \ol{S}_{\mu}
\end{equation*}
under the morphism $\varsigma$.
\end{proof}

\begin{rem}\label{sp fiber def free for vingr}{\em It follows from
      Proposition~\ref{fiber for vinbun} that the morphism $\vartheta_0$
      restricts to the isomorphism between the ind-schemes
      $(\on{VinGr}_{G,x})_0~\cap~_{0}\!\!\on{VinGr}_{G,x}$ and
$\bigsqcup_{\mu \in \La}T_{\mu} \times S_{\mu}$.}
\end{rem}

\sssec{Ind-scheme $\on{Vin}\ol{S}_{\nu}$} \label{VinS}
Fix a point $x \in X$ and a cocharacter $\nu \in \La$.
We consider a functor
\begin{equation*}
\on{Vin}\ol{S}_{\nu}\colon {\bf{Sch}} \ra {\bf{Sets}},
\end{equation*}
associating to a scheme $S$

1) an $S$-point $(\CF_1, \CF_2, \varphi_{\la^{\lvee}}, \tau_{\mu^{\lvee}}) \in \on{VinBun}_G(S)$ of $\on{VinBun}_G$,

2)  for every $\la^{\!\scriptscriptstyle\vee} \in \Lambda^{\vee+}$, morphisms of sheaves
\begin{equation*}
\eta_{\la^{\!\scriptscriptstyle\vee}}\colon
\CO_{S\times X}(\langle-\la^{\lvee},\nu\rangle\cdot (S\times x)) \ra
\CV^{\la^{\!\scriptscriptstyle\vee}}_{\CF_1},~
\zeta_{\la^{\!\scriptscriptstyle\vee}}\colon  \CV^{\la^{\lvee}}_{\CF_2} \ra \CO_{S\times X},\end{equation*}
satisfying the same conditions as in~\S\ref{degen of the affine Gr}.

Let $\widetilde{r}_{\nu,+}\colon \on{Vin}\ol{S}_{\nu}
\hookrightarrow \on{VinGr}_{G,x}$ be the natural embedding.
\begin{lem} The morphism $\widetilde{r}_{\nu,+}$ is a closed embedding.
\end{lem}
\begin{proof}
Recall the closed embedding
$\vartheta\colon\on{VinGr}_{G,x} \hookrightarrow
\Gr_G\times\Gr_G\times T^+_{\on{ad}}$
of \S\ref{embedding}.
It is enough to show that the composition
$\vartheta \circ \widetilde{r}_{\nu}$ is a closed embedding.
The proof is the same as the proof of Lemma~\ref{embedding into product}.
\end{proof}
The morphism
$
\Upsilon_{\nu}\colon
\on{Vin}\ol{S}_\nu \ra T^+_{\on{ad}}
$
is defined as the composition
of the morphisms
\begin{equation*}
\on{Vin}\ol{S}_{\nu} \ra \on{VinBun}_G \ra T^+_{\on{ad}}.
\end{equation*}

\sssec{The principal degeneration $\on{Vin}\ol{S}_{\nu}^{\on{princ}}$} \label{VinS princ}
Let us denote by $\on{Vin}\ol{S}_{\nu}^{\on{princ}}$
the restriction of the degeneration
$\Upsilon_{\nu}\colon \on{Vin}\ol{S}_{\nu} \ra T^+_{\on{ad}}$
to the ``principal'' line
\begin{equation*}\BA^1 \hookrightarrow T_{\on{ad}}^+,\ a \mapsto (a, \dots, a)~\on{in~coordinates}~\al^{\lvee}_i,\ i \in I.\end{equation*}
Let us denote the corresponding morphism $\on{Vin}\ol{S}_{\nu}^{\on{princ}} \ra \BA^1$ by
$\Upsilon^{\on{princ}}_{\nu}$.

It follows from Lemma~\ref{restriction for VinGrBD}
that the fiber of the morphism
$\Upsilon^{\on{princ}}_{\nu}$ over the point
$1\in\BA^1$ is isomorphic
to $\ol{S}_{\nu}$.
It follows from~Lemma~\ref{special fiber of vartheta} that
the morphism $\widetilde{r}_{\nu,+}$ induces an isomorphism
from $(\on{Vin}\ol{S}_{\nu})_0$ to
\begin{equation*}
\bigcup\limits_{\mu \in \La,\mu \leq \nu}
(\ol{T}_{\mu} \cap \ol{S}_{\nu})\times \ol{S}_{\mu}.
\end{equation*}

\sssec{} \label{embedding of local models}
Let us fix
$\theta_1, \theta_2 \in \La$, $\theta_1 \geq \theta_2$, set
$\theta:=\theta_1-\theta_2$.
Recall the families
$\ol{\mathfrak{Y}}^{\theta},$
$\ol{\mathfrak{Y}}^{\theta_1,\theta_2}$
over $T^+_{\on{ad}}$ of~\S\ref{compact Sch loc model},
~\ref{shifted compact Sch loc mod}.
Let us denote by ${\imath}_{\theta_1,\theta_2}$, ${\imath}_{\theta}$
the closed embeddings of the families
$\ol{\mathfrak{Y}}^{\theta},$
$\ol{\mathfrak{Y}}^{\theta_1,\theta_2}$
into $\Gr_G\times \Gr_G \times T^+_{\on{ad}}$
that are constructed in the same way as the embedding $\vartheta$
of \S\ref{embedding}.

\begin{lem} \label{identification of local models} (cf.~\cite[\S6.3]{bfgm})
The family $\ol{\mathfrak{Y}}^{\theta_1,\theta_2} \ra T^{+}_{\on{ad}}$
is isomorphic to the family
$\ol{\mathfrak{Y}}^{\theta} \ra T^{+}_{\on{ad}}$.
\end{lem}

\begin{proof}
Recall the identification $\Gr_G=G_{\CK}/G_{\CO}$ of \S\ref{setup}.
Recall that $\theta_2 \in \La$ defines an element
$z^{\theta_2} \in T(\CK)$. We get the isomorphism
\begin{equation*}
\aleph_{\theta_2}\colon \Gr_G \times \Gr_G \times T^{+}_{\on{ad}} \iso
\Gr_G \times \Gr_G \times T^{+}_{\on{ad}}
\end{equation*}
given by $([g_1],[g_2],t) \mapsto ([z^{\theta_2}g_1],[z^{\theta_2}g_2],t)$.
It is easy to see that the isomorphism $\aleph_{\theta_2}$ identifies the
subscheme
${\imath}_{\theta}(\ol{\mathfrak{Y}}^{\theta})\simeq \ol{\mathfrak{Y}}^{\theta}$ of $\Gr_G\times \Gr_G \times T^+_{\on{ad}}$ with the subscheme
${\imath}_{\theta_1,\theta_2}(\ol{\mathfrak{Y}}^{\theta_1,\theta_2})\simeq \ol{\mathfrak{Y}}^{\theta_1,\theta_2}$ of $\Gr_G\times \Gr_G \times T^+_{\on{ad}}$.
\end{proof}

\subsection{Drinfeld-Gaitsgory interpolation}

\sssec{The space of~$\BC^{\times}$-equivariant morphisms} Let $Z_1$ and $Z_2$ be schemes equipped with an action of $\BC^\times$. Then we define the space ${\on{Maps}}^{\BC^{\times}}(Z_1,Z_2)$ as follows: for any scheme $S$,
\begin{equation*}{\on{Maps}}^{\BC^{\times}}(Z_1,Z_2)(S):=\on{Mor}(S\times Z_1,Z_2)^{\BC^{\times}}\end{equation*}

\sssec{Attractors and repellents} (c.f.~\cite[Sections~1.4,~1.5]{dg2}). Let $Z$ be a scheme equipped with an action of $\BC^{\times}$. Then we set $Z^{\on{attr}}:={\on{Maps}}^{\BC^{\times}}(\BA^1,Z)$,
where $\BC^{\times}$ acts on $\BA^1$ by dilations. It follows from~\cite[Corollary~1.5.3(ii)]{dg2} that the functor $Z^{\on{attr}}$ is represented by a scheme.

Let $\BA^1_{-}$ be an affine line equipped with the following action of $\BC^{\times}$: \begin{equation*}\BC^{\times}\times \BA^1_{-} \ra \BA^1_{-},~ (c,a) \mapsto c^{-1}a. \end{equation*}
We set $Z^{\on{rep}}:={\on{Maps}}^{\BC^{\times}}(\BA^1_{-},Z)$.
The scheme $Z^{\on{attr}}$ is called the {\em{attractor}} of $Z$, and
$Z^{\on{rep}}$ is called the {\em{repellent}} of $Z$.

Recall that $Z^{\BC^{\times}}:={\on{Maps}}^{\BC^{\times}}(\on{pt},Z).$ The $\BC^{\times}$-equivariant morphisms $0\colon \on{pt} \ra \BA^{1},~0_{-}\colon \on{pt} \ra \BA^{1}_{-}$ induce the morphisms
$q^{+}\colon Z^{\on{attr}} \ra Z^{\BC^{\times}},~q^{-}\colon Z^{\on{rep}} \ra Z^{\BC^{\times}}$.

\sssec{Definition of the interpolation} Let $Z$ be a scheme equipped with a $\BC^{\times}$-action. In \cite[\S2]{dg2}, certain interpolation $\widetilde{Z}$ over $\BA^1$ was defined. Let us recall the construction of $\widetilde{Z}$.

Set $\BX:=\BA^2$ and consider the morphism $\BX \ra \BA^1,\ (\tau_1,\tau_2) \mapsto \tau_1 \tau_2.$ For any scheme $S$ over $\BA^1$ set $\BX_S:= \BX \underset{\BA^1}\times S$.
Let us consider the following $\BC^{\times}$-action on $\BX\colon
c \cdot (\tau_1, \tau_2):=(c \cdot \tau_1,c^{-1} \cdot \tau_2)$.
This action preserves the morphism $\BX \ra \BA^1$, so for any scheme $S$ one obtains a $\BC^{\times}$-action on $\BX_S$.

Define $\widetilde{Z}$ to be the following space over $\BA^1$:

\begin{equation*}\on{Mor}_{\BA^1}(S,\widetilde{Z}):=\on{Mor}(\BX_S,Z)^{\BC^{\times}}.\end{equation*}

\sssec{Properties of the interpolation} \label{main prop} Let us recall the main properties of $\widetilde{Z}$ from \cite{dg2}. The projection $\BX \ra \BA^1$ admits two sections: \begin{equation*}s_1(a):=(a,1),~s_2(a):=(1,a).\end{equation*}
The sections $s_{1}, s_2$ define morphisms $	\gamma_1\colon\widetilde{Z}\ra Z,~
\gamma_2\colon\widetilde{Z} \ra Z.
$
Let $
\gamma\colon \widetilde{Z} \ra
Z\times Z \times \BA^1$
denote the morphism whose third component is the tautological projection $\widetilde{Z}\ra \BA^1$, and the first and the second components are $\gamma_1$ and $\gamma_2$ respectively.
It follows from~\cite[Proposition~2.2.6]{dg2}
that the morphism $\gamma$ induces an isomorphism
between $\widetilde{Z}|_{\BG_m}$ and the graph of
the action morphism
$\BC^{\times} \times Z \ra Z$.

\sssec{The special fiber of the interpolation $\widetilde{Z}$}\label{zero fiber of DG} For $a \in \BA^1$ let us denote by $\widetilde{Z}_a$ the fiber of $\widetilde{Z}$ over $a$. It follows from \cite[Proposition~2.2.9]{dg2} that the following diagram is cartesian:
\begin{equation*}\begin{CD} \widetilde{Z}_{0} @>\gamma_1>> Z^{\on{rep}}\\
@VV\gamma_2V @VVq^-V\\
Z^{\on{attr}} @>q^{+}>> Z^{\BC^{\times}}
\end{CD}\end{equation*}
i.e. the fiber $\widetilde{Z}_0$ is canonically isomorphic to $Z^{\on{rep}}\underset{Z^{\BC^{\times}}}\times Z^{\on{attr}}$.

\sssec{Drinfeld-Gaitsgory interpolation of $\Gr_G$} Recall that $\Gr_G$ is the union of the projective schemes $\ol\Gr{}_G^\la,\ \la \in \La^+$. We consider the $\BC^\times$-action on $\Gr_G$ arising
from the coweight $2\rho\colon\BC^\times\to T\curvearrowright\Gr_G$.
It follows from~\cite{dg2} that the closed embeddings of the Schubert varieties induce
the closed embeddings of the corresponding Drinfeld-Gaitsgory interpolations
$\widetilde{\ol\Gr}{}_G^\mu \subset \widetilde{\ol\Gr}{}_G^\la,\ \mu \leq \la \in \La^+$,
and we define $\widetilde{\Gr}_G:=\lim\limits_{\lambda}
\widetilde{\ol\Gr}{}_G^\la$.

\sssec{The special fiber of the interpolation $\widetilde{\Gr}_G$} \label{specfiber} From \S\ref{zero fiber of DG} it follows that \begin{equation*}(\widetilde{\Gr}_G)_0 \simeq \Gr^{\on{rep}}_{G}\times_{(\Gr_G)^{\BC^{\times}}}\Gr^{\on{attr}}_G \simeq \bigsqcup\limits_{\mu \in \La} T_{\mu} \times S_{\mu}.
\end{equation*}

\sssec{The open embedding $\jmath\colon \widetilde{\Gr}_G \hookrightarrow \on{VinGr}^{\on{princ}}_{G,x}$} Let us construct an open embedding of the interpolation $\widetilde{\Gr}_G \ra \BA^1$ into the degeneration $\on{VinGr}^{\on{princ}}_{G,x} \ra \BA^1$ (considered as the schemes over $\BA^1$) such that the following diagram is commutative:

\begin{equation*}
\xymatrix{ \widetilde{\Gr}_G \ar[d]^{\gamma} \ar[r]^{\jmath} & \ar[d]^{\vartheta} \on{VinGr}^{\on{princ}}_{G,x}   \\
\Gr_G\times \Gr_G \times \BA^1  \ar[r]^{\on{Id}} & \Gr_G\times \Gr_G \times \BA^1.
}
\end{equation*}

Recall the open ind-subscheme
$_{0}\!\!\on{VinGr}^{\on{princ}}_{G,x} \subset \on{VinGr}^{\on{princ}}_{G,x}$
of \S\ref{princ deg VinGr}.

\begin{prop}[D.~Gaitsgory]
  \label{DG completion}
There exists an isomorphism
$\upeta\colon _{0}\!\!\on{VinGr}^{\on{princ}}_{G,x} \iso \widetilde{\on{Gr}}_G$ of the families over $\BA^1$
such that the following diagram is commutative:
\begin{equation}
  \label{don't copy diagrams!}
\xymatrix{ _{0}\!\!\on{VinGr}^{\on{princ}}_{G,x} \ar[d]^{\vartheta|_{_{0}\!\!\on{VinGr}^{\on{princ}}_{G,x}}} \ar[r]^{\upeta} & \ar[d]^{\gamma} \widetilde{\on{Gr}}_G   \\
\Gr_G\times \Gr_G \times \BA^1  \ar[r]^{\on{Id}} & \Gr_G\times \Gr_G \times \BA^1.
}
\end{equation}
\end{prop}

\begin{proof}
Will be given in Appendix~\ref{proof DG completion}.
\end{proof}

\section{The action of Schieder bialgebra on the fiber functor}
\label{four}

\subsection{Construction of the action} \label{construction action}
Fix $\CP \in \on{Perv}_{G_{\CO}}(\Gr_G)$. Set $V:=H^{\bullet}(\Gr_G,\CP)$.
Recall the closed embedding
\begin{equation*}
\vartheta^{\on{princ}}\colon\on{VinGr}^{\on{princ}}_{G,x} \hookrightarrow
\Gr_G \times \Gr_G \times \BA^1\end{equation*}
of \S\ref{princ deg VinGr}.
Let us fix a cocharacter $\nu \in \La$.
Recall the closed embedding
$\on{Vin}\ol{S}^{\on{princ}}_{\nu}
\hookrightarrow \on{VinGr}^{\on{princ}}_{G,x}$
of \S\ref{VinS}.
Let $\vartheta^{\on{princ}}_\nu\colon\on{Vin}\ol{S}^{\on{princ}}_\nu \hookrightarrow
\Gr_G \times \Gr_G \times \BA^1$
be the composition
\begin{equation*}
\on{Vin}\ol{S}^{\on{princ}}_{\nu}
\hookrightarrow
\on{VinGr}^{\on{princ}}_{G,x} \hookrightarrow
\Gr_G \times \Gr_G \times \BA^1.
\end{equation*}
For $\mu \in \La$ set
$\CP_\mu:=\CP|_{\ol{S}_\mu},~
\widetilde{\CP}_{\nu}:=
(\ul\BC\boxtimes\CP\boxtimes\ul\BC)|_{\on{Vin}\ol{S}_{\nu}^{\on{princ}}}$ (the $*$-restrictions
to the corresponding closed subvarieties).
We will see in Remark~\ref{support} that the support of the complex
$\widetilde{\CP}_\nu$ is finite-dimensional.

Recall the one-parameter deformation
$\Upsilon^{\on{princ}}_{\nu}\colon
\on{Vin}\ol{S}^{\on{princ}}_{\nu} \ra \BA^1$
of \S\ref{VinS princ}.

\begin{prop} \label{restrictions of Upsilon}
The one-parametric family $\Upsilon^{\on{princ}}_{\nu}$ is trivial over
$\BG_m$. The special fiber ${\Upsilon^{\on{princ}}_{\nu}}^{-1}(0)$ is
\begin{equation*}
\bigcup\limits_{\mu \leq \nu,\mu \in \La} (\ol{S}_{\nu}\cap \ol{T}_{\mu})\times \ol{S}_{\mu}.
\end{equation*}
A general fiber is $\ol{S}_{\nu}$.
The restriction
$\widetilde{\CP}_{\nu}
|_{{\Upsilon^{\on{princ}}_{\nu}}^{-1}(\BG_m)}$
is isomorphic to
$\CP|_{\ol{S}_{\nu}}\boxtimes \ul\BC{}_{\BG_m}$.
The restriction of $\widetilde{\CP}_{\nu}$
to $(\ol{S}_{\nu}\cap \ol{T}_{\mu}) \times \ol{S}_{\mu}$
is isomorphic to $\ul\BC{}_{(\ol{S}_{\nu}\cap \ol{T}_{\mu})}\boxtimes (\CP|_{\ol{S}_{\mu}})$.
\end{prop}

\begin{proof}
Follows from~\S\ref{VinS princ}.
\end{proof}

\begin{rem} \label{support}{\em{Let us show that the support of the complex $\widetilde{\CP}_\nu$ is finite-dimensional. It is enough to show that the supports of
$(\widetilde{\CP}_\nu)|
_{{\Upsilon^{\on{princ}}_\nu}^{-1}(\BG_m)}$
and
$(\widetilde{\CP}_\nu)|
_{{\Upsilon^{\on{princ}}_\nu}^{-1}(0)}$
are finite-dimensional.
It follows from Proposition~\ref{restrictions of Upsilon}
using the fact that there are finitely many $\mu \leq \nu$
such that $\CP_\mu$ is nonzero.
}
}
\end{rem}

\sssec{The action}
\label{The action}
Given a positive coweight $\nu \in \Lambda^\pos$
we define a morphism of vector spaces
\begin{equation*}
\bigoplus\limits_{\mu \leq \nu}(\CA_{\nu-\mu} \otimes V_{\mu})= H^{\langle2\rho^{\lvee},\nu\rangle}_{c}({\Upsilon^{\on{princ}}_{\nu}}^{-1}(0),(\widetilde{\CP}_{\nu})_0)\ra
H^{\langle2\rho^{\lvee},\nu\rangle}_{c}({\Upsilon^{\on{princ}}_{\nu}}^{-1}(1),(\widetilde{\CP}_{\nu})_1)=V_{\nu},
\end{equation*}
as the cospecialization morphism (with coefficients in the sheaf
$\widetilde{\CP}_{\nu}$) corresponding to the one-parameter
degeneration $\Upsilon^{\on{princ}}_{\nu}$. Here $(\widetilde{\CP}_{\nu})_0$
(resp.\ $(\widetilde{\CP}_{\nu})_1$) stands for the restriction of $\widetilde{\CP}_{\nu}$ to
the fiber of $\Upsilon^{\on{princ}}_{\nu}$ over $0\in\BA^1$ (resp.\ over $1\in\BA^1$).
Summing over all $\nu \in \Lambda^\pos$ we obtain the desired morphism
\begin{equation*}
\on{act}_V\colon \CA \otimes V \ra V.
\end{equation*}

\subsection{Associativity}
\sssec{The two-parameter deformation of Grassmannian}
\label{twoparam of Gr}
Recall two projections
\begin{equation*}
\xymatrix{
& \on{VinBun}_G \ar[dl]_{\wp_1} \ar[dr]^{\wp_2} &\\
\on{Bun}_G && \on{Bun}_G
}
\end{equation*}
of~Definition~\ref{vinbun}.
Set $\CW:=\on{VinBun}_G \underset{{\on{Bun}_G}}\times \on{VinBun}_G$.
For $n \in \BN$ let $\on{Vin^2Gr}_{G,X^n}$ be the moduli
space of the following data:
it associates to a scheme $S$

1) a collection of $S$-points $(x_1,\dots,x_n)\in X^n(S)$ of the curve $X$,

2) an $S$-point $(\CF_1,\CF_2,\CF_3,\varphi_{1,\la^{\lvee}},\varphi_{2,\la^{\lvee}},
\tau_{1,\mu^{\lvee}},\tau_{2,\mu^{\lvee}}) \in
\CW(S),$

3) for every $\la \in \La^{\vee+}$, the rational morphisms
\begin{equation*}
\eta_{\la^{\lvee}}\colon
\CO_{S\times X} \ra \CV^{\la^{\lvee}}_{\CF_1},~
\zeta_{\la^{\lvee}}\colon
\CV^{\la^{\lvee}}_{\CF_3} \ra \CO_{S\times X},
\end{equation*}
regular on $(S\times X)\setminus\{\varGamma_{x_1}\cup\dots\cup\varGamma_{x_n}\}$,
such that the data
\begin{equation*}
(\CF_1,\CF_3,\varphi_{2,\la^{\lvee}} \circ \varphi_{1,\la^{\lvee}},
\tau_{2,\mu^{\lvee}} \circ \tau_{1,\mu^{\lvee}},
\eta_{\la^{\lvee}},\zeta_{\la^{\lvee}})
\end{equation*}
define an $S$-point of $\on{VinGr}_{G,X^n}$.

Let us denote by $_0\!\CW$ the substack of $\CW$ consisting
of points
\begin{equation*}
(\CF_1,\CF_2,\CF_3,\varphi_{1,\la^{\lvee}},\varphi_{2,\la^{\lvee}},
\tau_{1,\mu^{\lvee}},\tau_{2,\mu^{\lvee}}) \in
\CW
\end{equation*}
such that the data
\begin{equation*}
(\CF_1,\CF_3,\varphi_{2,\la^{\lvee}} \circ \varphi_{1,\la^{\lvee}},
\tau_{2,\mu^{\lvee}} \circ \tau_{1,\mu^{\lvee}})
\end{equation*}
define a point of $\on{VinBun}_G$.

\begin{rem}{\em\label{Double VinGr via fiberprod}
The family $\on{Vin^2Gr}_{G,X^n}$ can be obtained as
the fibre product:
\begin{equation*} \xymatrix{
\on{Vin^2Gr}_{G,X^n}=
(\on{VinGr}_{G,X^n} \underset{\on{VinBun}_G}
\times {_0\!\CW})
\ar[r] \ar[d] &
{_0\!\CW}
\ar[d]^{\on{m}} \\
\on{VinGr}_{G,X^n} \ar[r]
& \on{VinBun}_G
}
\end{equation*}
where the morphism
$\on{m}\colon
{_0\!\CW} \ra
\on{VinBun}_G$
corresponds to the multiplication in $\on{Vin}_G$.}
\end{rem}

The degeneration morphism
$\Upsilon^2\colon \on{Vin^2Gr}_{G,X^n} \ra
T^+_{\on{ad}}\times T^+_{\on{ad}}$
is defined as the composition of the morphisms
\begin{equation*}
\on{Vin^2Gr}_{G,X^n} \ra
\CW \ra
T^+_{\on{ad}}\times T^+_{\on{ad}}.
\end{equation*}

Let us denote by $\on{Vin^2Gr}^{\on{princ}}_{G,X^n}$
the restriction of the degeneration $\Upsilon^2\colon\on{Vin^2Gr}_{G,X^n}
\ra T^+_{\on{ad}}\times T^+_{\on{ad}}$
to the product of~``principal'' lines
\begin{equation*}
\BA^1 \times \BA^1 \hookrightarrow T^+_{\on{ad}}\times T^+_{\on{ad}},~
(a_1,a_2) \mapsto ((a_1,\dots,a_1),(a_2,\dots,a_2)).
\end{equation*}
Let us denote the corresponding morphism
$\on{Vin^2Gr}^{\on{princ}}_{G,X^n} \ra \BA^2$ by
$\Upsilon^{2,\on{princ}}$.

\begin{lem}
The restrictions of the two-parameter family
$\Upsilon^{2,\on{princ}}\colon\on{Vin^2Gr}^{\on{princ}}_{G,X^n} \ra \BA^2$
to the lines
\begin{equation*}
\BA^{1}\times\{1\} \hookrightarrow \BA^1\times \BA^1,~
\{1\}\times\BA^{1} \hookrightarrow \BA^1\times \BA^1
\end{equation*}
are both isomorphic to the one-parameter family
$\Upsilon^{\on{princ}}\colon \on{VinGr}^{\on{princ}}_{G,X^n} \ra \BA^1$.
\end{lem}

\begin{proof}
Follows from Remark~\ref{Double VinGr via fiberprod}.
\end{proof}

Let us construct a closed embedding
\begin{equation*}\vartheta^2\colon
\on{Vin^2Gr}_{G,X^n} \hookrightarrow
\Gr_{G,X^n}\underset{X^n}\times\Gr_{G,X^n}
\underset{X^n}\times\Gr_{G,X^n}\times
T^+_{\on{ad}}\times T^+_{\on{ad}}.
\end{equation*}
It sends a point
\begin{equation*}
(\CF_1,\CF_2,\CF_3,\varphi_{1,\la^{\lvee}},\varphi_{2,\la^{\lvee}},
\tau_{1,\mu^{\lvee}},\tau_{2,\mu^{\lvee}},
\eta_{\la^{\lvee}},\zeta_{\la^{\lvee}}) \in \on{Vin^2Gr}_{G,X^n}:
\end{equation*}
\begin{equation*}
\CO_{S\times X} \xrightarrow{\eta_{\la^{\lvee}}}\CV^{\la^{\lvee}}_{\CF_1} \xrightarrow{\varphi_{1,\la^{\lvee}}} \CV^{\la^{\lvee}}_{\CF_2} \xrightarrow{\varphi_{2,\la^{\lvee}}} \CV^{\la^{\lvee}}_{\CF_3} \xrightarrow{\zeta_{\la^{\lvee}}} \CO_{S\times X}
\end{equation*}
to the point
\begin{equation*}
(\CF_1,\eta_{\la^{\lvee}},\zeta_{\la^{\lvee}}\circ
\varphi_{2,\la^{\lvee}}\circ\varphi_{1,\la^{\lvee}})
\times
(\CF_2,\varphi_{1,\la^{\lvee}} \circ \eta_{\la^{\lvee}},
\zeta_{\la^{\lvee}}\circ\varphi_{2,\la^{\lvee}})
\times
(\CF_3,\varphi_{2,\la^{\lvee}} \circ \varphi_{1,\la^{\lvee}}
\circ \eta_{\la^{\lvee}},\zeta_{\la^{\lvee}})\times
\tau_{1,\mu^{\lvee}}\times
\tau_{2,\mu^{\lvee}}.
\end{equation*}

\begin{lem}
The morphism $\vartheta^2$ is a closed embedding.
\end{lem}

\begin{proof}
The proof is the same as the one of
Lemma~\ref{embedding into product}.
\end{proof}

Let us denote by
\begin{equation*}
(\vartheta^{2})^{\on{princ}}\colon \on{Vin^2Gr}^{\on{princ}}_{G,X^n}
\hookrightarrow \Gr_{G,X^n}\underset{X^n}\times\Gr_{G,X^n}
\underset{X^n}\times\Gr_{G,X^n}\times
\BA^1\times \BA^1
\end{equation*}
the restriction of the morphism $\vartheta^2$
to $\on{Vin^2Gr}^{\on{princ}}_{G,X^n} \subset
\on{Vin^2Gr}_{G,X^n}$.

\sssec{The two-parameter deformations of closures of semi-infinite orbits} \label{two-param for S}
We fix $x \in X$.
Fix a cocharacter $\nu \in \La$.
Let $\on{Vin}^2\ol{S}_{\nu}$ be the moduli space
of the following data: it associates to a
scheme $S$

1) an $S$-point $(\CF_1,\CF_2,\CF_3,\varphi_{1,\la^{\lvee}},\varphi_{2,\la^{\lvee}},
\tau_{1,\mu^{\lvee}},\tau_{2,\mu^{\lvee}}) \in
\CW.$

2) For every $\la \in \La^{+\vee}$, a morphism
of sheaves
\begin{equation*}
\eta_{\la^{\lvee}}\colon
\CO_{S\times X}(\langle-\la^{\lvee},\nu\rangle \cdot (S\times x))
\ra \CV^{\la^{\lvee}}_{\CF_1}
\end{equation*}
and a rational morphism
$
\zeta_{\la^{\lvee}}\colon
\CV^{\la^{\lvee}}_{\CF_3} \ra \CO_{S\times X}
$
regular on $S\times (X\setminus \{x\})$, such that the data
\begin{equation*}
(\CF_1,\CF_3,\varphi_{2,\la^{\lvee}} \circ \varphi_{1,\la^{\lvee}},
\tau_{2,\mu^{\lvee}} \circ \tau_{1,\mu^{\lvee}},
\eta_{\la^{\lvee}},\zeta_{\la^{\lvee}})
\end{equation*}
defines an $S$-point of $\on{Vin}\ol{S}_{\nu}$.

\begin{rem}{\em\label{Double VinS via fiberprod}
The family $\on{Vin}^2\ol{S}_\nu$ can be obtained as
the fibre product:
\begin{equation*} \xymatrix{
\on{Vin}^2\ol{S}_\nu=
\on{Vin}\ol{S}_\nu \underset{\on{VinBun}_G}
\times {_0\!\CW}
\ar[r] \ar[d] &
{_0\!\CW}
\ar[d]^{\on{m}} \\
\on{Vin}\ol{S}_\nu \ar[r]
& \on{VinBun}_G.
}
\end{equation*}}
\end{rem}

The degeneration morphism
$\Upsilon^2_{\nu}\colon \on{Vin}^2\ol{S}_\nu \ra
T^+_{\on{ad}}\times T^+_{\on{ad}}$
is defined as the composition of the morphisms
\begin{equation*}
\on{Vin}^2\ol{S}_\nu \ra
\CW \ra
T^+_{\on{ad}}\times T^+_{\on{ad}}.
\end{equation*}

Let us denote by $\on{Vin}^2\ol{S}{}^{\on{princ}}_\nu$
the restriction of the degeneration $\Upsilon^2_{\nu}\colon\on{Vin}^2\ol{S}_\nu
\ra T^+_{\on{ad}}\times T^+_{\on{ad}}$
to the product of~``principal'' lines
\begin{equation*}
\BA^1 \times \BA^1 \hookrightarrow T^+_{\on{ad}}\times T^+_{\on{ad}},~
(a_1,a_2) \mapsto
((a_1,\dots,a_1),(a_2,\dots,a_2)).
\end{equation*}
Let us denote the corresponding morphism
$\on{Vin}^2\ol{S}{}^{\on{princ}}_\nu \ra \BA^2$ by
$\Upsilon^{2,\on{princ}}_{\nu}$.

Let
$
r^2_{\nu,+}\colon \on{Vin}^2\ol{S}_\nu
\hookrightarrow \on{Vin^2Gr}_G
$
be the natural closed embedding.
Recall the morphism $\vartheta^2$
of \S\ref{twoparam of Gr}.
Let \begin{equation*}
\vartheta^2_\nu\colon
\on{Vin}^2\ol{S}_\nu \hookrightarrow
\Gr_{G}\times\Gr_{G}
\times\Gr_{G}\times
T^+_{\on{ad}}\times T^+_{\on{ad}}
\end{equation*}
be the composition $\vartheta^2_\nu:=
r^2_{\nu,+} \circ~
\vartheta^2$.
Let
\begin{equation*}
(\vartheta^2_\nu)^{\on{princ}}\colon\on{Vin}^2\ol{S}{}^{\on{princ}}_\nu
\hookrightarrow \Gr_G\times \Gr_G \times \Gr_G \times \BA^1 \times \BA^1
\end{equation*}
be the restriction of the morphism $\vartheta^2_\nu$ to
$\on{Vin}^2\ol{S}{}^{\on{princ}}_\nu
\subset \on{Vin}^2\ol{S}_\nu$.

\begin{prop}
  \label{associ}
  Let $\CP\in\on{Perv}_{G_\CO}(\Gr_G)$, and $V=H^{\bullet}(\Gr_G,\CP)$.
  Given $\nu\geq\mu_1\geq\mu_2$, the following diagram commutes:
  \begin{equation} \label{assoc of action}
\xymatrix{
\CA_{\nu-\mu_1} \otimes \CA_{\mu_1-\mu_2} \otimes V_{\mu_2} \ar[rrr]^{{\bf{m}}_{\nu-\mu_1,\mu_1-\mu_2} \otimes \on{Id}} \ar[d]^{\on{Id} \otimes \on{act}_V} &&& \CA_{\nu-\mu_2} \otimes V_{\mu_2} \ar[d]^{\on{act}_V} \\
\CA_{\nu-\mu_1} \otimes V_{\mu_1} \ar[rrr]^{\on{act}_V} &&& V_\nu.
}
  \end{equation}
\end{prop}

\begin{proof}
Set
\begin{equation*}
\widetilde{\CP}^2_\nu:=
(\ul\BC{}\boxtimes\ul\BC{}\boxtimes\CP\boxtimes\ul\BC{}\boxtimes\ul\BC{})|_{\on{Vin}^2\ol{S}{}^{\on{princ}}_\nu}.
\end{equation*}
It follows from Corollary~\ref{fibers of two-param} below that the support of
the complex $\widetilde{\CP}^2_\nu$ is
finite-dimensional.
For a point $(a_1,a_2)\in \BA^1\times \BA^1$
let us denote by $(\widetilde{\CP}^2_\nu)_{(a_1,a_2)}$
the restriction of $\widetilde{\CP}^2_\nu$ to
the fiber $(\Upsilon^{2,\on{princ}}_\nu)^{-1}(a_1,a_2)$.
Let us consider the tautological action
$\BC^\times \times \BC^\times \curvearrowright \BA^1\times\BA^1$.
Let us consider the stratification
\begin{equation*}
\BA^1\times\BA^1=(\BG_m\times\BG_m)\sqcup
(\BG_m\times\{0\})\sqcup
(\{0\}\times\BG_m)\sqcup
(\{0\}\times\{0\})
\end{equation*}
by $\BC^\times\times\BC^\times$-orbits. Recall the closed embedding
\begin{equation*}
(\vartheta^2_\nu)^{\on{princ}}\colon\on{Vin}^2\ol{S}{}^{\on{princ}}_\nu
\hookrightarrow \Gr_G\times \Gr_G \times \Gr_G \times \BA^1 \times \BA^1.
\end{equation*}
\begin{lem} \label{two-param to strata}
a) The restriction of the morphism $(\vartheta^2_\nu)^{\on{princ}}$
of the families over $\BA^1\times \BA^1$
\begin{equation*}
\xymatrix{
\on{Vin}^2\ol{S}{}^{\on{princ}}_\nu \ar[dr] \ar[rr]^{(\vartheta^2_\nu)^{\on{princ}}}
&& \ar[dl]
\Gr_G\times \Gr_G \times \Gr_G \times \BA^1 \times \BA^1\\
& \BA^1\times\BA^1 &
}
\end{equation*}
to the stratum $\BG_m \times \BG_m \subset \BA^1\times \BA^1$
is isomorphic to
\begin{equation*}
\xymatrix{
\ol{S}_\nu \times \BG_m \times \BG_m \ar[dr]^{} \ar[rr]^{(\vartheta^2_\nu)^{\on{princ}}|_{(\BG_m \times \BG_m)}}
&& \ar[dl]
\Gr_G\times \Gr_G \times \Gr_G \times \BG_m \times \BG_m \\
& \BG_m \times \BG_m &
}
\end{equation*}
where the morphism $(\vartheta^2_\nu)^{\on{princ}}|_{(\BG_m \times \BG_m)}$ is given by
$
(g,c_1,c_2) \mapsto (g,2\rho(c_1)\cdot g,2\rho(c_1c_2)\cdot g,c_1,c_2).
$

b) The restriction of the morphism $(\vartheta^2_\nu)^{\on{princ}}$
of the families over $\BA^1\times \BA^1$
to the stratum $\BG_m \times \{0\} \subset \BA^1\times \BA^1$
is isomorphic to
\begin{equation*}
\xymatrix{
\bigcup\limits_{\mu \in \La, \nu \geq \mu}
(\ol{T}_\mu \cap \ol{S}_\nu) \times \ol{S}_\mu \times
\BG_m \ar[dr]^{}
\ar[rr]^{(\vartheta^2_\nu)^{\on{princ}}|_{(\BG_m \times \{0\})}}
&& \ar[dl]
\Gr_G\times \Gr_G \times \Gr_G \times \BG_m \times \{0\} \\
& \BG_m \times \{0\} &
}
\end{equation*}
where the map $(\vartheta^2_\nu)^{\on{princ}}|_{(\BG_m \times \{0\})}$ is given by
$
(g_1,g_2,c) \mapsto (g_1,2\rho(c)\cdot g_1,g_2,c,0).
$

c) The restriction of the morphism $(\vartheta^2_\nu)^{\on{princ}}$
of the families over $\BA^1\times \BA^1$
to the stratum $\{0\}\times\BG_m \subset \BA^1\times \BA^1$
is isomorphic to
\begin{equation*}
\xymatrix{
\bigcup\limits_{\mu \in \La, \nu \geq \mu}
(\ol{T}_\mu \cap \ol{S}_\nu) \times \ol{S}_\mu \times
\BG_m \ar[dr]^{}
\ar[rr]^{(\vartheta^2_\nu)^{\on{princ}}|_{(\{0\} \times \BG_m)}}
&& \ar[dl]
\Gr_G\times \Gr_G \times \Gr_G \times \{0\} \times \BG_m \\
& \{0\}\times\BG_m &
}
\end{equation*}
where the map $(\vartheta^2_\nu)^{\on{princ}}|_{(\{0\}\times\BG_m)}$ is given by
$
(g_1,g_2,c) \mapsto (g_1,g_2,2\rho(c)\cdot g_2,0,c).
$

d) The restriction of the morphism $(\vartheta^2_\nu)^{\on{princ}}$
of the families over $\BA^1\times \BA^1$
to the point $(\{0\} \times \{0\})$
is isomorphic to
\begin{equation*}
\xymatrix{
\bigcup\limits_{\mu_1,\mu_2 \in \La, \nu \geq \mu_1 \geq \mu_2}
(\ol{S}_\nu \cap \ol{T}_{\mu_1})\times
(\ol{S}_{\mu_1}\cap\ol{T}_{\mu_2})\times\ol{S}_{\mu_2}
\ar[dr] \ar[rr]
&& \ar[dl]
\Gr_G\times \Gr_G \times \Gr_G \\
& (\{0\}\times\{0\}) &
}.
\end{equation*}
\end{lem}

\begin{proof}
The proof is the same as the one of Lemma~\ref{special fiber of vartheta}.
\end{proof}

\begin{cor} \label{fibers of two-param}
Under the identifications of
Lemma~\ref{two-param to strata} we have
\begin{equation*}
(\widetilde{\CP}^2_{\nu})_{(1,1)}=\CP_\nu,
~(\widetilde{\CP}^2_{\nu})_{(1,0)}|
_{(\ol{T}_\mu \cap \ol{S}_\nu) \times \ol{S}_\mu}=\ul\BC{}\boxtimes\CP_\mu,
~(\widetilde{\CP}^2_{\nu})_{(0,1)}|
_{(\ol{T}_\mu \cap \ol{S}_\nu) \times \ol{S}_\mu}=\ul\BC{}\boxtimes\CP_\mu,
\end{equation*}
\begin{equation*}
~(\widetilde{\CP}^2_{\nu})_{(0,0)}|
_{(\ol{S}_\nu \cap \ol{T}_{\mu_1})\times
(\ol{S}_{\mu_1}\cap\ol{T}_{\mu_2})\times\ol{S}_{\mu_2}}=
\ul\BC{}\boxtimes\ul\BC{}\boxtimes\CP_{\mu_2}.
\end{equation*}
\end{cor}

Let us fix a cocharacter $\mu \in \La, \mu \leq \nu$.
Recall the family $\Upsilon^{2,\on{princ}}_\nu\colon
\on{Vin}\ol{S}{}^{\on{princ}}_{\nu} \ra \BA^1$ of~\S\ref{two-param for S} and
the families $\ol{\mathfrak{Y}}{}^{\nu-\mu,\on{princ}}, \ol{\mathfrak{Y}}{}^{\nu,\mu,\on{princ}}$
of~\S\ref{compact Sch loc model} and
\S\ref{shifted compact Sch loc mod}.

\begin{lem} \label{closures of fam}

a) The closure of the family
$(\ol{T_\mu}\cap \ol{S}_\nu)\times\ol{S}_\mu\times \BG_m\ra \BG_m \times \{0\}$
in the family
$(\Upsilon^{2,\on{princ}}_\nu)^{-1}(\BA^1\times\{0\}) \ra \BA^{1}$
is isomorphic to the family
$\ol{\mathfrak{Y}}{}^{\nu,\mu,\on{princ}}\times \ol{S}_\mu \ra \BA^1$
on the level of reduced schemes.

b) The closure of the family
$(\ol{T_\mu}\cap \ol{S}_\nu)\times\ol{S}_\mu\times \BG_m \ra \{0\}\times \BG_m$
in the family
$(\Upsilon^{2,\on{princ}}_\nu)^{-1}(\{0\}\times\BA^1) \ra \BA^{1}$
is isomorphic to the family
$(\ol{T_\mu}\cap \ol{S}_\nu)\times \on{Vin}\ol{S}{}^{\on{princ}}_\mu \ra \BA^1$
on the level of reduced schemes.
\end{lem}

\begin{proof}
To prove a) let us construct a closed embedding
$\varkappa \colon\ol{\mathfrak{Y}}{}^{\nu,\mu,\on{princ}} \times \ol{S}_{\mu}
\hookrightarrow {\on{Vin}^2\ol{S}{}^{\on{princ}}_{\nu}}|_{\BA^1 \times \{0\}}$ of families over $\BA^1$. It sends an $S$-point
\begin{equation*}
((\CF_1,\CF_2,\varphi_{\la^{\lvee}},\tau_{\mu^{\lvee}},\eta_{\la^{\lvee}},\zeta_{\la^{\lvee}}),(\CF_3,\eta'_{\la^{\lvee}},\zeta'_{\la^{\lvee}})) \in (\ol{\mathfrak{Y}}{}^{\nu,\mu,\on{princ}} \times \ol{S}_{\mu})(S):
\end{equation*}
\begin{equation*}
(\CO_{S\times X}(-\langle\la^{\lvee},\nu\rangle\cdot (S\times x)) \xrightarrow{\eta_{\la^{\lvee}}} \CV^{\la^{\lvee}}_{\CF_1} \xrightarrow{\varphi_{\la^{\lvee}}} \CV^{\la^{\lvee}}_{\CF_2} \xrightarrow{\zeta_{\la^{\lvee}}} \CO_{S\times X}(-\langle\la^{\lvee},\mu\rangle\cdot (S\times x)),
\end{equation*}
\begin{equation*}
\CO_{S\times X}(-\langle\la^{\lvee},\mu\rangle\cdot (S\times x)) \xrightarrow{\eta'_{\la^{\lvee}}} \CV^{\la^{\lvee}}_{\CF_3} \xrightarrow{\zeta'_{\la^{\lvee}}} \CO_{S\times X})
\end{equation*}
to the point
$(\CF_1,\CF_2,\CF_3,\varphi_{\la^{\lvee}},\tau_{\mu^{\lvee}},\eta'_{\la^{\lvee}}\circ \zeta_{\la^{\lvee}},\eta_{\la^{\lvee}},\zeta'_{\la^{\lvee}}) \in
{\on{Vin}^2\ol{S}{}^{\on{princ}}_{\nu}}|_{\BA^1 \times \{0\}}(S)$:
\begin{equation*}
\CO_{S\times X}(-\langle\la^{\lvee},\nu\rangle\cdot (S\times x)) \xrightarrow{\eta_{\la^{\lvee}}} \CV^{\la^{\lvee}}_{\CF_1} \xrightarrow{\varphi_{\la^{\lvee}}} \CV^{\la^{\lvee}}_{\CF_2} \xrightarrow{\eta'_{\la^{\lvee}} \circ \zeta_{\la^{\lvee}}} \CV^{\la^{\lvee}}_{\CF_3} \xrightarrow{\zeta'_{\la^{\lvee}}} \CO_{S\times X}.
\end{equation*}
It follows that the morphism $\varkappa$ induces the isomorphism from
$\ol{\mathfrak{Y}}{}^{\nu,\mu,\on{princ}} \times \ol{S}_{\mu}$ to the closure of
$(\ol{T_\mu}\cap \ol{S}_\nu)\times\ol{S}_\mu\times \BG_m$ in
the family $(\Upsilon^{2,\on{princ}}_\nu)^{-1}(\BA^1\times\{0\})$.

To prove b) we construct a closed embedding
$\varpi\colon (\ol{T}_{\mu} \cap \ol{S}_{\nu})\times \on{Vin}\ol{S}{}^{\on{princ}}_\mu
\hookrightarrow {\on{Vin}^2\ol{S}{}^{\on{princ}}_{\nu}}|_{\{0\} \times \BA^1}$
of families over $\BA^1$. It sends an $S$-point
\begin{equation*}
((\CF_1,\eta_{\la^{\lvee}},\zeta_{\la^{\lvee}}),(\CF_2,\CF_3,\varphi_{\la^{\lvee}},\tau_{\mu^{\lvee}},\eta'_{\la^{\lvee}},
\zeta'_{\la^{\lvee}}))
\in ((\ol{T}_{\mu} \cap \ol{S}_{\nu})\times \on{Vin}\ol{S}{}^{\on{princ}}_\mu)(S):
\end{equation*}
\begin{equation*}
(\CO_{S\times X}(-\langle\la^{\lvee},\nu\rangle\cdot (S\times x)) \xrightarrow{\eta_{\la^{\lvee}}} \CV^{\la^{\lvee}}_{\CF_1} \xrightarrow{\zeta_{\la^{\lvee}}} \CO_{S\times X}(-\langle\la^{\lvee},\mu\rangle\cdot (S\times x)),
\end{equation*}
\begin{equation*}
\CO_{S\times X}(-\langle\la^{\lvee},\mu\rangle\cdot (S\times x)) \xrightarrow{\eta'_{\la^{\lvee}}} \CV^{\la^{\lvee}}_{\CF_2} \xrightarrow{\varphi_{\la^{\lvee}}} \CV^{\la^{\lvee}}_{\CF_3} \xrightarrow{\zeta'_{\la^{\lvee}}} \CO_{S\times X})
\end{equation*}
to the point
$(\CF_1,\CF_2,\CF_3,\eta'_{\la^{\lvee}}\circ \zeta_{\la^{\lvee}},\varphi_{\la^{\lvee}},\eta_{\la^{\lvee}},\zeta'_{\la^{\lvee}}):$
\begin{equation*}
\CO_{S\times X}(-\langle\la^{\lvee},\nu\rangle\cdot (S\times x)) \xrightarrow{\eta_{\la^{\lvee}}} \CV^{\la^{\lvee}}_{\CF_1} \xrightarrow{\eta'_{\la^{\lvee}} \circ \zeta_{\la^{\lvee}}} \CV^{\la^{\lvee}}_{\CF_2} \xrightarrow{\varphi_{\la^{\lvee}}} \CV^{\la^{\lvee}}_{\CF_3} \xrightarrow{\zeta'_{\la^{\lvee}}} \CO_{S\times X}.
\end{equation*}
Here we use the identification of the scheme $\ol{T}_{\mu} \cap \ol{S}_{\nu}$ with
$\ol\fZ{}^{\nu,\mu}$. We see that the morphism $\varpi$ induces an isomorphism from
$(\ol{T}_{\mu} \cap \ol{S}_{\nu})\times \on{Vin}\ol{S}{}^{\on{princ}}_\mu$
to the closure of the family
$
(\ol{T_\mu}\cap \ol{S}_\nu)\times\ol{S}_\mu\times \BC^\times
\ra \BC^\times
$
in the family
$
(\Upsilon^{2,\on{princ}}_\nu)^{-1}(\{0\}\times\BA^1) \ra \BA^{1}.
$
\end{proof}

Let us consider the cospecialization morphism
\begin{equation*}
H_c^{\langle2\rho^{\lvee},\nu\rangle}((\on{Vin}^2\ol{S}_{\nu})_{(0,0)},(\widetilde{\CP}^2_\nu)_{(0,0)}) \ra
H_c^{\langle2\rho^{\lvee},\nu\rangle}((\on{Vin}^2\ol{S}_{\nu})_{(1,1)},(\widetilde{\CP}^2_\nu)_{(1,1)}).
\end{equation*}
From Corollary~\ref{fibers of two-param} it follows that
$H_c^{\langle2\rho^{\lvee},\nu\rangle}((\on{Vin}^2\ol{S}_{\nu})_{(0,0)},(\widetilde{\CP}^2_\nu)_{(0,0)})=\\
\bigoplus\limits_{\nu \geq \mu_1 \geq \mu_2}
H_c^{\langle2\rho^{\lvee},\nu-\mu_1\rangle}
(\ol{S}_\nu\cap \ol{T}_{\mu_1},\ul\BC{})\otimes
H_c^{\langle2\rho^{\lvee},\mu_1-\mu_2\rangle}
(\ol{S}_{\mu_1}\cap \ol{T}_{\mu_2},\ul\BC{})\otimes
H_c^{\langle2\rho^{\lvee},\mu_2\rangle}(\ol{S}_{\mu_2},\CP)$, and
\begin{equation*}
H_c^{\langle2\rho^{\lvee},\nu\rangle}((\on{Vin}^2\ol{S}_{\nu})_{(1,1)},(\widetilde{\CP}^2_\nu)_{(1,1)})=
H_c^{\langle2\rho^{\lvee},\nu\rangle}(\ol{S}_\nu,\CP).
\end{equation*}
Thus we get the morphisms

\begin{equation*}
H_c^{\langle2\rho^{\lvee},\nu-\mu_1\rangle}
(\ol{S}_\nu\cap \ol{T}_{\mu_1},\ul\BC{})\otimes
H_c^{\langle2\rho^{\lvee},\mu_1-\mu_2\rangle}
(\ol{S}_{\mu_1}\cap \ol{T}_{\mu_2},\ul\BC{})\otimes
H_c^{\langle2\rho^{\lvee},\mu_2\rangle}(\ol{S}_{\mu_2},\CP) \ra
H_c^{\langle2\rho^{\lvee},\nu\rangle}(\ol{S}_\nu,\CP).
\end{equation*}

Note that the following diagram is commutative:
\begin{equation} \label{associat}
\xymatrix{
H_c^{\langle2\rho^{\lvee},\nu\rangle}((\ol{S}_\nu\cap \ol{T}_{\mu_1})
\times (\ol{S}_{\mu_1}\cap \ol{T}_{\mu_2})
\times \ol{S}_{\mu_2},\ul\BC{}\boxtimes\ul\BC{}\boxtimes\CP) \ar[r] \ar[d] &
H_c^{\langle2\rho^{\lvee},\nu\rangle}((\ol{S}_\nu\cap \ol{T}_{\mu_2})
\times \ol{S}_{\mu_2},\ul\BC{}\boxtimes\CP) \ar[d]\\
H_c^{\langle2\rho^{\lvee},\nu\rangle}((\ol{S}_\nu\cap \ol{T}_{\mu_1})
\times \ol{S}_{\mu_1},\ul\BC{}\boxtimes\CP) \ar[r] &
H_c^{\langle2\rho^{\lvee},\nu\rangle}(\ol{S}_\nu,\CP).
}
\end{equation}
where the morphisms in the diagram are the cospecialization morphisms.
Indeed, both compositions are equal to the cospecialization morphism from the fiber
over $(0,0)$ to the fiber over $(1,1)$.

From Lemma~\ref{closures of fam} and Lemma~\ref{identification of local models} it follows
that the diagram (\ref{associat}) is equal to the diagram~(\ref{assoc of action}).
Hence~(\ref{assoc of action}) is commutative and Proposition~\ref{associ} is proved.
\end{proof}

\subsection{Compatibility of the coproduct with the tensor structure}
\label{tensor compat}
Let $\CP,\CQ\in\on{Perv}_{G_\CO}(\Gr_G)$. We set $V:=H^\bullet(\Gr_G,\CP),\
W:=H^\bullet(\Gr_G,\CQ)\in \on{Rep}(G^\vee)$.
We need to check that the diagram
\begin{equation} \label{diagram for compat with tensor}
\xymatrix
{\CA \otimes V \otimes W \ar[rrrr]^{\on{act}_{V\otimes W}}
\ar[d]^{\Delta\otimes\on{Id}} &&&& V \otimes W \ar[d]^{\on{Id}}\\
\CA \otimes \CA \otimes V \otimes W \ar[rrrr]^{(\on{act}_V\otimes\on{act}_W)
\circ(\on{Id}\otimes\uptau\otimes\on{Id})}
&&&& V \otimes W
}
\end{equation}
commutes, where the morphism
$\uptau\colon \CA \otimes V \ra V \otimes \CA$
sends $a\otimes b$ to $b \otimes a$.

Fix $\theta \in \La$. Let $\on{Vin}\ol{S}{}^{\on{princ}}_{\theta,X}$
be the following moduli space: it
associates to a scheme $S$

1) an $S$-point $x \in X(S)$ of the curve $X$,

2) an $S$-point
$(\CF_1, \CF_2,\varphi_{\la^{\lvee}},\tau_{\mu^{\lvee}})$
of $\on{VinBun}^{\on{princ}}_G$,

3) for every $\la^{\lvee} \in \La^{\vee+}$,  morphisms of sheaves
$\eta_{\la^{\lvee}}$,
\begin{equation*}
\eta_{\la^{\lvee}}\colon
\CO_{S\times X}(-\langle\la^{\lvee},\theta\rangle\cdot\varGamma_{x})
\ra \CV^{\la^{\lvee}}_{\CF_{1}}
\end{equation*}
and rational  morphisms
$
\zeta_{\la^{\lvee}}\colon\CV^{\la^{\lvee}}_{\CF_{2}} \ra
\CO_{S\times X}
$
regular on $(S\times X) \setminus
\{\varGamma_{x}\},$
satisfying the same conditions as in~\S\ref{degen of the affine Gr}.

Fix $\theta_1, \theta_2 \in \La$. Let $\on{Vin}\ol{S}{}^{\on{princ}}_{\theta_1,\theta_2}$
be the following moduli space: it
associates to a scheme $S$

1) a pair of $S$-points $(x_1,x_2) \in X^2(S)$ of the curve $X$,

2) an $S$-point
$(\CF_1, \CF_2,\varphi_{\la^{\lvee}},\tau_{\mu^{\lvee}})
\in \on{VinBun}^{\on{princ}}_G(S)$
of $\on{VinBun}^{\on{princ}}_G$,

3) for every $\la^{\lvee} \in \La^{\vee+}$,  morphisms of sheaves
$\eta_{\la^{\lvee}}$,
\begin{equation*}
\eta_{\la^{\lvee}}\colon
\CO_{S\times X}(-\langle\la^{\lvee},\theta_1\rangle\cdot\varGamma_{x_1}
-\langle\la^{\lvee},\theta_2\rangle\cdot\varGamma_{x_2})
\ra \CV^{\la^{\lvee}}_{\CF_{1}}
\end{equation*}
and rational  morphisms
$
\zeta_{\la^{\lvee}}\colon\CV^{\la^{\lvee}}_{\CF_{2}} \ra
\CO_{S\times X}
$
regular on $(S\times X) \setminus
\{\varGamma_{x_1} \cup \varGamma_{x_2}\},$
satisfying the same conditions as in~\S\ref{degen of the affine Gr}.

\bigskip
We have a projection
$\pi^{\on{Vin}}_{\theta_1,\theta_2}\colon
\on{Vin}\ol{S}{}_{\theta_1,\theta_2}^{\on{princ}} \ra X^2$
that forgets the data of
$\CF_1,\CF_2,\varphi_{\la^{\lvee}},\tau_{\mu^{\lvee}},\eta_{\la^{\lvee}},\zeta_{\la^{\lvee}}$.
Let
\begin{equation*}
\vartheta^{\on{princ}}_{\theta_1,\theta_2}\colon\on{Vin}\ol{S}{}^{\on{princ}}_{\theta_1,\theta_2}
\hookrightarrow
\Gr_{G,X^2}\underset{X^2}\times\Gr_{G,X^2}\times\BA^1
\end{equation*}
be the restriction of the closed embedding $\vartheta$ of \S\ref{embedding}
to
$\on{Vin}\ol{S}{}^{\on{princ}}_{\theta_1,\theta_2}
\subset \on{VinGr}_{G,X^2}$.
The morphism
$
\Upsilon^{\on{princ}}_{\theta_1,\theta_2}\colon
\on{Vin}\ol{S}{}^{\on{princ}}_{\theta_1,\theta_2}
\ra \BA^1
$
is defined as the composition of morphisms
$
\on{Vin}\ol{S}{}^{\on{princ}}_{\theta_1,\theta_2}\hookrightarrow
\on{VinGr}^{\on{princ}}_{G,X^2} \ra \BA^1.
$
\sssec{The three-parameter deformation}

Let us consider the morphism
\begin{equation*}
\Upsilon^{\on{princ}}_{\theta_1,\theta_2}\times
\pi_{\theta_1,\theta_2}\colon
\on{Vin}\ol{S}{}^{\on{princ}}_{\theta_1,\theta_2}\ra
X^2\times \BA^1.
\end{equation*}
Recall the closed embedding
\begin{equation*}
\vartheta^{\on{princ}}_{\theta_1,\theta_2}\colon\on{Vin}\ol{S}{}^{\on{princ}}_{\theta_1,\theta_2}
\hookrightarrow
\Gr_{G,X^2}\underset{X^2}\times\Gr_{G,X^2}\times\BA^1.
\end{equation*}
Recall the sheaf $\CP\underset{X}\star\CQ$ on $\Gr_{G,X^2}$ of
\S\ref{tensor structure}.
Set
\begin{equation*}
(\widetilde{\CP}\underset{X}\star\widetilde{\CQ})
_{\theta_1,\theta_2}:=
(\ul\BC{}\boxtimes(\CP\underset{X}\star\CQ)
\boxtimes\ul\BC{})|_{\on{Vin}\ol{S}{}^{\on{princ}}_{\theta_1,\theta_2}}.
\end{equation*}
It follows from Corollary~\ref{fibers for fusion} below that the complex
$(\widetilde{\CP}\underset{X}\star\widetilde{\CQ})
_{\theta_1,\theta_2}$
has finite-dimensional support.
Recall the embeddings
$
\Delta_{X} \hookrightarrow X^2
\hookleftarrow U
$
of \S\ref{tensor structure}.
\begin{lem} \label{restrictions}
a) The restriction of the morphism
$\vartheta^{\on{princ}}_{\theta_1,\theta_2}$
of the families over $X^2\times \BA^1$
\begin{equation*}
\xymatrix{
\on{Vin}\ol{S}{}^{\on{princ}}_{\theta_1,\theta_2} \ar[dr] \ar[rr]^{\vartheta^{\on{princ}}_{\theta_1,\theta_2}}
&& \ar[dl]
\Gr_{G,X^2}\underset{X^2}\times\Gr_{G,X^2}\times\BA^1\\
& X^2\times\BA^1 &
}
\end{equation*}
to the open subvariety $U\times\BA^1 \subset X^2\times \BA^1$
is isomorphic to
\begin{equation*}
\xymatrix{
(\on{Vin}\ol{S}{}^{\on{princ}}_{\theta_1,X}\underset{\BA^1}\times \on{Vin}\ol{S}{}^{\on{princ}}_{\theta_2,X})|_{(U\times\BA^1)} \ar[dr]^{} \ar[rr]^{\vartheta^{\on{princ}}_{\theta_1,\theta_2}|_{(U \times \BA^1)}}
&& \ar[dl]
((\Gr_{G,X}\times\Gr_{G,X})|_U\underset{U}\times(\Gr_{G,X}\times\Gr_{G,X})|_U) \times\BA^1 \\
& U \times \BA^1 &
}
\end{equation*}
where the morphism $\vartheta^{\on{princ}}_{\theta_1,\theta_2}|_{(U \times \BA^1)}$ is given by
\begin{equation*}
(\vartheta_{\theta_1}^{(1)},\vartheta_{\theta_2}^{(1)},
\vartheta_{\theta_1}^{(2)},\vartheta_{\theta_2}^{(2)},\vartheta_{\theta_1}^{(3)}=\vartheta_{\theta_2}^{(3)}),
\end{equation*} and
$\vartheta_{\theta_1}^{(i)},\vartheta_{\theta_2}^{(i)}$
are the corresponding components of the morphisms
\begin{equation*}
\vartheta^{\on{princ}}_{\theta_1}\colon
\on{Vin}\ol{S}{}^{\on{princ}}_{\theta_1,X}
\hookrightarrow \Gr_{G,X}\underset{X}\times\Gr_{G,X}\times\BA^1,~
\vartheta^{\on{princ}}_{\theta_2}\colon
\on{Vin}\ol{S}{}^{\on{princ}}_{\theta_2,X}
\hookrightarrow \Gr_{G,X}\underset{X}\times\Gr_{G,X}\times\BA^1.
\end{equation*}

b) The restriction of the morphism
$\vartheta^{\on{princ}}_{\theta_1,\theta_2}$
of the families over $X^2\times \BA^1$
to the closed subvariety $\Delta_X\times\BA^1 \subset X^2\times \BA^1$
is isomorphic to
\begin{equation*}
\xymatrix{
\on{Vin}\ol{S}{}^{\on{princ}}_{\theta_1+\theta_2,X} \ar[dr]^{} \ar[rr]^{\vartheta^{\on{princ}}_{\theta_1,\theta_2}|_{(\Delta_X \times \BA^1)}}
&& \ar[dl]
\Gr_{G,X}\underset{X}\times\Gr_{G,X} \times\BA^1 \\
& \Delta_X \times \BA^1 &
}
\end{equation*}
where the morphism $(\vartheta^{\on{princ}}_{\theta_1,\theta_2})|_{\Delta_X \times \BA^1}$
coincides with the morphism $\vartheta^{\on{princ}}_{\theta_1+\theta_2}$.

\begin{comment} c) The restriction of the morphism
$\vartheta^{\on{princ}}_{\theta_1,\theta_2}$
of the families over $X^2\times \BA^1$
to the open subvariety $X^2\times\BG_m \subset X^2\times \BA^1$
is isomorphic to
\begin{equation*}
\xymatrix{
\ol{S}_{\theta_1,\theta_2}\times \BG_m \ar[dr]^{} \ar[rr]^{\vartheta^{\on{princ}}_{\theta_1,\theta_2}|_{(X^2 \times \BG_m)}}
&& \ar[dl]
\Gr_{G,X^2}\underset{X^2}\times\Gr_{G,X^2} \times\BG_m \\
& X^2 \times \BG_m &
}
\end{equation*}
where the morphism $(\vartheta^{\on{princ}}_{\theta_1,\theta_2})|_{X^2 \times \BG_m}$
is given by
$
(P,c) \mapsto (P,2\rho(c)\cdot P,c).
$
Here $(P,c)$ is a point of $\ol{S}_{\theta_1,\theta_2}\times \BG_m$.

d) The restriction of the morphism
$\vartheta^{\on{princ}}_{\theta_1,\theta_2}$
of the families over $X^2\times \BA^1$
to the closed subvariety $X^2\times\{0\} \subset X^2\times \BA^1$
is isomorphic to
\begin{equation*}
\xymatrix{
\bigcup\limits_{\mu_1\leq\theta_1,\mu_2\leq\theta_2}(\ol{T}_{\mu_1,\mu_2} \underset{\Gr_{G,X^2}}\times \ol{S}_{\theta_1,\theta_2}) \underset{X^2}\times \ol{S}_{\mu_1,\mu_2} \ar[dr]^{} \ar[rr]^{\vartheta^{\on{princ}}_{\theta_1,\theta_2}|_{(X^2 \times \{0\})}}
&& \ar[dl]
\Gr_{G,X^2}\underset{X^2}\times\Gr_{G,X^2} \\
& X^2 \times \{0\} &
}
\end{equation*}
where the morphism $(\vartheta^{\on{princ}}_{\theta_1,\theta_2})|_{(X^2 \times \{0\})}$
is the natural embedding.
\end{comment}
\end{lem}
\begin{proof}
It follows from Proposition~\ref{factor for VinGr}.
\begin{comment}
, c) follows from Lemma~\ref{restriction for VinGrBD}, the proof of d) is the same as the proof of Lemma~\ref{special fiber of vartheta}.
\end{comment}
\end{proof}

\begin{cor} \label{fibers for fusion}
Under the identifications of Lemma~\ref{restrictions}, Proposition~\ref{restrictions of Upsilon} we have

a)
$
{((\widetilde{\CP}\underset{X}\star\widetilde{\CQ})
_{\theta_1,\theta_2})_{((x,x),0)}}|
_{(\ol{T}_\mu\cap\ol{S}_{\theta_1+\theta_2})\times\ol{S}_\mu}=
\ul\BC{}\boxtimes(\CP\star\CQ)_\mu
$
for $x \in X$.

b)
$
{((\widetilde{\CP}\underset{X}\star\widetilde{\CQ})
_{\theta_1,\theta_2})_{((x,y),0)}}|
_{((\ol{T}_{\mu_1}\cap\ol{S}_{\theta_1})\times(\ol{T}_{\mu_2}\cap\ol{S}_{\theta_2}))
\times
\ol{S}_{\mu_1}\times\ol{S}_{\mu_2}}=
\ul\BC{}\boxtimes\ul\BC{}\boxtimes\CP_{\mu_1}\boxtimes\CQ_{\mu_2}
$
for $(x,y) \in U$.

c)
$
{((\widetilde{\CP}\underset{X}\star\widetilde{\CQ})
_{\theta_1,\theta_2})_{((x,x),1)}}=
(\CP\star\CQ)_{\theta_1+\theta_2}
$
for $x \in X$.

d)
$
{((\widetilde{\CP}\underset{X}\star\widetilde{\CQ})
_{\theta_1,\theta_2})_{((x,y),1)}}=
\CP_{\theta_1}\boxtimes \CQ_{\theta_2}
$
for $(x,y) \in U$.
\end{cor}
\sssec{Proof of compatibility with the tensor structure}
Let us fix two distinct points $x,y \in X$.
Let us consider the cospecialization morphism
\begin{equation*}
\xymatrix{
H_c^{\langle2\rho^{\lvee},\theta_1+\theta_2\rangle}((\on{Vin}\ol{S}{}^{\on{princ}}_{\theta_1,\theta_2})_{((x,x),0)},
((\widetilde{\CP}\underset{X}\star\widetilde{\CQ})
_{\theta_1,\theta_2})_{((x,x),0)}) \ar[d]\\
H_c^{\langle2\rho^{\lvee},\theta_1+\theta_2\rangle}((\on{Vin}\ol{S}{}^{\on{princ}}_{\theta_1,\theta_2})_{((x,y),1)},
((\widetilde{\CP}\underset{X}\star\widetilde{\CQ})
_{\theta_1,\theta_2})_{((x,y),1)})
}
\end{equation*}
from the point $((x,x),0)$ to the point $((x,y),1)$.
It may be obtained in two ways:
by first cospecializing from $((x,x),0)$ to $((x,x),1)$ and then to $((x,y),1)$
or first cospecializing from $((x,x),0)$ to $((x,y),0)$ and then to
$((x,y),1)$.
For $\nu_1,\nu_2\in \La$, let $\on{p}_{\nu_1,\nu_2}\colon (V\otimes W)_{\nu_1+\nu_2} \twoheadrightarrow V_{\nu_1}\otimes W_{\nu_2}$ be the natural projection. Using Corollary~\ref{fibers for fusion} and Lemma~\ref{restrictions} we see that the following
diagram commutes:
\begin{equation*}
\xymatrix{
\CA_{\theta_1+\theta_2-\mu} \otimes (V \otimes W)_{\mu} \ar[rrrr]^{\on{act}_{V \otimes W}} \ar[d]^{\oplus\Delta_{\mu_1,\mu_2} \otimes \on{p}_{\theta_1-\mu_1,\theta_2-\mu_2}} &&&& (V \otimes W)_{\theta_1+\theta_2} \ar[d]^{\on{p}_{\theta_1,\theta_2}} \\
\bigoplus\limits_{\mu_1+\mu_2=\theta_1+\theta_2-\mu}\CA_{\mu_1} \otimes \CA_{\mu_2} \otimes V_{\theta_1-\mu_1} \otimes W_{\theta_2-\mu_2} \ar[rrrr]^{\hspace{2,5cm}(\on{act}_V \otimes \on{act}_W) \circ (\on{Id} \otimes \tau \otimes \on{Id})} &&&& V_{\theta_1} \otimes W_{\theta_2}
}.
\end{equation*}
Summing over all $\theta_1,\theta_2,\mu \in \La$ such that $\mu \leq \theta_1+\theta_2$ we get the diagram
~(\ref{diagram for compat with tensor}).

\begin{lem} \label{action nontriv}
Let us fix $\theta \in \Lambda^\pos$, $a \in \CA_{\theta}\setminus\{0\}$. There exists $\nu \in \La^{+}$ such that the operator $\on{act}_{V^{\nu}}(a)$ is nonzero.
\end{lem}

\begin{proof}
It follows from~\cite[Proposition~6.4]{bfgm} that there exists
$\la \in \La^{+}$ such that $\ol{S}_{\theta-\la}\cap \ol{T}_{-\la}$
is contained inside $\ol\Gr{}^{-w_0(\la)}_G \cap \ol{S}_{\theta-\la}$. 
It follows that
$V^{-w_{0}(\la)}_{-\la+\theta}=
H^{\langle2\rho^{\lvee},\theta\rangle}_{c}(\ol\Gr{}^{-w_0(\la)}_G \cap \ol{S}_{\theta-\la},\ul\BC{})$
embeds into
$\CA_{\theta}=H^{\langle2\rho^{\lvee},\theta\rangle}_{c}(\ol{S}_{\theta-\la}\cap \ol{T}_{-\la},\ul\BC{})$.
Indeed, the latter space has a basis formed by the fundamental classes of irreducible components
of $\ol{S}_{\theta-\la}\cap \ol{T}_{-\la}$, while the former space has a basis formed by the
fundamental classes of irreducible components of $\ol\Gr{}^{-w_0(\la)}_G \cap \ol{S}_{\theta-\la}$,
and each of those irreducible components is an irreducible component of
$\ol\Gr{}^{-w_0(\la)}_G \cap \ol{S}_{\theta-\la}$.
Thus the vertical arrows of the following commutative diagram are embeddings:
\begin{equation}\label{act vs mult}
\xymatrix{\CA_{\theta}\otimes V^{-w_0(\la)}_{-\la} \ar[rrr]^{\on{act}_{V^{-w_0(\la)}}} \ar[d] &&&
V^{-w_0(\la)}_{-\la+\theta}\ar[d] \\
\CA_{\theta} \otimes \CA_{0} \ar[rrr]^{\bf{m}} &&& \CA_{\theta}.
}
\end{equation}
Commutativity of the diagram (\ref{act vs mult}) follows from the definitions of $\on{act}, {\bf{m}}$ using the identifications
$\ol{S}_{\theta-\la}\cap \ol{T}_{-\la} \simeq \ol{S}_{\theta} \cap \ol{T}_{0}$,
$\ol{\mathfrak{Y}}^{\theta} \simeq \ol{\mathfrak{Y}}^{\theta-\la,-\la}$ (see Lemma~\ref{identification of local models}) and the natural closed embedding
$\ol{\mathfrak{Y}}^{\theta-\la,-\la} \hookrightarrow \on{Vin}\ol{S}_{\theta-\la}$. Note that ${\bf{m}}(a\otimes 1)=a$ is nonzero. It follows that
the operator $\on{act}_{V^{-w_0(\la)}}(a)$ is nonzero.
\end{proof}

\subsection{Schieder conjecture}
According to~\cite[Question~6.6.1]{s3}, the bialgebra $\CA$ is expected to be isomorphic
to the universal enveloping algebra $U(\nvee)$. We will construct such an isomorphism.

\begin{cor}
  \label{five}
  Schieder bialgebra $\CA$ is isomorphic to $U(\nvee)$.
\end{cor}

\begin{proof}
Let $\CA^{\on{prim}}\subset\CA$ stand for the subspace of primitive elements equipped with the
natural structure of Lie algebra. We will denote this Lie algebra $\fa$.
We have a canonical morphism $\epsilon\colon U(\fa)\to\CA$. We have to check that $\epsilon$
is an isomorphism and to construct an isomorphism $\varepsilon\colon \fa\iso\nvee$.
According to~\S\ref{tensor compat}, $\fa$ acts on the fiber functor
$H^\bullet(\Gr_G,\bullet)\colon \on{Perv}_{G_\CO}(\Gr_G)\to\on{Vect}$, and hence we obtain
a homomorphism of Lie algebras $\fa\to\gvee$, see~\cite[Lemma~23.69]{mil}.
%(to this end, following~\cite{del}, we consider the $\on{Spec}(\BC[x]/x^2)$-points
%of the automorphisms scheme of the fiber functor).
This homomorphism clearly lands into $\nvee\subset\gvee$
(see~\ref{The action}). This is the desired homomorphism $\varepsilon$.

The vector space $\CA^{\on{prim}}$ was computed in~\cite[Proposition~4.2]{ffkm}: it is a
direct sum of lines $\bigoplus_{\alpha\in R^+}\CA^{\on{prim}}_\alpha$. Note that while the definition
of {\em algebra} structure on $\CA$ in~\cite{ffkm} is {\em different} from Schieder's
one~\ref{multiplication}, the definition of {\em coalgebra} structure on $\CA$
in~\cite[2.11]{ffkm} is manifestly the same as Schieder's one~\ref{comultiplication}.
Thus the character of the $\Lambda^\pos$-graded vector space $\fa=\bigoplus_{\alpha\in R^+}\fa_\alpha$
coincides with the character of $\nvee$ as elements in $\BZ[\Lambda^\pos]$.

It follows from Lemma~\ref{action nontriv} that for $i\in I$, the action of $\fa_{\alpha_i}$ on the
fiber functor is not trivial, and
hence $\varepsilon(\fa_{\alpha_i})\ne0$, that is $\varepsilon(\fa_{\alpha_i})=\nvee_{\alpha_i}$.
Since $\bigoplus_{i\in I}\nvee_{\alpha_i}$ generates $\nvee$, we see that $\varepsilon$ is
surjective. Since the characters of $\fa$ and $\nvee$ coincide, $\varepsilon$ is an
isomorphism. Since the action of $U(\fa)\simeq U(\nvee)$ on the fiber functor is effective
and factors through the action of $\CA$, we conclude that $\epsilon\colon U(\fa)\to\CA$
is injective. One last comparison of characters shows that $\epsilon$ is an isomorphism,
and hence $U(\nvee)\iso U(\fa)\iso\CA$ (the first arrow is $\varepsilon^{-1}$, the second
one is $\epsilon$).
\end{proof}

\subsection{Comparison with the Ext-algebra of~\cite{ffkm}}
\label{ivan}
There is another construction of the algebra structure in $\CA$ going back to~\cite{ffkm}.
Let us denote it by $\CA\ni a,b\mapsto a\circ b$. In this Section we identify this product
with Schieder's product.

Let us denote the shriek extension of the constant sheaf on $T_\nu\cap\ol\Gr{}_G^\lambda$
by $\CR^\lambda_\nu$. As $\lambda$ grows, these sheaves form an inverse system, and we denote
its inverse limit by $\CR_\nu$. By Braden's Theorem,
$\Ext^{\langle2\rho^{\lvee},\nu-\mu\rangle}(\CR^\lambda_\nu,\CR^\lambda_\mu)=
H^{\langle2\rho^{\lvee},\nu-\mu\rangle}(\iota_{\nu,-}^*r_{\nu,-}^!r_{\mu,-,!}\underline\BC)=
H^{\langle2\rho^{\lvee},\nu-\mu\rangle}(\iota_{\nu,+}^!r_{\nu,+}^*r_{\mu,-,!}\underline\BC)=
H^{\langle2\rho^{\lvee},\nu-\mu\rangle}_c(T_\mu\cap S_\nu\cap\ol\Gr{}_G^\lambda)$. As $\lambda$ grows,
this compactly supported cohomology forms an inverse system that eventually stabilizes;
the stable value will be denoted by $\Ext^{\langle2\rho^{\lvee},\nu-\mu\rangle}(\CR_\nu,\CR_\mu)=
H^{\langle2\rho^{\lvee},\nu-\mu\rangle}_c(T_\mu\cap S_\nu)=\CA_{\nu-\mu}$.

The composition of the above Ext's defines the desired product
$(a\in\CA_{\theta_1},b\in\CA_{\theta_2})\mapsto a\circ b\in\CA_{\theta_1+\theta_2}$.
Furthermore, for $\CP\in\on{Perv}_{G_\CO}(\Gr_G)$ we have $\Phi_\nu(\CP)=
\Ext^{\langle2\rho^{\lvee},\nu\rangle}(\CR_\nu,\CP)$ (the RHS is again defined as the limit of a
direct system that eventually stabilizes to $\iota_{\nu,-}^*r_{\nu,-}^!\CP$).
Hence the composition of Ext's also defines an action
$(a\in\CA_{\nu-\mu},\phi\in\Phi_\mu(\CP))\to a\circ\phi\in\Phi_\nu(\CP)$.

\begin{lem}
  \label{mirko}
  \textup{(a)} The $\circ$-product on $\CA$ coincides with Schieder's product $\mathbf m$.

  \textup{(b)} The $\circ$-action of $\CA$ on $\Phi$ coincides with the action
  of~\ref{The action}.
\end{lem}

\begin{proof}
  (a) follows from (b). To prove (b), due to the fact that $\Phi_\mu$ is represented by
  $\CR_\mu$, it suffices to compare the two actions of $\CA_{\nu-\mu}$ on
  $\on{Id}\in\Phi_\mu(\CR_\mu)$. That is we have to check that the cospecialization
  morphism~\ref{The action} $\CA_{\nu-\mu}=\CA_{\nu-\mu}\otimes\BC=
  \CA_{\nu-\mu}\otimes H^0_c(S_\mu,\CR_\mu)\to H^{\langle2\rho^{\lvee},\nu-\mu\rangle}_c(S_\nu,\CR_\mu)
  =H^{\langle2\rho^{\lvee},\nu-\mu\rangle}_c(S_\nu\cap T_\mu)=
  H^{\langle2\rho^{\lvee},\nu-\mu\rangle}_c(\ol{S}_\nu\cap T_\mu)=\CA_{\nu-\mu}$ is the identity
  morphism. This follows from the fact that the Drinfeld-Gaitsgory interpolation of
  $\ol{S}_\nu\cap T_\mu$ is trivial.
\end{proof}

\subsection{Integral form}
\label{integra}
Note that the bialgebra $\CA$ comes equipped with a natural basis of
fundamental classes of irreducible components of $\ofZ^\theta$, and the
structure constants of multiplication in this basis belong to $\BZ$.
Hence $\CA$ acquires an integral form $\CA_\BZ$. The algebra $U(\nvee)$
also has the integral Chevalley-Kostant form $U(\nvee)_\BZ$, see
e.g.~\cite[Chapter~2]{st}.

\begin{prop}
  \label{integr}
  The isomorphism $U(\nvee)\iso\CA$ of~Corollary~\ref{five} induces
  an isomorphism $U(\nvee)\supset U(\nvee)_\BZ\iso\CA_\BZ\subset\CA$.
\end{prop}

\begin{proof}
  We consider the perverse sheaves with integral coefficients
  $\CI_!(\lambda,\BZ)\iso\CI_{!*}(\lambda,\BZ)$ of~\cite[Lemma~11.5]{bari}.
  The fiber functor applied to them gives the integral Weyl modules
  $\Phi(\CI_!(\lambda,\BZ))\simeq V^\lambda_\BZ\in\Rep(G^\vee_\BZ)$.
  These are free $\BZ$-modules with the bases given by the fundamental
  classes of the irreducible components of $\Gr_G^\lambda\cap S_\nu$,
  see~\cite[Proposition~11.1]{bari}. For a fixed $\theta\in\Lambda^\pos$, and a
  dominant coweight $\lambda\gg0$,
  the action of $\CA$ on the fundamental class of
  $\Gr^\lambda\cap S_{w_0(\lambda)}$ induces an isomorphism of the weight
  spaces $\CA_\theta\iso\Phi_{\theta+w_0\lambda}(\on{IC}(\ol\Gr{}^\lambda_G))
  \simeq V^\lambda_{\theta+w_0\lambda}$. Moreover, it follows
  from the proof of~Lemma~\ref{action nontriv} that this isomorphism respects
  the integral forms:
  $\CA_{\BZ,\theta}\iso V^\lambda_{\BZ,\theta+w_0\lambda}$.
  However, according to~\cite[Corollary~1 of~Theorem~2]{st}, the action of
  $U(\nvee)$ on the lowest weight vector of $V^\lambda_\BZ$ also induces
  an isomorphism of the weight spaces respecting the integral forms
  $U(\nvee)_{\BZ,\theta}\iso V^\lambda_{\BZ,\theta+w_0\lambda}$.

  The proposition follows.
\end{proof}

\section{Arbitrary reductive groups and symmetric Kac-Moody Lie algebras}
\label{5five}

\subsection{Configurations}
In general, we may have $\Lambda^\pos:=\bigoplus_{i\in I}\BN\alpha_i\subsetneqq
\Lambda_{\geq0}:=\{\alpha\in\Lambda : \langle\lambda^{\!\scriptscriptstyle\vee},\alpha\rangle\geq0\
\forall \lambda^{\!\scriptscriptstyle\vee}\in\Lambda^{\vee+}\}$. It happens e.g.\ for $G=\on{PGL}_2$.
Recall that for $\theta\in\Lambda_{\geq0}$, a point $D\in X^\theta$
is a collection of effective divisors
$D_{\la^{\lvee}} \in X^{(\langle \la^{\lvee}, \theta\rangle)}$ for $\la \in \La^{\vee+}$ such that for every
$\la^{\lvee}_1, \la^{\lvee}_2 \in \La^{\vee+}$ we have
$D_{\la^{\lvee}_1}+D_{\la^{\lvee}_2}=D_{\la^{\lvee}_1+\la^{\lvee}_2}$.
For $\theta=\sum\limits_{i\in I} n_i\al_i\in \Lambda^\pos$ we define
$X^{(\theta)}=\prod\limits_{i\in I}X^{(n_i)}.$ We have a closed (``diagonal'') embedding
$X^{(\theta)}\hookrightarrow X^\theta$ (it fails to be an isomorphism e.g.\ for $G=\on{PGL}_2$).

The complement $X^\theta\setminus X^{(\theta)}$ accounts for some undesirable components of
$\on{VinBun}_G$ if we use the naive definition for arbitrary $G$.
The reason lies in the ``bad'' components of $\ol{\on{Bun}}_B$. Schieder explains in~\cite[\S7]{s0}
how to get rid of the ``bad'' components using a central extension
\begin{equation}
  \label{extension}
  1\to\CZ\to\widehat{G}\to G\to1
\end{equation}
such that $\CZ$ is a (connected) central torus in $\widehat{G}$, and the derived subgroup
$[\widehat{G},\widehat{G}]\subset\widehat{G}$ is simply connected.
The same remedy works in our setup.

\subsection{Example for $G=\on{PGL}_2$}
We identify $\Lambda$ with $\BZ$, so that the simple root $\alpha$ corresponds to~2.
We have $\Lambda^\pos=2\BN,\ \Lambda_{\geq0}=\BN$. Let $X=\BP^1$ with homogeneous coordinates $x,y$.
We identify $G=\on{PGL}_2=\on{SO}_3$
and consider the fiber $\fB$ of $\ol{\on{Bun}}{}_B^1$ over the $G$-bundle
$\CV=\CO(-1)\oplus\CO\oplus\CO(1)$ (with the evident symmetric self-pairing $\CV\otimes\CV\to\CO$).
In other words, $\fB$ is the projectivization of
$\Hom^0(\CO(-1),\CV)$, where $\Hom^0$ is formed by the isotropic homomorphisms.
Thus $\Hom^0(\CO(-1),\CV)=\{(P_0,P_1,P_2) : P_0P_2-P_1^2=0\}$, where $P_i$ is a homogeneous
polynomial in $x,y$ of degree $i$. Clearly, $\Hom^0(\CO(-1),\CV)$ consists of 2 irreducible
3-dimensional components: the first one given by $P_0=P_1=0$, and the second one given by
(away from the first one) $P_0\ne0,\ P_2=P_1^2P_0^{-1}$. We denote the projectivization of the
first (resp.\ second) component by $\fB_{\on{bad}}$ (resp.\ $\fB_{\on{good}}$).

We also consider the ``shifted zastava'' space $Z^{1,-1}$: the open part of $\fB$ formed by
those $\CO(-1)\hookrightarrow\CV$ whose composition with the projection
$\CV\twoheadrightarrow\CO(1)$ is nonzero. In other words, it is given by the condition
$P_2\ne0$. The factorization projection $Z^{1,-1}\to X^\alpha$ is given by
$(P_0,P_1,P_2)\mapsto P_2$. We see that $Z_{\on{bad}}^{1,-1}$ projects isomorphically onto
$X^\alpha$, while $Z_{\on{good}}^{1,-1}$ projects onto $X^{(\alpha)}\hookrightarrow X^\alpha$.

The existence of ``bad'' components in the zastava spaces implies the existence of
``bad'' components in the local models $Y^\theta$ of~\S\ref{local Y}, and hence the existence
of ``bad'' components in $\on{VinBun_G}$.

\subsection{Modified constructions}
Let us redenote $\ol{\on{Bun}}_B$ by $\ol{\on{Bun}}{}_B^{\on{naive}}$.
In order to get rid of the ``bad'' components of $\ol{\on{Bun}}{}_B^{\on{naive}}$,
Schieder modified the definition as
\begin{equation*}
  \ol{\on{Bun}}_B:=\on{Maps}_{\on{gen}}(X, G\backslash\ol{\widehat{G}/N}/\widehat{T}\supset
  G\backslash(\widehat{G}/N)/\widehat T=\on{pt}\!/\!B),
\end{equation*}
where $\widehat{G}$ is chosen as
in~(\ref{extension}), and $1\to\CZ\to\widehat{T}\to T\to1$ is the corresponding extension
of Cartan tori. Schieder proved that $\ol{\on{Bun}}_B$ is canonically independent of
the choice of central extension~(\ref{extension}), i.e.\ $\ol{\on{Bun}}_B$ is well defined.
Note also that $G\backslash\ol{\widehat{G}/N}/\widehat{T}=
G\backslash(\ol{\widehat{G}/N}/\CZ)/T$.

Following Schieder, we modify~Definition~\ref{Vinberg via Rees} of $\on{Vin}_G^{\on{naive}}$
as $\on{Vin}_G:=\on{Vin}_{\widehat{G}}/\CZ$. Accordingly, we
modify~Definition~\ref{vinbun} of $\on{VinBun}_G^{\on{naive}}$ as
\begin{equation*}
  \on{VinBun}_G:=\on{Maps}_{\on{gen}}(X, \on{Vin}_G/(G\times G)\supset{}
  _0\!\!\on{Vin}_G/(G\times G))
\end{equation*}
(instead of $\on{Maps}_{\on{gen}}(X, \on{Vin}_G^{\on{naive}}/(G\times G)\supset{}
_0\!\!\on{Vin}_G^{\on{naive}}/(G\times G))$).

With this understanding, the definition of $\on{VinGr}_{G,X^n}$
of~\S\ref{second def VinGrBD}
stays intact, as well as the construction of the action of Schieder bialgebra
$\CA$ on the geometric Satake fiber functor.

\subsection{A Coulomb branch construction}
\label{Coulomb}
In case $G$ is almost simple simply laced, the relative compactification
$\ol{Z}{}^\theta\supset Z^\theta$~\cite[\S7.2]{gai} of zastava space for $X=\BA^1$
was constructed in~\cite[Remark~3.7]{bfn} in the course of study of Coulomb branches of
$3d\ \CN=4$ supersymmetric quiver gauge theories. The diagonal fibers of the factorization
morphism $\pi\colon \ol{Z}{}^\theta\to\BA^{(\theta)}$ are identified with
$\ol\fZ{}^{\theta,0}=\ol{S}_\theta\cap\ol{T}_0$, see~Remark~\ref{zastava as intersection}.
The construction of~\cite{bfn} proceeds in terms of $\theta$-dimensional representations
of the Dynkin graph $Q$ of $G$ equipped with an orientation.
Given an arbitrary quiver $Q$ with a set $Q_0$ of vertices and without loop edges, and a
dimension vector $\theta\in\BN^{Q_0}$, this construction produces an affine scheme
$\pi\colon Z^\theta_{\fg_Q}\to\BA^{(\theta)}$ and its relative projectivization
$\ol{\pi}\colon \ol{Z}{}^\theta_{\fg_Q}\to\BA^{(\theta)}$. Here we denote
by $\fg_Q$ the corresponding symmetric Kac-Moody algebra, and we expect $Z^\theta_{\fg_Q}$
(resp.\ $\ol{Z}{}^\theta_{\fg_Q}$)
to play the role of (compactified) zastava space for $\fg_Q$ (cf.~\cite[Remark~3.26.(1)]{bfn}).
In particular, we hope that
$\pi,\ol{\pi}$ enjoy the factorization property, and $\dim\ol\fZ{}^\theta_{\fg_Q}=
\dim\fZ^\theta_{\fg_Q}=|\theta|$, where
$\ol\fZ{}^\theta_{\fg_Q}$ (resp.\ $\fZ^\theta_{\fg_Q}$) stands for the diagonal fiber
$\ol{\pi}{}^{-1}(\theta\cdot0)$ (resp.\ $\pi^{-1}(\theta\cdot0)$).

Furthermore, the construction of~\cite[Remark~3.2]{bfn} equips $Z^\theta_{\fg_Q}$ and
$\ol{Z}{}^\theta_{\fg_Q}$ with an action of the Cartan torus $T_Q=\on{Spec}\BC[\BZ^{Q_0}]$.
We expect the set of
$T_Q$-fixed points in $\ol\fZ{}^\theta_{\fg_Q}$ to be discrete and parametrized by
$\{\gamma\in\BN^{Q_0} : \gamma\leq\theta\}$. Instead of $T_Q$-action we will consider a
$\BC^\times$-action arising from a regular cocharacter $\BC^\times\to T_Q$.
We expect the attractor (resp.\ repellent)
to a point $\gamma$ to be isomorphic to $\fZ^{\theta-\gamma}_{\fg_Q}$ (resp.\ $\fZ^\gamma_{\fg_Q}$).

Finally, we define $\CA^Q_\theta:=H^{2|\theta|}_c(\fZ^\theta_{\fg_Q})$, and
$\CA^Q:=\bigoplus_{\theta\in\BN^{Q_0}}\CA^Q_\theta$. It is equipped with a bialgebra structure:
a comultiplication arising from factorization as in~\S\ref{comultiplication}, and
a multiplication arising from the Drinfeld-Gaitsgory interpolation associated to the above
$\BC^\times$-action as in~\S\ref{multiplication}.
We conjecture that $\CA^Q$ is isomorphic to $U(\fg_Q^+)$.

\appendix

\section{Proofs of Propositions~\ref{proper over Bun} and~\ref{DG completion}}
\label{app}

\centerline{By Dennis Gaitsgory}

\bigskip

\subsection{Proof of Proposition~\ref{proper over Bun}}
\label{proof proper over Bun}
We will need two lemmas.

\begin{lem} \label{lemma about finite mor} Let $Z \ra Z'$ be a morphism of affine schemes,
equipped with actions of algebraic groups $H$ and $H'$, compatible under a surjective homomorphism
$H \ra H'$ with finite kernel that lies in the center of $H$.
Let $_{0}Z' \subset Z'$ be an open $H'$-invariant subscheme, and let
$_{0}Z \subset Z$
be its preimage. Assume that $Z$ is of finite type over a base field. Assume also  that the morphism $_{0}Z \ra ~_{0}Z'$ is finite.
Then the resulting morphism
\begin{equation*}
b\colon
\on{Maps}_{\on{gen}}(X, Z/H \supset~_{0}Z/H) \ra
\on{Maps}_{\on{gen}}(X, Z'/H' \supset~_{0}Z'/H')
\end{equation*}
is finite.
\end{lem}

\noindent{\it Warning:}
Note the morphism $b$ is not necessarily schematic; rather its base change by a scheme yields a Deligne-Mumford stack.
However, the notion of finiteness makes sense in this context as well: finite means proper + finite fibers over geometric points.

\begin{proof} It is easy to see that $b$ is quasi-finite (i.e., every geometric point has a finite preimage). Hence, it suffices to show that it is proper.

The morphism $b$ is the composition of the morphisms
\begin{equation*}
\on{Maps}_{\on{gen}}(X, Z/H \supset~_{0}Z/H) \ra
\on{Maps}_{\on{gen}}(X, Z'/H \supset~_{0}Z'/H) \ra
\on{Maps}_{\on{gen}}(X, Z'/H' \supset~_{0}Z'/H').
\end{equation*}

The first of these morphisms is schematic, and the second is not. We will show that both these morphisms are proper.

First we check that $\on{Maps}_{\on{gen}}(X, Z/H \supset~_{0}Z/H) \ra
\on{Maps}_{\on{gen}}(X, Z'/H \supset~_{0}Z'/H)$ is proper. We will do so by checking the valuative criterion.

Denote by $\upxi$ the generic point of $X$. Let $\CalD$ be the spectrum
of a discrete valuation ring, and let $\overset{\circ}\CalD\subset\CalD$ be the spectrum
of its fraction field. Given a square
\begin{equation*}
\begin{CD}
\overset{\circ}{\mathcal{D}} \times X @>>> Z/H \\ @VVV @VVV \\
\mathcal{D} \times X  @>>>  Z'/H
\end{CD}
\end{equation*}
such that the composition
$\mathcal{D} \times \upxi \ra \mathcal{D} \times X \ra Z'/H$
factors through the open embedding $_{0}Z'/H\hookrightarrow Z'/H$, we have to lift
$\CalD\times X\to Z'/H$ to a morphism
$\mathcal{D} \times X \ra Z/H$.

A morphism $\overset{\circ}{\mathcal{D}} \times X \ra Z/H$
(resp.\ $\mathcal{D} \times X \ra Z'/H$)
is the same as an $H$-bundle $\overset{\circ}{\CF}$ on $\overset{\circ}{\mathcal{D}} \times X$
(resp.\ $\CF$ on $\CalD\times X$) and an $H$-equivariant morphism $\overset{\circ}{\CF} \ra Z$
(resp.\ $\CF\to Z'$). Thus we have a square
\begin{equation*}
\begin{CD}
\overset{\circ}{\CF} @>>> Z \\ @VVV @VVV \\
\CF @>>>  Z'
\end{CD}
\end{equation*}
such that the composition
\begin{equation} \label{11}
\CF|_{\mathcal{D}\times \upxi} \ra \CF \ra Z'
\end{equation}
factors through the open embedding $_{0}Z'\hookrightarrow Z'$ and we have to lift
$\CF\to Z'$ to an $H$-equivariant morphism $\CF \ra Z$.

We now use the assumption that $_{0}Z \ra ~_{0}Z'$ is proper. This implies that
the morphism $_{0}Z/H \ra ~_{0}Z'/H$ is proper.
\begin{comment}
Here we use the fact that the properness is a local property in the smooth topology. You don't need the valuative criterion; it follows
from the definition of properness as ``universally closed". In fact the fpqc topology is enough.
\end{comment}
Since $\mathcal{D}\times \upxi$ is the spectrum of a DVR with generic point $\overset{\circ}{\mathcal{D}}\times \upxi$, in the diagram
$$
\begin{CD}
\overset{\circ}{\mathcal{D}} \times \upxi @>>> Z/H \\
@VVV @VVV \\
\mathcal{D} \times \upxi  @>>>  Z'/H,
\end{CD}
$$
the map $\mathcal{D} \times \upxi \to Z'/H$ lifts to a map $\mathcal{D} \times \upxi \to Z/H$. Hence, the
morphism~(\ref{11}) lifts to an $H$-equivariant morphism
$\CF|_{\mathcal{D} \times \upxi} \ra~_{0}Z$.

Thus  we obtain
an $H$-equivariant morphism from an open subset $\CF|_{U} \subset \CF$
to $Z$, such that
$\on{codim}_\CF(\CF|_{(\mathcal{D}\times X) \setminus U})=2$.
Since $Z$ is affine and $\CF$ is normal, this morphism extends to the whole of
$\CF$.

It remains to show that the morphism
\begin{equation*}
  \on{Maps}_{\on{gen}}(X, Z'/H \supset~_{0}Z'/H) \ra
  \on{Maps}_{\on{gen}}(X, Z'/H' \supset~_{0}Z'/H')
\end{equation*}
is proper. It suffices to check that the morphism $\on{Maps}(X,Z'/H) \ra \on{Maps}(X,Z'/H')$
is proper. Indeed, the following diagram is Cartesian:
\begin{equation*}
\begin{CD}
\on{Maps}(X,Z'/H) @>>> \on{Maps}(X,Z'/H') \\ @VVV @VVV \\
\on{Bun}_H @>>>  \on{Bun}_{H'}
\end{CD}\end{equation*}
so the desired properness of $\on{Maps}(X,Z'/H) \ra \on{Maps}(X,Z'/H')$
follows from the properness of the morphism
$\on{Bun}_H \ra \on{Bun}_{H'}$. (Note, however, that the latter morphism is not schematic; its fibers are isomorphic
to $\on{Bun}_{\on{ker}(H\ra H')}$.)

\end{proof}

For $\Lambda^{\vee+}\ni\lambda^{\lvee}\leq\mu^{\lvee}\in\Lambda^{\vee}$,
we denote by $V^{\lambda^{\lvee},\mu^{\lvee}}$ the irreducible representation of $\on{Vin}_G$
such that $V^{\lambda^{\lvee},\mu^{\lvee}}|_G$ coincides with $V^{\lambda^{\lvee}}$, and
the center $Z_{G_{\on{enh}}}=T$ acts on $V^{\lambda^{\lvee},\mu^{\lvee}}$ via the character $\mu^{\lvee}$.
Let $\la^{\lvee}_j \in \La^{\vee+}, j\in J$ be a collection of dominant weights that
generate $\Lambda^\vee\otimes\BQ$.
We have a natural morphism
$\upomega \colon\on{Vin}_G \ra \prod\limits_{j\in J}\on{End}(V^{\la_j^{\lvee},\la_j^{\lvee}})$.
Note that the preimage
$\upomega^{-1}(\prod\limits_{j\in J}(\on{End}(V^{\la_j^{\lvee},\la_j^{\lvee}}) \setminus \{0\}))$
is exactly $_{0}\!\!\on{Vin}_G$. We denote by $_0\upomega$ the restriction of $\upomega$ to
$_{0}\!\!\on{Vin}_G$.

\begin{lem} \label{from Vin to End}
The morphism $_0\upomega\colon _{0}\!\!\on{Vin}_G \ra \prod\limits_{j\in J}(\on{End}(V^{\la_j^{\lvee},\la_j^{\lvee}}) \setminus \{0\})$ is finite.
\end{lem}

\begin{proof}
  It follows from the Tannakian description of $\on{Vin}_G$ of~\S\ref{Tannakian def of Vin}
  that the morphism $_{0}\upomega$ is quasi-finite. Hence, it is enough to show that it is proper.
We have the action $T \curvearrowright \on{Vin}_G$ as in~\S\ref{Morphism v for Vin}.
We also have the action of a torus $T':= \prod\limits_{j\in J}\BC^{\times}$ on
$\prod\limits_{j\in J}\on{End}(V^{\la_j^{\lvee},\la_j^{\lvee}})$ via dilations.
The morphism $T\twoheadrightarrow T'$ given by
$t \mapsto \prod\limits_{j\in J}\la^{\lvee}_j(t)$ defines
an action of $T$ on $\prod\limits_{j\in J}\on{End}(V^{\la_j^{\lvee},\la_j^{\lvee}})$.
The morphism $\upomega$ is $T$-equivariant.
The following diagram is cartesian:
\begin{equation*}
\begin{CD}
_{0}\!\!\on{Vin}_G @>_{0}\upomega>> \prod\limits_{j\in J}(\on{End}(V^{\la_j^{\lvee},\la_j^{\lvee}}) \setminus \{0\}) \\ @VVV @VVV \\
_{0}\!\!\on{Vin}_G/T @>_{0}\upomega/T>> (\prod\limits_{j\in J}(\on{End}(V^{\la_j^{\lvee},\la_j^{\lvee}}) \setminus \{0\}))/T.
\end{CD}
\end{equation*}
Note that $_{0}\!\!\on{Vin}_G/T$ is proper, while
$(\prod\limits_{j\in J}(\on{End}(V^{\la_j^{\lvee},\la_j^{\lvee}}) \setminus \{0\}))/T$ is separated
(as the stack quotient of a smooth scheme
$\prod\limits_{j\in J}\BP\on{End}(V^{\la_j^{\lvee},\la_j^{\lvee}})$ by a finite group
$\on{Ker}(T\to T')$). Hence the morphism $_{0}\upomega/T$ is proper,
and thus the morphism $_{0}\upomega$ is proper.
\begin{comment}
Let $Ð“$ be a finite group acting a separated scheme $X$. Suppose we have to $Ð“$-torsors over $\mathcal{D}$ that map to $X$ and
these morphisms become identified after the restriction to
$\overset{\circ}{\mathcal{D}}$. Our goal is to show that they coincide on the whole $\mathcal{D}$. Note that any $Ð“$-torsor on $\mathcal{D}$
is trivial. Now the desired follows from the fact that the scheme $X$
is separated.
\end{comment}
\end{proof}

Now we are in a position to finish the proof of~Proposition~\ref{proper over Bun}.
We apply Lemma~\ref{lemma about finite mor} to
$Z:=\on{Vin}_G$,
$Z':=\prod\limits_{j\in J}\on{End}(V^{\la^{\lvee}_j,\la^{\lvee}_j})$,
$H:=G\times G \times T$, $H':=G\times G \times T'$,
$_{0}Z':=\prod\limits_{j\in J}(\on{End}(V^{\la_j^{\lvee},\la_j^{\lvee}}) \setminus \{0\})$.
Note that the preimage $\upomega^{-1}(_{0}Z')=~_{0}Z$ coincides with $~_{0}\!\!\on{Vin}_G$.
It follows from Lemma~\ref{from Vin to End} that the morphism
\begin{equation*}_{0}\upomega\colon_{0}\!\!\on{Vin}_G=~_{0}Z\ra ~_{0}Z'=\prod\limits_{1\leq i\leq \on{r}}(\on{End}(V^{\la^{\lvee}_j,\la^{\lvee}_j})\setminus \{0\})
\end{equation*}
is finite.
So to show that the stack $\ol{\on{Bun}}_G$
is proper over $\on{Bun}_G \times \on{Bun}_G$
it is enough to prove the properness of the stack
\begin{equation*}
\on{Maps}_{\on{gen}}(X,\prod\limits_{j\in J}\on{End}(V^{\la^{\lvee}_j,\la^{\lvee}_j})/(G\times G \times T')
\supset \prod\limits_{j\in J}(\on{End}(V^{\la^{\lvee}_j,\la^{\lvee}_j})\setminus \{0\})/(G\times G \times T'))
\end{equation*}
over $\on{Bun}_G \times \on{Bun}_G$.
Since $\prod\limits_{j\in J}(\on{End}(V^{\la^{\lvee}_j,\la^{\lvee}_j})\setminus \{0\})/T'\simeq
\prod\limits_{j\in J}\BP\on{End}(V^{\la^{\lvee}_j,\la^{\lvee}_j})$,
It follows that the stack in question is isomorphic to the stack of quasi-maps from $X$ to
$\prod\limits_{j\in J}\BP\on{End}(V^{\la^{\lvee}_j,\la^{\lvee}_j})/(G\times G)$
which is known to be proper over $\on{Bun}_G \times \on{Bun}_G$.
Proposition~\ref{proper over Bun} is proved.

\subsection{Proof of Proposition~\ref{DG completion}}
\label{proof DG completion}
The family $_{0}\!\!\on{VinGr}^{\on{princ}}_{G,x}$ can be identified
with the following fibre product:
\begin{equation*}
\on{Maps}_{\BA^1}(X\times \BA^1,~_{0}\!\!\on{Vin}^{\on{princ}}_{G}/(G\times G))
\underset{\on{Maps}_{\BA^1}((X\setminus \{x\})\times \BA^1,~_{0}\!\!\on{Vin}^{\on{princ}}_G/(G\times G))}\times \BA^1.
\end{equation*}
Let us consider the following $\BC^{\times}$-action on the group $G$:
\begin{equation*}
c \mapsto (g \mapsto 2\rho(c)\cdot g \cdot 2\rho(c^{-1}),~c\in\BC^{\times},~g\in G.
\end{equation*}
Let us denote by $\widetilde{G} \ra \BA^1$ the corresponding Drinfeld-Gaitsgory interpolation. Note that $\widetilde{G}$ is a group scheme over $\BA^1$. It follows from~\cite[\S2.4,~D.6]{dg1}
that the stack $_{0}\!\!\on{Vin}^{\on{princ}}_G/(G\times G)$ over $\BA^1$ is isomorphic to the stack $\BA^1/\widetilde{G}$ over $\BA^1$.
Thus the family $_{0}\!\!\on{VinGr}^{\on{princ}}_{G,x}$ is identified with
\begin{equation*}
\on{Maps}_{\BA^1}(X\times \BA^1,\BA^1/\widetilde{G})
\underset{\on{Maps}_{\BA^1}((X\setminus \{x\})\times \BA^1,\BA^1/\widetilde{G})}\times \BA^1.
\end{equation*}
Let us denote this family by $\Gr_{\widetilde{G}}$.
Note that $\Gr_{\widetilde{G}}$ is the affine Grassmannian for the group-scheme
$\widetilde{G}$ over $\BA^1$.

Let us construct the sought-for morphism of ind-schemes over $\BA^1$:
\begin{equation} \label{12}
\upeta:\Gr_{\widetilde{G}}\to \widetilde\Gr_G.
\end{equation}

We first construct a morphism of (pre)stacks over $\BA^1$:
\begin{equation} \label{13}
\upalpha:\BA^1/\widetilde{G}\to \on{Maps}_{\BA^1}(\BX,\BA^1/G)^{\BC^{\times}}.
\end{equation}

We start with the evaluation morphism of group-schemes over $\BA^1$
$$\BX\underset{\BA^1}\times \widetilde{G} \ra \BA^1 \times G,$$
which gives rise to a map of stacks over $\BA^1$:
\begin{equation*}
\Game\colon \BX \underset{\BA^1}\times (\BA^1/\widetilde{G}) \ra \BA^1 \times (\on{pt}\!/G).
\end{equation*}
Note that the above morphism $\Game$ is $\BC^{\times}$-equivariant with respect to the $\BC^{\times}$-action on $\BX \underset{\BA^1}\times (\BA^1/\widetilde{G})$
via the action of $\BC^\times$ on $\BX$ and trivial action on $\BA^1 \times (\on{pt}\!/G)$. Hence, it defines a point of
\begin{equation*}
\on{Maps}_{\BA^1}(\BX\underset{\BA^1}\times(\BA^1/\widetilde{G}),\BA^1/G)^{\BC^{\times}}
\simeq
\on{Maps}_{\BA^1}(\BA^1/\widetilde{G},\on{Maps}_{\BA^1}(\BX,\BA^1/G)^{\BC^{\times}}),
\end{equation*}
i.e., we obtain a map of (pre)stacks over $\BA^1$:
$$\BA^1/\widetilde{G}\to \on{Maps}_{\BA^1}(\BX,\BA^1/G)^{\BC^{\times}},$$
which is the desired $\upalpha$.

From $\upalpha$ we obtain the morphism
\begin{multline*}
\Gr_{\widetilde{G}} = \on{Maps}_{\BA^1}(X\times \BA^1,\BA^1/\widetilde{G})
\underset{\on{Maps}_{\BA^1}((X\setminus \{x\})\times \BA^1,\BA^1/\widetilde{G})}\times\BA^1 \overset{\upalpha}\to\\
\ra \on{Maps}_{\BA^1}(X\times \BA^1,\on{Maps}_{\BA^1}(\BX,\BA^1/G)^{\BC^{\times}})
\underset{\on{Maps}_{\BA^1}((X\setminus \{x\})\times
  \BA^1,\on{Maps}_{\BA^1}(\BX,\BA^1/G)^{\BC^{\times}})}\times \BA^1\simeq \\
\simeq\on{Maps}_{\BA^1}(\BX,\on{Maps}_{\BA^1}(X\times \BA^1,\BA^1/G))^{\BC^{\times}}
\underset{\on{Maps}_{\BA^1}(\BX,\on{Maps}_{\BA^1}((X\setminus \{x\})\times
  \BA^1,\BA^1/G))^{\BC^{\times}}}\times \BA^1
= \widetilde{\Gr}_{G}.
\end{multline*}
This is the desired morphism $\upeta$ in~(\ref{12}).
The diagram~(\ref{don't copy diagrams!}) is commutative by construction.

Note that the composition
\begin{equation*}
\Gr_{\widetilde{G}} \xrightarrow{\upeta} \widetilde{\Gr}_G
\xrightarrow{\gamma} \Gr_G \times \Gr_G \times \BA^1
\end{equation*}
coincides with the restriction of the morphism
$\vartheta\colon \on{VinGr}^{\on{princ}}_{G,x} \ra \Gr_G \times \Gr_G \times \BA^1$
to
the open ind-subscheme $_{0}\!\!\on{VinGr}^{\on{princ}}_{G,x} \simeq \Gr_{\widetilde{G}}$.
Note that the morphism
$\gamma \circ \upeta\colon \Gr_{\widetilde{G}} \ra \Gr_G \times \Gr_G \times \BA^1 = \Gr_{G\times G \times \BA^1}$
also coincides with the morphism induced by the closed embedding $\widetilde{G} \hookrightarrow G \times G \times \BA^1$ of group schemes over $\BA^1$.

It follows from Lemma~\ref{embedding into product}
that the morphism $\gamma \circ \upeta$ is a locally closed embedding.
The morphism $\gamma$ is a locally closed embedding by~\cite[\S2.5.11]{dg2}.
So the morphism $\upeta$ is a locally closed embedding.
It follows from Remark~\ref{sp fiber def free for vingr}, \S\ref{degen of the affine Gr} and \S\ref{zero fiber of DG}, \S\ref{main prop}
that the morphism $\upeta$ is bijective on the level of $\BC$-points. Thus $\upeta$ is
an isomorphism between the corresponding
reduced ind-schemes.

To show that $\upeta$
is an isomorphism of ind-schemes let us
note that the ind-scheme $\Gr_{\widetilde{G}}$
is formally smooth over $\BA^1$ and the morphism $\upeta$
induces the isomorphism between scheme fibers
of the families $\Gr_{\widetilde{G}}$ and $\widetilde{\Gr}_G$
over $\BA^1$.
\begin{comment}
I don't see a priori why $\widetilde\Gr_G$ is formally smooth.
\end{comment}

Proposition~\ref{DG completion} is proved.


\begin{thebibliography}{XXX}

\bibitem[BD]{bd} A.~Beilinson and V.~Drinfeld, {\em Quantization of Hitchin's
  Hamiltonians and Hecke eigensheaves}, Preprint, available at
  http://math.uchicago.edu/$\sim$mitya/langlands/hitchin/BD-hitchin.

\bibitem[BaG]{bag} P.~Baumann, S.~Gaussent, {\em On Mirkovi\'c-Vilonen cycles and crystal
  combinatorics}, Represent.\ Theory {\bf 12} (2008), 83--130.
  
\bibitem[BR]{bari} P.~Baumann, S.~Riche, {\em Notes on the geometric
  Satake equivalence}, Lecture Notes in Math.\ {\bf 2221}, Springer, Cham (2018), 1--134.

\bibitem[BKK]{bkk} P.~Baumann, J.~Kamnitzer, A.~Knutson, {\em The Mirkovi\'c-Vilonen
  basis and Duistermaat-Heckman measures (with appendix by  A.~Dranowski, J.~Kamnitzer,
  and C.~Morton-Ferguson)}, arXiv:1905.08460.
  
\bibitem[Bra]{br} T.~Braden, {\em Hyperbolic localization of intersection
  cohomology}, Transform.\ Groups {\bf 8} (2003), 209--216.

\bibitem[BrG]{bg} A.~Braverman, D.~Gaitsgory, {\em Geometric Eisenstein series},
  Invent.\ Math.\ {\bf 150} (2002), no.~2, 287--384.

\bibitem[Bri]{bri} M.~Brion, {\em Group completions via Hilbert schemes},
  J.~Algebraic Geom.\ {\bf 12} (2003), no.~4, 605--626.

\bibitem[BFGM]{bfgm}
A.~Braverman, M.~Finkelberg, D.~Gaitsgory and I.~Mirkovi\'c,
{\em Intersection cohomology of Drinfeld's compactifications},
Selecta Math.\ {\bf 8} (2002), no.~3, 381--418. {\em Erratum}:
Selecta Math.\ {\bf 10} (2004), no.~3, 429--430.

\bibitem[BFG]{bfg} A.~Braverman, M.~Finkelberg, D.~Gaitsgory,
  {\em Uhlenbeck spaces via affine Lie algebras}, Progr.\ Math.\
  {\bf 244} (2006), 17--135; see arXiv:0301176 for erratum.

\bibitem[BFN]{bfn} A.~Braverman, M.~Finkelberg, H.~Nakajima,
  {\em Coulomb branches of $3d\ \CN=4$ quiver gauge theories and slices in the affine
Grassmannian (with appendices by Alexander Braverman, Michael Finkelberg,
Joel Kamnitzer, Ryosuke Kodera, Hiraku Nakajima, Ben Webster, and Alex Weekes)},
Adv.\ Theor.\ Math.\ Phys.\ {\bf 23} (2019), no.~1, 75--166.

%  Brion, Michel The total coordinate ring of a wonderful variety.
  %  J. Algebra 313 (2007), no. 1, 61--99.

%\bibitem[D]{del} P.~Deligne, {\em Cat\'egories tannakiennes}, Progr.\ Math.\
%  {\bf 87} (1990), 111--195.

\bibitem[DG1]{dg1} V.~Drinfeld, D.~Gaitsgory, {\em Geometric constant
term functor(s)}, Selecta Math.\ (N.S.) {\bf 22} (2016), no.~4, 1881--1951.

\bibitem[DG2]{dg2} V.~Drinfeld, D.~Gaitsgory, {\em On a theorem of Braden},
Transform.\ Groups {\bf 19} (2014), 313--358.

 \bibitem[FFKM]{ffkm}
B. ~Feigin, M. ~Finkelberg, A. Kuznetsov and I. ~Mirkovi\'c,
\emph{Semi-infinite flags. $II$.Local and Global Intersection Cohomology of Quasimaps' Spaces},
Amer.\ Math.\ Soc.\ Transl.\ Ser.\ 2 {\bf 194} (1999), 113--148.

\bibitem[FG]{fg} M.~Finkelberg, V.~Ginzburg, {\em Calogero-Moser space and
  Kostka polynomials}, Adv.\ Math.\ {\bf 172} (2002), no.~1, 137--150.

\bibitem[G]{gai} D.~Gaitsgory, {\em Twisted Whittaker model and factorizable sheaves},
  Selecta Math.\ (N.S.) {\bf 13} (2008), no.~4, 617--659.

\bibitem[GN]{gn} D.~Gaitsgory, D.~Nadler, {\em Spherical varieties and Langlands duality},
  Mosc.\ Math.\ J.\ {\bf 10} (2010), no.~1, 65--137.

\bibitem[L1]{lu} G.~Lusztig, {\em An algebraic-geometric parametrization of the canonical
  basis}, Adv.\ Math.\ {\bf 120} (1996), 173--190.
  
\bibitem[L2]{lus} G.~Lusztig, {\em Semicanonical bases arising from enveloping algebras},
  Adv.\ Math.\ {\bf 151} (2000), no.~2, 129--139.

\bibitem[M]{mil} J.~S.~Milne, {\em Algebraic groups}, Cambridge Studies in Advanced Mathematics
  {\bf 170}, Cambridge University Press, Cambridge (2017), xvi+644pp.

\bibitem[MV]{mv} I.~Mirkovi\'c, K.~Vilonen, {\em Geometric Langlands duality
and representations of algebraic groups over commutative rings}, Ann.\ of Math.\
  (2) {\bf 166} (2007), no.~1, 95--143; {\em Erratum}: Ann.\ of Math. (2) {\bf 188} (2018), no.~3,
  1017--1018.

%\bibitem[M]{mir} I.~Mirkovi\'c, {\em Notes on Drinfeld's theory of classifying pairs,
%  Loop Grassmannian construction of $\dot{U}(\check\fn)$},
%  unpublished manuscripts, see http://people.math.umass.edu/\~{}mirkovic/A.Notes/xx.LoopGrassmannians/\\
%  ALGEBRAS/LoopGrassmannianConstruction.of.NegativeEnvelopingAlgebra.pdf

\bibitem[S1]{s0} S.~Schieder, {\em The Harder-Narasimhan stratification of the moduli
  stack of $G$-bundles via Drinfeld's compactifications}, Selecta Math.\ (N.S.)
  {\bf 21} (2015), no.~3, 763--831.

\bibitem[S2]{s1} S.~Schieder, {\em Picard-Lefschetz oscillators for the
  Drinfeld-Lafforgue-Vinberg degeneration for $SL_2$}, Duke Math.\ J.\
  {\bf 167} (2018), no.~5, 835--921.

\bibitem[S3]{s2} S.~Schieder, {\em Geometric Bernstein asymptotics and
  the Drinfeld-Lafforgue-Vinberg degeneration for arbitrary reductive
  groups}, arXiv:1607.00586.

\bibitem[S4]{s3} S.~Schieder, {\em Monodromy and Vinberg fusion for the
  principal degeneration of the space of $G$-bundles}, arXiv:1701.01898.

\bibitem[St]{st} R.~Steinberg, {\em Lectures on Chevalley groups},
  University Lecture Series {\bf 66} AMS, Providence, RI (2016),
  xi+160pp.

\bibitem[V1]{v} E.~B.~Vinberg, {\em On reductive algebraic semigroups},
Amer.\ Math.\ Soc.\ Transl.\ Ser.\ 2 {\bf 169} (1995), 145--182.

\bibitem[V2]{v2} E.~B.~Vinberg, {\em The asymptotic semigroup of a semisimple
  Lie group}, De Gruyter Exp.\ Math.\ {\bf 20}, de Gruyter, Berlin (1995), 293--310.

\bibitem[W]{w} G.~Wilson, {\em Collisions of Calogero-Moser particles and an
  adelic Grassmannian}, Invent.\ Math. {\bf 133} (1998), 1--41.

\bibitem[Z]{zh} X.~Zhu, {\em An introduction to affine Grassmannians and the geometric
  Satake equivalence}, IAS/Park City Math.\ Ser., {\bf 24}, Amer.\ Math.\ Soc.,
  Providence, RI (2017), 59--154.

\end{thebibliography}
\end{document}